\documentclass[12pt]{amsart}
\usepackage{tikz}
\usepackage{amsmath}
\usepackage{amsxtra}
\usepackage{amscd}
\usepackage{amsthm}
\usepackage{amsfonts}
\usepackage{amssymb}
\usepackage {pstricks}
\usepackage{pstricks,pst-node}
\usepackage[all]{xy}

\usepackage[plainpages,backref,urlcolor=blue]{hyperref}

\usepackage{dashundergaps}

\usepackage{tikz-cd}
\usepackage{enumerate}

\usepackage{graphicx}

\newtheorem{theorem}{Theorem}[section]

\newtheorem{lemma}[theorem]{Lemma}

\newtheorem{corollary}[theorem]{Corollary}

\newtheorem{thm}{Theorem}[section]
\newtheorem*{theorem*}{Theorem}
\newtheorem{lem}{Lemma}[section]
\newtheorem*{lem*}{Lemma}

\newtheorem*{proposition*}{Proposition}
\newtheorem{cor}{Corollary}[section]

\theoremstyle{definition}
\newtheorem{definition}[theorem]{Definition}
\newtheorem{example}[theorem]{Example}

\theoremstyle{remark}
\newtheorem{remark}[theorem]{Remark}

\numberwithin{equation}{section}

\newcommand{\A}{{\mathbb A}}
\newcommand{\B}{{\mathbb B}}
\newcommand{\C}{{\mathbb C}}
\newcommand{\bD}{{\mathbb D}}

\newcommand{\Z}{{\mathbb Z}}
\newcommand{\N}{{\mathbb N}}
\newcommand{\X}{{\mathbb X}}

\newcommand{\I}{{\mathbb I}}

\newcommand{\Q}{{\mathbb Q}}

\newcommand{\BH}{{\mathbb H}}
\newcommand{\BD}{{\mathbb D}}

\newcommand{\BB}{{\mathbb B}}
\newcommand{\BG}{{\mathbb G}}

\newcommand{\CB}{{\mathcal B}}

\newcommand{\CF}{{\mathcal F}}

\newcommand{\CO}{{\mathcal O}}

\newcommand{\CT}{{\mathcal T}}

\newcommand{\CR}{{\mathcal R}}

\newcommand{\id}{{\rm{id}}}

\newcommand{\mf}{\mathfrak}
\newcommand{\fg}{{\mf g}}

\newcommand{\be}{\begin{equation}}
\newcommand{\ee}{\end{equation}}

\newcommand{\inv}{^{-1}}

\newcommand{\lr}{\longrightarrow}
\newcommand{\wt}{\widetilde}
\newcommand{\wh}{\widehat}

\newcommand{\U}{{\rm{U}}}
\newcommand{\End}{{\rm{End}}}
\newcommand{\Hom}{{\rm{Hom}}}

\newcommand{\Sym}{{\rm{Sym}}}

\advance\headheight by 2pt

\newcommand{\fsl}{{\mathfrak {sl}}}
\newcommand{\fsp}{{\mathfrak {sp}}}
\newcommand{\fso}{{\mathfrak {so}}}

\newcommand{\gl}{{\mathfrak {gl}}}
\newcommand{\osp}{{\mathfrak {osp}}}

\newcommand{\ot}{\otimes}

\newcommand{\sdim}{{\rm sdim\,}}

\newcommand{\sdet}{{\rm{sdet}}}

\newcommand{\OSp}{{\rm OSp}}

\newcommand{\ATL}{{\mathcal{ATL}}}
\newcommand{\TL}{\mathcal{TL}}

\newcommand{\TLBC}{\mathcal{TLB}}

\newcommand{\lra}{\longrightarrow}

\newcommand{\beq}{\begin{eqnarray}}
\newcommand{\eeq}{\end{eqnarray}}

\newcommand{\AB}{{\mathcal A}{\mathcal B}}

\newcommand{\ad}{{\rm{ad}}}

\newcommand{\CW}{{\mathcal W}}

\newcommand{\baln}{\begin{aligned}}
\newcommand{\ealn}{\end{aligned}}

\newcommand{\OB}{{\mathcal{OB}}}
\newcommand{\OAB}{{\mathcal{OAB}}}
\usepackage{comment}

\makeatletter
\@namedef{subjclassname@2020}{\textup{2020} Mathematics Subject Classification}
\makeatother

\begin{document}


\title[Polar Brauer category]{A polar Brauer category and  Lie superalgebra representations}
\author{G.I. Lehrer and R.B. Zhang} 
\address{School of Mathematics and Statistics,
University of Sydney, N.S.W. 2006, Australia}
\email{gustav.lehrer@sydney.edu.au, ruibin.zhang@sydney.edu.au}


\begin{abstract} 
We define and study categories whose morphisms are  linear combinations of  ``polar Brauer diagrams'', 
which are polar enhancements of Brauer diagrams. 
These are shown to have a rich structure in their own right, and provide a natural context 
where algebras of chord diagrams arise in place of Brauer algebras as the endomorphism 
algebras of objects in the category. 
We define functors from our categories of polar Brauer diagrams to certain full subcategories of modules 
for the orthosymplectic Lie superalgebra $\mathfrak{g}$ by exploiting the Casimir algebra of $\mathfrak{g}$, 
and show in some special cases that our functors are full. 
One instance of the functors provides a useful tool for studying the universal enveloping superalgebra $\U(\fg)$. 
It gives a categorical description of the centre $Z({\rm U}(\mathfrak{g}))$, and yields ``characteristic identities'' in 
${\rm U}(\mathfrak{g})\otimes {\rm End}_{\mathbb C}(V)$ where $V$ is the natural $\mathfrak{g}$-module.
\end{abstract}

\maketitle

\tableofcontents


\section{Introduction}\label{sect:intro}

Following the pioneering work \cite{Br37}, it is now well established in e.g., \cite{LZ10, LZ15, LZ17, LZ17-Ecate, LZ21} and other works (see, e.g., \cite{Betal, RSo} and references therein) that diagrams 
may be used to analyse certain tensor products of modules for a Lie algebra,  Lie superalgebra, 
or the related quantum group. 
A general method resulting from diagrammatic techniques enables one to investigate certain classes of tensor products of modules collectively. One creates diagram categories and seeks functors from them to categories of such modules for a Lie (super)algebra $\fg$. The graphical description of morphisms of a diagram category usually leads to simpler descriptions of morphisms in the target category of $\fg$-modules, 
and hence a better understanding of the latter.

A widely studied example is the Brauer category introduced in \cite{LZ15} and functors from it to the categories of tensor powers $V^{\ot r}$ of the natural module for the orthogonal or symplectic Lie algebra, or the orthosymplectic Lie superalgebra.  Similarly, an oriented version of the Brauer category (see, e.g., \cite{LZ23}) may be applied to the study of repeated tensor products of the natural module and its dual for the general linear Lie (super)algebra. 

A modern categorical formulation of the invariant theory of the classical (super)groups and the corresponding quantum (super)groups has been developed using these diagram categories \cite{LZ15}. 
Much new insight has been gained in understanding the endomorphism algebras $\End_{\U(\fg)}(V^{\ot r})$ \cite{LZ12, LZ15, Zy}, e.g., in developing presentations for them in terms of explicit generators and relations. 

In \cite{ILZ} we introduced the notion of polar tangle diagrams, and defined diagrammatic categories with morphisms spanned by such diagrams. These categories naturally arose from the study of quantum group modules of the type $M\ot V^{\ot r}$ with $M$ being an arbitrary module and $V$ a finite dimensional simple module. We showed in op.~cit.  how to identify the category of representations of 
$\U_q(\fsl_2)$ of the form $M\ot V^{\ot r}$ (where $M$ is a projective Verma module and $V$ is 
the two dimensional Weyl module) with a certain category of such polar tangle diagrams.  

The principal aim of the present paper is to develop a theory of new diagrammatic categories for application to  representation theory in a Lie theoretic context. 

We introduce  a category $\AB(\delta)$ of polar Brauer diagrams, which heuristically, 
``approximates''  the polar BMW category \cite{ILZ}  to first order in $\hbar$ with $q \approx 1+\hbar$
  (see Section \ref{sect:comm}). 
This category is particularly suitable for studying modules of  type $M\ot V^{\ot r}$ for the orthosymplectic Lie superalgebra, 
where $M$ is an arbitrary module and $V$ is the natural module.

%
We will see that the algebra of chord diagrams (see Section \ref{sect:chord}) and the
Casimir algebra (see Definition \ref{def:Cr}) naturally arise from the theory of
$\AB(\delta)$.  The algebra of chord diagrams occurs in the literature in several contexts, 
such as the holonomy of connections
on spaces of configurations (the KZ connection)  and representations of
the pure braid group (a good source is \cite[\S XX]{K}). 
It enters the representation theory of Lie (super)algebras through Casimir algebras \cite{LZ06}, and gave rise to a unified treatment  of the invariant theory of classical groups in op. cit.. 

We develop the structure theory of the diagrammatic category $\AB(\delta)$ and certain of its quotient categories, and 
investigate several applications of $\AB(\delta)$ to illustrate its connections with representation theory.   More technical applications will be developed elsewhere. 

We now describe in more precise terms the key points of this paper.  

\subsection{Categories of polar Brauer diagrams} 

Let $K$ be a commutative ring (the ``base ring'') and let $\delta\in K$ be an arbitrary but fixed parameter.
We introduce a pre-additive $K$-linear diagram category $\AB(\delta)$,  which is a ``polar'' enhancement of the Brauer category. The set of objects of $\AB(\delta)$ is $\N=\{0, 1, 2, \dots\}$,  
the $K$-modules of morphisms are spanned by certain diagrams which we define below.  
These ``polar diagrams'' are Brauer diagrams connected to a pole in a specific way. Composition of morphisms is given essentially by composition of diagrams. 
The category $\AB(\delta)$ is called a {\em diagrammatic affine Brauer category}, and each morphism diagram an {\em affine Brauer diagram}. 
We will also loosely refer to $\AB(\delta)$ as a {\em polar Brauer category}.

The category $\AB(\delta)$ contains a subcategory isomorphic to the usual Brauer category $\CB(\delta)$. 
It admits a bi-functor $\ot: \AB(\delta) \times \CB(\delta)\lra  \AB(\delta)$ 
arising from juxtaposition of affine Brauer diagrams (placed on the left) with ordinary Brauer diagrams, which   
makes $\AB(\delta)$ a right module category over $\CB(\delta)$ as a monoidal category.  
The $K$-module of morphisms of $\AB(\delta)$ will be described by generators 
and relations similarly to the description of $\CB(\delta)$ in \cite{LZ15}. 

The  category $\AB(\delta)$ has a close relationship with the algebras $T_r$ of chord diagrams (see Section \ref{sect:chord}) for all integers $r\ge 2$.
These are algebras arising in several different aspects of Lie-theoretic representation theory, 
and play a similar role in our category $\AB(\delta)$
to that of the Brauer algebras in the Brauer category.
While the  Brauer algebras of different degrees appear in the Brauer category 
as endomorphism algebras of objects, for each object $r\in\AB(\delta)$, 
$\Hom_{\AB(\delta)}(r, r)$  is a quotient of the algebra $T_r$, which
is an affine generalisation of the Brauer algebra of degree $r$.  

The algebras $T_r$ are in some sense the semi-classical limits of the group rings $K'(q)B_r$ of the $r$-string braid groups,
where $K'$ is a suitable integral extension of $K$ and $q$ is an indeterminate. Since the braid group $B_r$ acts on 
tensor products of quantum group modules via $R$-matrices, our theory may be regarded as related to this action. 

The category $\AB(\delta)$ provides the natural framework for applications of the Casimir algebra $C_r(\fg)$ and
algebra $T_r$ of chord diagrams in invariant theory \cite{LZ06}.
The diagrammatics intrinsic to $\AB(\delta)$ leads 
to a graphical description of these quotient algebras of $T_r$, which is interesting in its own right. 

The structure of $\AB(\delta)$ is much richer than that of  the Brauer category $\CB(\delta)$. Even the endomorphism algebras $\Hom_{\AB(\delta)}(0, 0)$ and $\Hom_{\AB(\delta)}(1, 1)$ are quite interesting,
as we shall see below. 
They shall be seen to play significant roles in a categorical construction of the centre of $\U(\fg)$, and in the study of characteristic identities for the orthogonal and 
symplectic Lie algebras discovered in Bracken's thesis \cite[\S 4.10]{B} and the papers \cite{BG, G} (also see \cite{Go}). 

A particular quotient category $\ATL(\delta)$ of the category $\AB(\delta)$,  which we refer to as a {\em diagrammatic affine Temperley-Lieb category} or {\em polar Temperley-Lieb category},  
contains a subcategory isomorphic to the usual Temperley-Lieb category \cite{ILZ}. We are able to determine its structure fairly completely; 
in particular, given $r$ and $s$, $\Hom_{\ATL(\delta)}(r, s)$  is a free module over a polynomial algebra 
over $K$, and we determine its rank.   

 The category $\ATL(\delta)$ has a quotient category $\TLBC(\delta, \mu)$ depending on another parameter $\mu$, analogous to the Temperley-Lieb category  of type $B$ \cite{GL03,ILZ}.

\subsection{Applications of polar diagram categories} 

The diagram categories introduced  in this paper have deep connections with other areas.  We investigate some of their most direct applications.

\subsubsection{Orthosymplectic Lie superalgebras} 
For any pair of fixed  non-negative integers $m$ and $n$, let $V=\C^{m|2n}$, and let $\omega$ be a non-degenerate supersymmetric even bilinear form on $V$.
Let $\omega$ be a non-degenerate, supersymmetric, even bilinear form on $V$.
 Denote by $\fg=\osp(V; \omega)$ the corresponding orthosymplectic  Lie superalgebra. Let $\CT(V)$ be the full subcategory category of $\U(\fg)$-modules with objects $V^{\ot r}$ for all $r\in \N$. There is a tensor functor $\CF: \CB(\sdim)\lra \CT(V)$, where $\CB(\delta)$ is the relevant Brauer category, which is well studied in the literature (see, e.g., \cite{LZ17}). Here $\sdim:=\sdim(V)=m-2n$ is the super dimension of $V$.

Given an arbitrary $\U(\fg)$-module $M$, we consider the full subcategory $\CT_M(V)$ of $\U(\fg)$-modules with objects $M\ot V^{\ot r}$ for all $r\in \N$.  There is a tensor product functor $\ot: \CT_M(V)\times \CT(V)\lra \CT_M(V)$ defined in the obvious way. 
We construct a functor $\CF_M:  \AB(\sdim)\lra \CT_M(V)$,  which preserves the tensor product $\ot: \AB(\delta) \times \CB(\delta)\lra  \AB(\delta)$ in the sense that the commutative diagram \eqref{eq:tensor-funct}.
Clearly $\CT_M(V)=\CT(V)$ when $M$ is chosen to be $V$, and $\CF_M$ essentially coincides with $\CF$ in this case. 

The construction of the functor $\CF_M:  \AB(\sdim)\lra \CT_M(V)$ generalises the natural action of the 
Casimir algebra $C_{r+1}(\fg)$ of $\fg$ on $M\ot V^{\ot r}$ for any $r$,  by exploiting the connection between $T_{r+1}$ and $C_{r+1}(\fg)$. The situation may be succinctly summarised by the following commutative diagram for all $r$, 
\[
\begin{tikzcd}
T_{r+1} \arrow[d] \arrow[r] \arrow[dr] & \Hom_{\AB(\sdim)}(r, r)\arrow[d, "\CF_M|_r"] \\
C_{r+1}(\fg) \arrow[r] & \End_{\CT_M(V)}(M\ot V^{\ot r}),
\end{tikzcd}
\]
where the horizontal map at the bottom is the natural action of $C_{r+1}(\fg)$ on $M\ot V^{\ot r}$, and $\CF_M|_r$ is the restriction of $\CF_M$ to $\Hom_{\AB(\sdim)}(r, r)$.  

In a future publication, we investigate the functor $\CF_M$  in detail for generalised Verma modules $M$ in some parabolic category $\CO$ of $\fg$.  For a class of projective generalised Verma modules, the corresponding factors are full.   

\subsubsection{Category $\CO$ of $\fsp_2(\C)$} 
When $V=\C^{0|2}$ thus $\sdim=-2$, the corresponding Lie superalgebra $\osp(V;\omega)$ is the ordinary Lie algebra $\fsp_2(\C)$. When $M$ is
 the Verma module $M_\lambda$ with highest weight $\lambda$ or the simple module $L_\lambda$ for $\fsp_2(\C)$, we shall show that
 the functor $\CF_{M}:  \AB(-2)\lra \CT_{M}(V)$ factors through the quotient (Temperley-Lieb) category $\TLBC(-2, \lambda)$. 
 Furthermore, if $M_\lambda$ is projective in category $\mathcal O$ of $\fsp_2(\C)$, we show that $\TLBC(-2, \lambda)$ is isomorphic to $\CT_{M_\lambda}(V)$. 
 This result is a classical analogue of the equivalence between a certain Temperley-Lieb algebra of type B and a category of infinite dimensional $\U_q(\fsl_2)$-modules 
 proved in \cite{ILZ}; it provides an explicit description of homomorphisms in $\CT_{M_\lambda}(V)$.   

\subsubsection{Universal enveloping superalgebras} 

For any orthosymplectic Lie superalgebra  $\fg=\osp(V; \omega)$, 
if the $\fg$-module $M$ is the universal enveloping superalgebra $\U=\U(\fg)$ itself,  the functor $\CF_\U$ from $\AB(\sdim)$ to $\CT_\U(V)$ provides an effective tool for studying the universal enveloping superalgebra.

We show that the centre of $\U(\fg)$ can be analysed by investigating the image of $\Hom_{\AB(\sdim)}(0, 0)$ under the functor $\CF_U$; 
in particular, explicit generators of the centre are obtained. 

In general the restriction $\CF_\U|_1$ of $\CF_\U$ to $\Hom_{\AB(\sdim)}(1, 1)$ has a non-trivial kernel, 
and non-zero elements of $\ker(\CF_\U|_1)$ lead to interesting identities in $\U(\fg)\ot\End_\C(V)$. 
In the special cases with $n=0$ or $m=0$, which respectively correspond to $\fg$ being $\mathfrak{so}_m$ or $\mathfrak{sp}_{2n}$, we obtain a categorical derivation 
of the characteristic identities for these Lie algebras discovered in \cite[\S 4.10]{B} and \cite{BG, G}.  

These characteristic identities provide a useful tool for, and have been widely applied in, explicit Lie theoretical computations in the mathematical physics literature (see, e.g., \cite{IWD} for a brief review and references).
Our work may be considered as a categorification of these identities.

\subsection{Related work}
The paper \cite{ILZ} introduced polar braid diagrams and constructed categories with such diagrams as morphisms. A  polar braid diagram is a braid diagram with a pole, a vertical 
open thick string on the left of the diagram, around which thin strings may wrap around. One such category is the polar Temperley-Lieb category,  which was shown to be isomorphic to a 
category of infinite dimensional $\U_q(\fsl_2)$-representations, 
and also to the affine Temperley-Lieb category introduced in \cite{GL98}. 

%

Although it does not involve the ``polar'' theme which is central to this work, we should mention that in \cite{LZ06}, it was proved that if $V_1,\dots,V_r$ is any set
of finite dimensional simple $\fsl_2$-modules, we have a surjections $T_r\lr\End_{\fsl_2}(V_1\ot\dots\ot V_r)$. This also has a categorical context.

There exist an affine Brauer category 
in \cite[Definition 1.2]{RSo} and a related affine VW supercategory \cite[\S 1.4]{Betal},  
whose endomorphism algebras are the affine Nazarov-Wenzl algebras \cite[Definition 2.1]{ES} 
and some of their variants \cite[Definition 39]{Betal} respectively.
The categories of \cite{Betal, RSo} are presented in terms of dotted Brauer diagrams 
(analogous to the dotted Temperley-Lieb diagrams in \cite{GL98}), which makes use of a different   
type of diagrammatics from that of the polar Brauer diagrams of $\AB(\delta)$. 
However, these categories were designed for the purpose of studying modules of 
the type $M\ot V^{\ot r}$ for the orthogonal and symplectic Lie algebras, and the 
periplectic Lie superalgebra $\mathfrak{p}(n)$ respectively, where $V$ is the natural module, 
thus have a high degree of similarity to our category in their general features.  We will make a comparison of our category with the categories of \cite{Betal, RSo} in precise terms in Section \ref{sect:comm}. 

We note that Schur-Weyl dualities for modules $M\ot V^{\ot r}$ with $M$ being some types of projective parabolic Verma modules were established in \cite{ES, RSo}.

\section{The polar Brauer category}\label{sect:construct}
%
%
%
%
In this section we shall give the basic definitions and properties of the categories we study.
Let $K$  be a commutative ring with identity.

 \subsection{The Brauer category} 
The Brauer category $\CB(\delta)$ over $K$ with fixed parameter $\delta\in K$ was introduced in \cite{LZ15}; it has objects $\N=\{0, 1, 2, \dots\}$, and the $K$-modules of morphisms 
 have bases consisting of Brauer diagrams. It is a tensor category with the tensor product functor $\ot: \CB(\delta)\times \CB(\delta)\lra \CB(\delta)$ such 
 that $r\ot s=r+s$ for all objects, and $A\ot B$ for any Brauer diagrams $A, B$ is a the juxtaposition of them with $A$ positioned on the left. 

It is proved in {\it op.~cit.} that the set of all Brauer diagrams is generated by the basic Brauer diagrams $I, X, \cap, \cup$ depicted in Figure \ref{fig:generators} 
under the operations of composition and tensor product.  
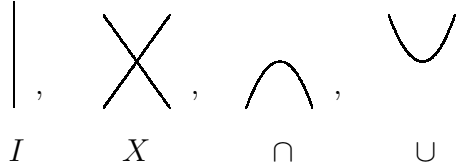
\begin{figure}[h]
\begin{picture}(30, 55)(0,-20)
\put(0, 0){\line(0, 1){40}}
\put(8, 5){, }
\put(-2, -20){$I$}
\end{picture}
\begin{picture}(50, 40)(0,-20)
\qbezier(0, 0)(0, 0)(25, 35)
\qbezier(0, 35)(0, 35)(25, 0)
\put(33, 5){, }
\put(7, -20){$X$}
\end{picture}
\begin{picture}(50, 40)(0,-20)
\qbezier(0, 0)(13, 35)(25, 0)
\put(33, 5){, }
\put(10, -20){$\cap$}
\end{picture}
\begin{picture}(50, 40)(0,-20)
\qbezier(0, 35)(13, 0)(25, 35)
\put(10, -20){$\cup$}
\end{picture}
\caption{Generators of Brauer diagrams}
\label{fig:generators}
\end{figure}

The defining relations among the generators are described in \cite{LZ15}. 

The endomorphism algebra $\Hom_{\CB(\delta)}(r, r)$ for any $r\in\N$ is the Brauer algebra $B_r(\delta)$ of degree $r$ over $K$,  with parameter $\delta$, 
whose identity element is $I_r= \underbrace{\Big| \ \Big| \dots \Big|}_r $.  


Consider the element $H=X-{(\cup\circ\cap)}\in B_2(\delta)$; it is represented diagrammatically in Figure \ref{fig:t-image}
\begin{figure}[h]
\begin{picture}(200, 40)(0,0)
\put(0, 15){$ H =$}

\put(35, 0){\line(0, 1){40}}

\put(35, 20){\line(1, 0){20}}

\put(55, 0){\line(0, 1){40}}

\put(65, 15){$:=$}

\qbezier(90, 0)(90, 0)(115, 35)
\qbezier(90, 35)(90, 35)(115, 0)

\put(125, 15){$-$}

\qbezier(155, 35)(168, 10)(180, 35)
\qbezier(155, 0)(168, 25)(180, 0)
\end{picture} 
\caption{Element $H$ of $\Hom_{\CB(\delta)}(2, 2)$}
\label{fig:t-image}
\end{figure}
and  will play an important role later. 
One sees easily that $H$ satisfies the relation depicted in Figure \ref{fig:Brauer-H-skew},
\begin{figure}[h]
\setlength{\unitlength}{0.25mm}
\begin{picture}(125, 65)(0,0)

\qbezier(40, 35)(50, 55)(60, 35)
\qbezier(40, 25)(50, 5)(60, 25)

\qbezier(40, 35)(30, 5)(20, 0)
\qbezier(40, 25)(30, 55)(20, 60)


\put(60, 25){\line(0, 1){10}}

\put(60, 30){\line(1, 0){20}}
\put(80, 0){\line(0, 1){60}}


\put(100, 27){$=\  $}
\end{picture}  
\begin{picture}(115, 65)(-5,0)
\put(0, 0){\line(0, 1){60}}
\put(0, 30){\line(1, 0){20}}

\qbezier(20, 35)(30, 55)(40, 35)
\qbezier(20, 25)(30, 5)(40, 25)
\put(20, 25){\line(0, 1){10}}


\qbezier(40, 35)(50, 5)(60, 0)
\qbezier(40, 25)(50, 55)(60, 60)

\put(72, 27){$=\ - $}
\end{picture} 
\begin{picture}(40, 65)(-10,0)
\put(0, 0){\line(0, 1){60}}
\put(0, 30){\line(1, 0){30}}

\put(30, 0){\line(0, 1){60}}


\end{picture} 
\caption{Skew symmetry of $H$}
\label{fig:Brauer-H-skew}
\end{figure}
which will be referred to as {\em skew symmetry} of $H$. 

We offer a brief algebraic proof of this relation for the reader's convenience. The left side of Figure \ref{fig:Brauer-H-skew} may be written $(I\ot\cap\ot I)\circ(X\ot H)\circ(I\ot\cup\ot I)$. Using the formula for
$H$, this may be expanded and simplified, using the relations in the Brauer category, to $(\cup\circ \cap)- X=-H$.

Let $\Hom(\CB(\delta))=\sum_{r, s} \Hom_{\CB(\delta)}(r, s)$.  The tensor product restricted to morphisms leads to an associative multiplication $\ot: \Hom(\CB(\delta))\times \Hom(\CB(\delta))\lra \Hom(\CB(\delta))$, which we
refer to as {\em horizontal multiplication} of $\Hom(\CB(\delta))$. This associative algebra is denoted by $(\Hom(\CB(\delta)), \ot)$.  Its identity is the empty Brauer diagram ($\in\Hom_{\CB(\delta)}(0,0)$).

In this section we wish to give an analogous description of the morphisms in our new polar category $\AB(\delta)$.

\subsection{Polar diagrams}  Let us start by introducing polar diagrams, an intermediate notion needed for building our category $\AB(\delta)$. 

Let $r_0, r_1, \dots, r_k\in\N$, and let $D_i$  be any Brauer $(r_{i-1}, r_i)$-diagram for each $ i=1, \dots, k$. We introduce a diagram associated with this data as depicted in Figure \ref{fig:AffD}. 
\begin{figure}[h]
\begin{picture}(120, 190)(-20,0)
{
\linethickness{1mm}
\put(0, 0){\line(0, 1){190}}
}
\put(25, 170){\line(0, 1){20}}
\put(35, 170){\line(0, 1){20}}
\put(55, 170){\line(0, 1){20}}
\put(40, 180){...}

\put(35, 155){\small$D_k$}
\put(20, 150){\line(0, 1){20}}
\put(60, 150){\line(0, 1){20}}
\put(20, 150){\line(1, 0){40}}
\put(20, 170){\line(1, 0){40}}

\put(25, 130){\line(0, 1){20}}
\put(35, 130){\line(0, 1){20}}
\put(55, 130){\line(0, 1){20}}
\put(40, 140){...}

\put(0, 140){\uwave{\hspace{9mm}}}

\put(40, 115){$\vdots$}

\put(0, 100){\uwave{\hspace{9mm}}}

\put(25, 90){\line(0, 1){20}}
\put(35, 90){\line(0, 1){20}}
\put(55, 90){\line(0, 1){20}}
\put(40, 100){...}
\put(35, 75){\small$D_2$}
\put(20, 70){\line(0, 1){20}}
\put(60, 70){\line(0, 1){20}}
\put(20, 70){\line(1, 0){40}}
\put(20, 90){\line(1, 0){40}}

\put(25, 70){\line(0, -1){30}}
\put(35, 70){\line(0, -1){30}}
\put(55, 70){\line(0, -1){30}}
\put(40, 55){...}

\put(0, 60){\uwave{\hspace{9mm}}}

\put(35, 25){\small$D_1$}
\put(20, 20){\line(0, 1){20}}
\put(60, 20){\line(0, 1){20}}
\put(20, 20){\line(1, 0){40}}
\put(20, 40){\line(1, 0){40}}

\put(25, 20){\line(0, -1){20}}
\put(35, 20){\line(0, -1){20}}
\put(55, 20){\line(0, -1){20}}
\put(40, 10){...}

\put(65, 5){.}

\end{picture} 
\caption{Polar diagram}
\label{fig:AffD}
\end{figure}
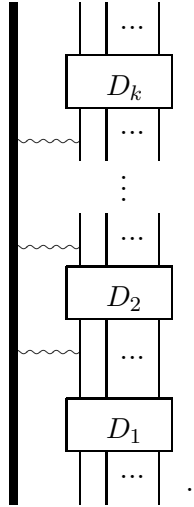

\vspace{2mm}

\noindent
The strings of  usual Brauer diagrams are drawn as thin arcs.  
There is a unique thick arc called the {\em pole}, which is always vertical and is placed on the left boundary of the diagram;  wavy lines are always horizontal; 
the left end of a wavy line is connected to the pole and the right end connected to a vertical thin line as shown; the pole and wavy lines do not cross any arc.  

\begin{remark}\label{rem:hor}
The `wavy lines' which connect the pole to a Brauer diagram will be referred to as ``connectors''.
\end{remark}

\subsubsection{Relations among the polar diagrams}

Denote by ${\bf P}$ the free $K$-module with a basis consisting of the above diagrams, and let ${\bf 0}_\CB$ be the $K$-submodule generated by the elements described below.  
\begin{enumerate}
\item {\bf Going-up and going-down relations}.  For any diagram $\mathbb D$ in ${\bf P}$, if the left (resp. right) hand side of Figure \ref{fig:up-down} 
appears as a sub-diagram, we replace it by the right (resp. left) hand side to obtain a diagram ${\mathbb D}'$. Then $\mathbb D-{\mathbb D}'$ is a generator of ${\bf 0}_\CB$.

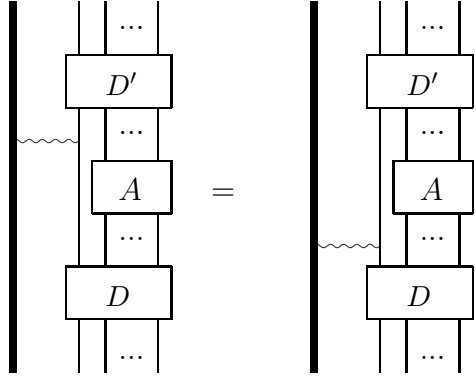
\begin{figure}[h]
\begin{picture}(120, 140)(-20,45)
{
\linethickness{1mm}
\put(0, 50){\line(0, 1){140}}
}
\put(25, 170){\line(0, 1){20}}
\put(35, 170){\line(0, 1){20}}
\put(55, 170){\line(0, 1){20}}
\put(40, 180){...}

\put(35, 155){\small$D'$}
\put(20, 150){\line(0, 1){20}}
\put(60, 150){\line(0, 1){20}}
\put(20, 150){\line(1, 0){40}}
\put(20, 170){\line(1, 0){40}}

\put(25, 130){\line(0, -1){20}}

\put(25, 130){\line(0, 1){20}}
\put(35, 130){\line(0, 1){20}}
\put(55, 130){\line(0, 1){20}}
\put(40, 140){...}

\put(30, 130){\line(1, 0){30}}
\put(30, 110){\line(1, 0){30}}

\put(30, 130){\line(0, -1){20}}
\put(60, 130){\line(0, -1){20}}

\put(0, 140){\uwave{\hspace{9mm}}}

\put(40, 115){$A$}


\put(25, 90){\line(0, 1){20}}
\put(35, 90){\line(0, 1){20}}
\put(55, 90){\line(0, 1){20}}
\put(40, 100){...}
\put(35, 75){\small$D$}
\put(20, 70){\line(0, 1){20}}
\put(60, 70){\line(0, 1){20}}
\put(20, 70){\line(1, 0){40}}
\put(20, 90){\line(1, 0){40}}

\put(25, 70){\line(0, -1){20}}
\put(35, 70){\line(0, -1){20}}
\put(55, 70){\line(0, -1){20}}
\put(40, 55){...}
\put(75, 115){$=$}
\end{picture}
\begin{picture}(120, 140)(-10,45)
{
\linethickness{1mm}
\put(0, 50){\line(0, 1){140}}
}
\put(25, 170){\line(0, 1){20}}
\put(35, 170){\line(0, 1){20}}
\put(55, 170){\line(0, 1){20}}
\put(40, 180){...}

\put(35, 155){\small$D'$}
\put(20, 150){\line(0, 1){20}}
\put(60, 150){\line(0, 1){20}}
\put(20, 150){\line(1, 0){40}}
\put(20, 170){\line(1, 0){40}}

\put(25, 130){\line(0, -1){20}}

\put(25, 130){\line(0, 1){20}}
\put(35, 130){\line(0, 1){20}}
\put(55, 130){\line(0, 1){20}}
\put(40, 140){...}

\put(30, 130){\line(1, 0){30}}
\put(30, 110){\line(1, 0){30}}

\put(30, 130){\line(0, -1){20}}
\put(60, 130){\line(0, -1){20}}


\put(40, 115){$A$}

\put(0, 100){\uwave{\hspace{9mm}}}

\put(25, 90){\line(0, 1){20}}
\put(35, 90){\line(0, 1){20}}
\put(55, 90){\line(0, 1){20}}
\put(40, 100){...}
\put(35, 75){\small$D$}
\put(20, 70){\line(0, 1){20}}
\put(60, 70){\line(0, 1){20}}
\put(20, 70){\line(1, 0){40}}
\put(20, 90){\line(1, 0){40}}

\put(25, 70){\line(0, -1){20}}
\put(35, 70){\line(0, -1){20}}
\put(55, 70){\line(0, -1){20}}
\put(40, 55){...}
\end{picture}
\caption{$A$ can move above or below the connector}
\label{fig:up-down}
\end{figure}

\item {\bf Removal of free loops}. Applications of the going-up and going-down relations to $\mathbb{D}\in{\bf P}$ involve composition of Brauer diagrams, which may introduce free loops,
i.e. loops not connected to any connector. Removing all free loops from ${\mathbb D}$ yields 
a diagram $\wt{\mathbb D}\in{\bf P}$.  If $\ell({\mathbb D})$ is the number of free loops in ${\mathbb D}$, then ${\mathbb D}-\delta^{\ell({\mathbb D})} \wt{\mathbb D}$ is a generator of ${\bf 0}_\CB$.

\item{\bf Other Brauer relations}. More generally, if ${\mathbb D}$ in ${\bf P}$  and $\lambda\bD'\in \bf P$ ($\lambda\in K$) arises from $\bD$ by repeated applications of the going-up and going-down relations
as well as relations in the Brauer category which arise from a composition with no intervening connector, then $\bD-\lambda\bD'\in{\bf 0}_\CB$.
\end{enumerate}
\begin{definition}
Let  ${\bf H}={\bf P}/{\bf 0}_\CB$. Call the image in ${\bf H}$ of any diagram ${\mathbb D}\in {\bf P}$ a {\em polar diagram}. 
It will be represented by the same picture as that for ${\mathbb D}$ (as shown in Figure \ref{fig:AffD}). 
\end{definition}

Note that the relations above include all the relations in $\CB(\delta)$ for Brauer diagrams not connected to the pole.

The polar diagram depicted in Figure \ref{fig:AffD} has 
$r_0+1$ (resp. $r_k+1$) endpoints at the bottom (resp. on the top) with one of them being the lower (resp. upper) endpoint of the pole and the rest being endpoints of thin arcs. 
It will be called a {\em polar $(r_0, r_k)$-diagram}. 
The number of connectors in the polar diagram is $k-1$. This is  the {\em order}  of the diagram.

Clearly ${\bf H}$ is a free $K$-module. 
If ${\bf H}(r, s)$ denotes the $K$-submodule spanned by polar $(r, s)$-diagrams, then ${\bf H}=\sum_{r, s}{\bf H}(r, s)$. Note that ${\bf H}$ is $\Z_+$-graded by the order of the polar diagrams. We denote by ${\bf H}_d$ and ${\bf H}(r, s)_d$ the degree $d$ homogeneous components of ${\bf H}$ and ${\bf H}(r, s)$ respectively, and denote  by   
$
\pi_d: {\bf H}\lra  {\bf H}_d
$
the projection map onto homogeneous component of degree $d$.

\begin{lem}\label{lem:embed}
For any given $(r, s)$, there is a $K$-module isomorphism 
\[
\iota: \Hom_{\CB(\delta)}(r, s) \lra {\bf H}(r, s)_0, 
\]
which sends a Brauer diagram $A\in  \Hom_{\CB(\delta)}(r, s)$ to the unique polar $(r, s)$-diagram $\A_0$ of order $0$ depicted in  Figure \ref{fig:ord0}.
\end{lem}

\begin{figure}[h]
\begin{picture}(120, 60)(-40,50)
\put(-35, 80){$\A_0  =$}
{
\linethickness{1mm}
\put(5, 50){\line(0, 1){60}}
}

\put(25, 90){\line(0, 1){20}}
\put(35, 90){\line(0, 1){20}}
\put(55, 90){\line(0, 1){20}}
\put(40, 100){...}
\put(35, 75){\small$A$}
\put(20, 70){\line(0, 1){20}}
\put(60, 70){\line(0, 1){20}}
\put(20, 70){\line(1, 0){40}}
\put(20, 90){\line(1, 0){40}}

\put(25, 70){\line(0, -1){20}}
\put(35, 70){\line(0, -1){20}}
\put(55, 70){\line(0, -1){20}}
\put(40, 55){...}

\put(75, 55){.}
\end{picture} 
\caption{Diagrams of order $0$}
\label{fig:ord0}
\end{figure}
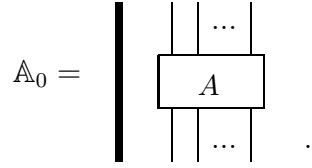 

Here are some simple polar diagrams, which will play crucial roles later:
\begin{itemize}
\item  the $(0, 0)$-diagram $\I_0$ consisting of the pole only, and
the $(r, r)$-diagrams $\I_r=\A_0$ for $A=I_r$ with $r\ge 1$ as defined by Figure \ref{fig:ord0};  

\item the $(1, 1)$-diagrams $\BH $,  $\BH ^T$ and $(\BH^\ell)^T$ (for $\ell\ge 1$) depicted in 
Figure \ref{fig:bbH}; 

\item 
the $(r, r)$-diagrams $\BH _{0 i}(r)$, for 
$i=1, 2, \dots, r$ and $r=1, \dots$,  depicted in Figure \ref{fig:Hir}.  

\item 
the $(r, r)$-diagrams $\BH _{i j}(r)$, for 
$1\le i<j\le r$ and $r=1, \dots$,  depicted in Figure \ref{fig:Hijr}. Note that each $\BH _{i j}(r)$ is a linear combination of two diagrams. 

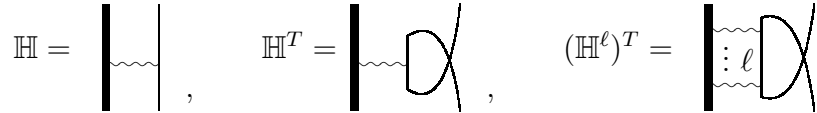
\begin{figure}[h]
\begin{picture}(80, 40)(0,0)
\put(0, 20){$\BH   =$}
{
\linethickness{1mm}
\put(35, 0){\line(0, 1){40}}
}
\put(35, 20){\uwave{\hspace{7mm}}}

\put(55, 0){\line(0, 1){40}}

\put(65, 5){, }

\end{picture} 
\begin{picture}(120, 40)(-10, 0)
\put(0, 20){$\BH ^T  =$}
{
\linethickness{1mm}
\put(35, 0){\line(0, 1){40}}
}
\put(35, 20){\uwave{\hspace{7mm}}}

\put(55, 8){\line(0, 1){20}}

\qbezier(55, 28)(70, 38)(75, 0)
\qbezier(55, 8)(70, -2)(75, 40)

\put(85, 5){,}
\end{picture}
\begin{picture}(100, 40)(-20, 0)
\put(-20, 20){$(\BH^\ell) ^T  =$}
{
\linethickness{1mm}
\put(35, 0){\line(0, 1){40}}
}
\put(35, 33){\uwave{\hspace{7mm}}}
\put(40, 15){$\vdots$ $\ell$}
\put(35, 12){\uwave{\hspace{7mm}}}

\put(55, 5){\line(0, 1){30}}

\qbezier(55, 35)(70, 40)(75, 0)
\qbezier(55, 5)(70, 0)(75, 40)

\end{picture}
\caption{Diagrams $\BH$, $\BH^T$ and $(\BH^\ell)^T$}
\label{fig:bbH}
\label{fig:H-transpose}
\end{figure}

\begin{figure}[h]
\begin{picture}(170, 70)(0,-10)
\put(-5, 28){$\BH _{0 i}(r)\  =$}
{
\linethickness{1mm}
\put(60, 0){\line(0, 1){60}}
}
\put(60, 32){\uwave{\hspace{7mm}}}
\put(80, 22){\line(0, 1){15}}

\qbezier(80, 22)(80, 20)(125, 0)
\qbezier(80, 37)(80, 39)(125, 60)

\put(90, -10){\small$1$}
\put(90, 0){\line(0, 1){60}}
\put(110, 0){\line(0, 1){60}}
\put(95, 30){...}
\put(126, -10){\small$i$}

\put(140, 0){\line(0, 1){60}}
\put(145, 30){...}
\put(160, 0){\line(0, 1){60}}
\put(160, -10){\small$r$}

\end{picture} 
\caption{The diagrams $\BH _{0 i}(r)$}
\label{fig:Hir}
\end{figure}

\begin{figure}[h]
\begin{picture}(200, 70)(-70,-10)
\put(-45, 28){$\BH _{i j}(r)\  =$}
{
\linethickness{1mm}
\put(20, 0){\line(0, 1){60}}
}


\put(40, 0){\line(0, 1){60}}
\put(45, 30){...}

\put(60, 0){\line(0, 1){60}}
\put(60, -10){\small$i$}

\put(60, 30){\line(1, 0){20}}
\put(80, 22){\line(0, 1){15}}

\qbezier(80, 22)(80, 20)(125, 0)
\qbezier(80, 37)(80, 39)(125, 60)

\put(90, 0){\line(0, 1){60}}
\put(110, 0){\line(0, 1){60}}
\put(95, 30){...}
\put(126, -10){\small$j$}

\put(140, 0){\line(0, 1){60}}
\put(145, 30){...}
\put(160, 0){\line(0, 1){60}}
\put(160, -10){\small$r$}

\end{picture} 
\caption{The diagrams $\BH _{i j}(r)$}
\label{fig:Hijr}
\end{figure}

\end{itemize}

\subsection{Composition of polar diagrams} 

\begin{definition} For any $r, r', s$, we 
define a $K$-bilinear map  
\beq\label{eq:comp}
{\bf H}(r, r')\times {\bf H}(r', s) \lra {\bf H}(r, s) 
\eeq
by composition of diagrams. Explicitly,  
for any polar $(r, r')$-diagram $\A$ and $(r', s)$-diagram $\B$ shown in the first two diagrams in Figure \ref{fig:compose},  
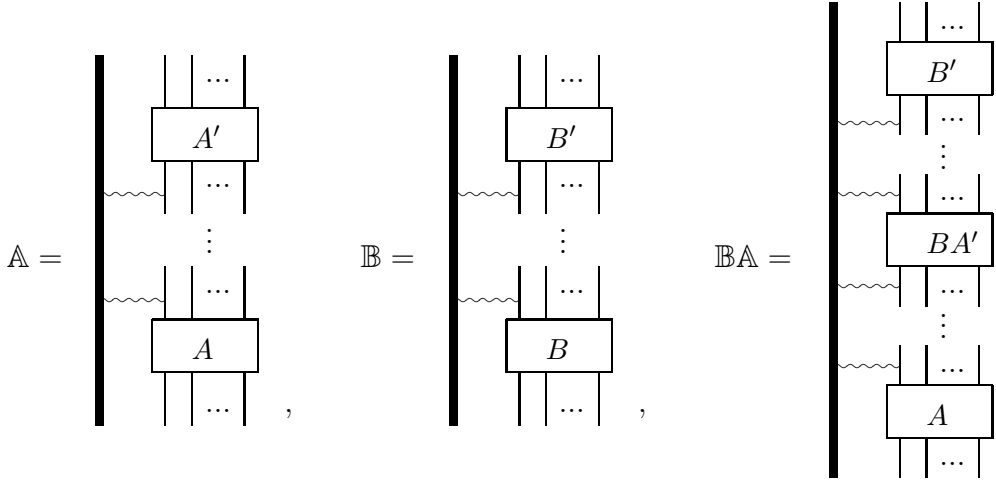
\begin{figure}[h]
\begin{picture}(100, 140)(0, 30)
\put(-35, 110){${\A}  =$}
{
\linethickness{1mm}
\put(0, 50){\line(0, 1){140}}
}
\put(25, 170){\line(0, 1){20}}
\put(35, 170){\line(0, 1){20}}
\put(55, 170){\line(0, 1){20}}
\put(40, 180){...}

\put(35, 155){\small$A'$}
\put(20, 150){\line(0, 1){20}}
\put(60, 150){\line(0, 1){20}}
\put(20, 150){\line(1, 0){40}}
\put(20, 170){\line(1, 0){40}}

\put(25, 130){\line(0, 1){20}}
\put(35, 130){\line(0, 1){20}}
\put(55, 130){\line(0, 1){20}}
\put(40, 140){...}

\put(0, 140){\uwave{\hspace{9mm}}}

\put(40, 115){$\vdots$}

\put(0, 100){\uwave{\hspace{9mm}}}

\put(25, 90){\line(0, 1){20}}
\put(35, 90){\line(0, 1){20}}
\put(55, 90){\line(0, 1){20}}
\put(40, 100){...}
\put(35, 75){\small$A$}
\put(20, 70){\line(0, 1){20}}
\put(60, 70){\line(0, 1){20}}
\put(20, 70){\line(1, 0){40}}
\put(20, 90){\line(1, 0){40}}

\put(25, 70){\line(0, -1){20}}
\put(35, 70){\line(0, -1){20}}
\put(55, 70){\line(0, -1){20}}
\put(40, 55){...}

\put(70, 55){,}
\end{picture} 
\begin{picture}(100, 80)(-30, 30)
\put(-35, 110){${\B}  =$}
{
\linethickness{1mm}
\put(0, 50){\line(0, 1){140}}
}
\put(25, 170){\line(0, 1){20}}
\put(35, 170){\line(0, 1){20}}
\put(55, 170){\line(0, 1){20}}
\put(40, 180){...}

\put(35, 155){\small$B'$}
\put(20, 150){\line(0, 1){20}}
\put(60, 150){\line(0, 1){20}}
\put(20, 150){\line(1, 0){40}}
\put(20, 170){\line(1, 0){40}}

\put(25, 130){\line(0, 1){20}}
\put(35, 130){\line(0, 1){20}}
\put(55, 130){\line(0, 1){20}}
\put(40, 140){...}

\put(0, 140){\uwave{\hspace{9mm}}}

\put(40, 115){$\vdots$}

\put(0, 100){\uwave{\hspace{9mm}}}

\put(25, 90){\line(0, 1){20}}
\put(35, 90){\line(0, 1){20}}
\put(55, 90){\line(0, 1){20}}
\put(40, 100){...}
\put(35, 75){\small$B$}
\put(20, 70){\line(0, 1){20}}
\put(60, 70){\line(0, 1){20}}
\put(20, 70){\line(1, 0){40}}
\put(20, 90){\line(1, 0){40}}

\put(25, 70){\line(0, -1){20}}
\put(35, 70){\line(0, -1){20}}
\put(55, 70){\line(0, -1){20}}
\put(40, 55){...}
\put(70, 55){,}
\end{picture} 
\begin{picture}(110, 180)(-70,-5)
\put(-45, 75){$\B\A  =$}
{
\linethickness{1mm}
\put(0, -5){\line(0, 1){180}}
}
\put(25, 160){\line(0, 1){15}}
\put(35, 160){\line(0, 1){15}}
\put(55, 160){\line(0, 1){15}}
\put(40, 165){...}

\put(35, 145){\small$B'$}
\put(20, 140){\line(0, 1){20}}
\put(60, 140){\line(0, 1){20}}
\put(20, 140){\line(1, 0){40}}
\put(20, 160){\line(1, 0){40}}

\put(25, 125){\line(0, 1){15}}
\put(35, 125){\line(0, 1){15}}
\put(55, 125){\line(0, 1){15}}
\put(40, 130){...}

\put(0, 132){\uwave{\hspace{9mm}}}

\put(40, 112){\small$\vdots$}

\put(0, 105){\uwave{\hspace{9mm}}}

\put(25, 95){\line(0, 1){15}}
\put(35, 95){\line(0, 1){15}}
\put(55, 95){\line(0, 1){15}}
\put(40, 100){...}
\put(35, 80){\small$B A'$}
\put(20, 75){\line(0, 1){20}}
\put(60, 75){\line(0, 1){20}}
\put(20, 75){\line(1, 0){40}}
\put(20, 95){\line(1, 0){40}}

\put(25, 75){\line(0, -1){15}}
\put(35, 75){\line(0, -1){15}}
\put(55, 75){\line(0, -1){15}}
\put(40, 65){...}

\put(0, 70){\uwave{\hspace{9mm}}}

\put(40, 48){\small$\vdots$}

\put(25, 45){\line(0, -1){15}}
\put(35, 45){\line(0, -1){15}}
\put(55, 45){\line(0, -1){15}}
\put(40, 35){...}

\put(0, 40){\uwave{\hspace{9mm}}}

\put(35, 15){\small$A$}
\put(20, 10){\line(0, 1){20}}
\put(60, 10){\line(0, 1){20}}
\put(20, 10){\line(1, 0){40}}
\put(20, 30){\line(1, 0){40}}

\put(25, 10){\line(0, -1){15}}
\put(35, 10){\line(0, -1){15}}
\put(55, 10){\line(0, -1){15}}
\put(40, 0){...}
\end{picture} 
\caption{Composition of polar diagrams}
\label{fig:compose}
\end{figure}
the composition $\B\A$ is the polar $(r, s)$-diagram depicted in the last diagram of Figure 
\ref{fig:compose} with $B A'$ being the composition of the two usual Brauer diagrams $B$ and $A'$, 
where the relations in $\CB(\delta)$ may be applied to simplify $B A'$.  Thus the result of composing two polar diagrams generally is a linear combination of diagrams.
We refer to the collection of the maps \eqref{eq:comp} for all $r, r', s$ as {\em vertical multiplication} in ${\bf H}$. 
\end{definition}

The operation \eqref{eq:comp}  gives rise to a unital associative $K$-algebra structure for ${\bf H}(r, r)$ for any $r$. The following simple properties of 
this multiplication are easily verified. 
\begin{lem}\label{lem:compose-2}
\begin{enumerate}[a)]
\item For any polar $(r, s)$-diagram $\A$, 
\[
\A \I_r = \I_s \A = \A.
\]
\item For any usual Brauer diagrams $A\in \Hom_{\CB(\delta)}(r, r')$ and $B\in \Hom_{\CB(\delta)}(r', s)$, write $D=BA$, then
\[
{\mathbb D}_0=\B_0\A_0,
\] 
where the order $0$ polar diagrams $\A_0$ etc. are as depicted in Figure \ref{fig:ord0}. 
\item Vertical multiplication preserves the grading defined by order, that is,  the composition $\B\A$ of two homogeneous polar diagrams $\B$ and $\A$
 is homogeneous, of order equal to the sum of the orders of $\B$ and $\A$. 
\end{enumerate}
\end{lem}

\begin{remark}\label{rmk:cat-alg}
If we decree that the set of $\I_r$ be mutually orthogonal idempotents under vertical multiplication, that is, $\I_r \I_s=\delta_{r s}$ for all $r, s$, then ${\bf H}$ becomes an associative algebra with
 vertical multiplication as its multiplication. 
\end{remark}

\subsection{Juxtaposition of diagrams and horizontal multiplication} 
%
%

For any polar $(r, s)$-diagram $\A$ and usual Brauer $(k, \ell)$ diagram $B$, we juxtapose $\A$ with $B$ by placing $B$ on the right of $\A$ without overlap. This yields a polar $(r+k, s+\ell)$-diagram, which we denote by $\A\ot B$. We can extend juxtaposition $K$-bilinearly to a map 
\beq\label{eq:ot-polar}
\ot: {\bf H}\times \Hom(\CB(\delta)) \lra {\bf H}, 
\eeq 
which makes ${\bf H}$ a free right module for $(\Hom(\CB(\delta)), \ot)$. 
Call this map {\em horizontal or (right) tensor multiplication} of ${\bf H}$ by $\Hom(\CB(\delta))$.

In this context, we adopt the following notation
\beq\label{eq:Hr}
\BH _{01}(r) = \BH \ot I_{r-1}, \quad \A_0=\I_0\ot A  \text{  \  for any Brauer diagram $A$}.
\eeq

With this notation, the following statement is evident. 
\begin{lem} \label{lem:generators}
\begin{enumerate}
\item The polar diagrams are generated by $\I_0$, $\BH $ and usual Brauer diagrams under vertical and horizontal multiplication. 

\item For any $A, B\in \Hom_{\CB(\delta)}$ such that $BA$ is defined, 
\[
\begin{aligned}
&(\I_0\ot B)(\I_0\ot A) = \I_0\ot BA; \\
&(\BH\ot B)(\I_1\ot A) = (\I_1\ot B)(\BH\ot A)=\BH\ot BA.
\end{aligned}
\]

\end{enumerate}
\end{lem}
\begin{proof}
Any polar diagram can be represented graphically as in Figure \ref{fig:AffD}, from which we can see that it is the composition of $\I_0\ot D_i$ and $\BH_{0 1}(r_i)$.  
Part (2) of the lemma is a consequence of the going-up and going-down relations. 
\end{proof}

\subsection{Polar Brauer diagrams} 
As was pointed out in Remark \ref{rmk:cat-alg}, vertical multiplication makes ${\bf H}$ into an algebra. We  now consider a particular quotient of this algebra. 
In view of Lemma \ref{lem:generators}(1), we only need to impose relations among the generators appearing there. 

We call a $K$-submodule ${\bf J}$ of ${\bf H}$ a {\em $2$-ideal} with respect to the vertical and horizontal multiplications if  ${\bf J}$ is a two-sided ideal with respect to
 vertical multiplication, and a right submodule for $(\Hom(\CB(\delta)), \ot)$ with respect to horizontal multiplication. 

Let us write $\BH _{0 i}= \BH _{0 i}(2)$ for $i=1, 2$ as shown in Figure \ref{fig:Hir}, and $\BH _{12}=\I_0\ot H$, where  $H$ is the combination of Brauer $(2, 2)$-diagrams 
depicted in Figure \ref{fig:t-image}. We also recall
the elements $\BH, \BH^T$ from Figure \ref{fig:H-transpose}.

\begin{definition}
Call $\wh{\bf H}= {\bf H}/\langle \CR\rangle$ the  $K$-module of {\em affine, or polar, Brauer diagrams}, where 
$\langle \CR\rangle$ is the $2$-ideal of ${\bf H}$ generated by
\[
\CR=\{[ \BH _{0 1},  \BH _{0 2} + \BH _{12}], \ \BH ^T +\BH \}
\]
by horizontal multiplication with $\Hom(\CB(\delta))$ and vertical multiplication. 
\end{definition}

We note that these relations are adapted to our later
use of $\wh {\bf H}$ for the analysis of the $\osp(m|2n)$-modules $M\ot V^{\ot r}$, where $M$ is arbitrary and $V$ is the natural module for $\osp(m|2n)$.

Given any polar $(r, s)$-diagram in ${\bf H}$, we shall use the same picture to represent its image in $\wh{\bf H}$, and call it an {\em affine, or polar, Brauer $(r, s)$-diagram}.  We denote the $K$-module of affine Brauer $(r, s)$-diagram by $\wh{\bf H}(r, s)$. 

Note that the element $[ \BH _{0 1},  \BH _{0 2} + \BH _{12}]$ of $\CR$ is not homogeneous, thus the order no longer provides a $\Z_+$-grading for $\wh{\bf H}$. Instead, for any $(r, s)$, 
 we have the filtration
 
 \beq
F\wh{\bf H}(r, s)_0\subset F\wh{\bf H}(r, s)_1\subset F\wh{\bf H}(r, s)_2\subset ..., 
\eeq
where $F\wh{\bf H}(r, s)_i$ is the image of $\sum_{j\le i}{\bf H}_j(r, s)$ in $\wh{\bf H}$.
 Denote 
$F\wh{\bf H}_k=\sum_{r, s} F\wh{\bf H}(r, s)_k$.

Vertical multiplication clearly  descends to $\wh{\bf H}$.
Note that as $K$-modules, $F\wh{\bf H}_0\cong {\bf H}_0\cong \Hom(\CB(\delta))$.  We will write $\wh{\bf H}_0$ for $F\wh{\bf H}_0$. 
Thus horizontal multiplication also descends to $\wh{\bf H}$, leading to the map
\beq\label{eq:ot}
\ot: \wh{\bf H}\times \Hom(\CB(\delta)) \lra \wh{\bf H}, 
\eeq 
which defines a right $(\Hom(\CB(\delta)), \ot)$-module structure on $\wh{\bf H}$.

We use the same symbols to denote the images in $\wh{\bf H}$ of the elements described in Lemma \ref{lem:generators}, which generate $\wh{\bf H}$.  
With this convention, the space $\wh{\bf H}$ may be described in terms of generators and relations as follows. 

\begin{theorem}
$\wh{\bf H}$ is generated by the elements 
$\I_0, \, \BH$ and $\A_0=\I_0\ot A$ for all $A$ $\in$ $\Hom(\CB(\delta))$, 
by vertical multiplication and horizontal multiplication \eqref{eq:ot},  subject to the following relations. 
\begin{enumerate}
\item {\bf Brauer relations}: for all $A, B\in \Hom_{\CB(\delta)}$ such that $BA$ is defined, 
\[
\begin{aligned}
&(\I_0\ot B)(\I_0\ot A) - \I_0\ot BA=0; \\
&(\BH\ot B)(\I_1\ot A) = (\I_1\ot B)(\BH\ot A)=\BH\ot BA;
\end{aligned}
\]

\item {\bf Four-term relation}: 
\[
[ \BH _{0 1}, \BH _{0 2} + \BH _{12}]=0; 
\]

\item {\bf Skew symmetry}:  
\[
\BH ^T+\BH =0.
\]

\end{enumerate}
\end{theorem}
\begin{proof} It follows Lemma \ref{lem:generators}(1) that $\wh{\bf H}$ is generated by the given elements.
The Brauer relations were built into the definition of ${\bf H}$, while the four-term relation and skew symmetry of $\mathbb{H}$ are
the additional relations needed to pass from $\bf H$ to the quotient ${\bf H}/\langle \CR\rangle$.
\end{proof}

\noindent{\bf Diagrammatic representations}. 
The four-term relation can be represented diagrammatically by Figure \ref{fig:4-term}, and the skew symmetry of $\BH $ by Figure \ref{fig:H-skew}. These graphical representations provide the basis for diagrammatic investigations of the various categories to be introduced in later sections. 

\begin{figure}[h]
\begin{picture}(50, 70)(0,0)
{
\linethickness{1mm}
\put(10, 0){\line(0, 1){70}}
}

\put(10, 50){\uwave{\hspace{5mm}}}
\put(10, 25){\uwave{\hspace{5mm}}}

\put(25, 70){\line(0, -1){30}}
\put(40, 70){\line(0, -1){30}}

\qbezier(25, 40)(25, 40)(40, 30)
\qbezier(25, 30)(25, 30)(40, 40)

\put(25, 30){\line(0, -1){15}}
\put(40, 30){\line(0, -1){15}}
\qbezier(25, 15)(25, 15)(40, 0)
\qbezier(25, 0)(25, 0)(40, 15)
\put(45, 30){$-$}
\end{picture}
\begin{picture}(50, 70)(0,0)
{
\linethickness{1mm}
\put(10, 0){\line(0, 1){70}}
}

\put(10, 50){\uwave{\hspace{5mm}}}
\put(10, 25){\uwave{\hspace{5mm}}}

\qbezier(25, 55)(25, 55)(40, 70)
\qbezier(25, 70)(25, 70)(40, 55)

\put(25, 55){\line(0, -1){15}}
\put(40, 55){\line(0, -1){15}}

\put(25, 30){\line(0, -1){30}}
\put(40, 30){\line(0, -1){30}}
\qbezier(25, 30)(25, 30)(40, 40)
\qbezier(25, 40)(25, 40)(40, 30)
\put(45, 30){$+$}
\end{picture}
\begin{picture}(50, 70)(0,0)
{
\linethickness{1mm}
\put(10, 0){\line(0, 1){70}}
}

\put(10, 50){\uwave{\hspace{5mm}}}

\put(25, 70){\line(0, -1){40}}
\put(40, 70){\line(0, -1){40}}


\put(25, 30){\line(0, -1){15}}
\put(40, 30){\line(0, -1){15}}
\qbezier(25, 15)(25, 15)(40, 0)
\qbezier(25, 0)(25, 0)(40, 15)
\put(45, 30){$-$}
\end{picture}
\begin{picture}(50, 70)(0,0)
{
\linethickness{1mm}
\put(10, 0){\line(0, 1){70}}
}

\put(10, 25){\uwave{\hspace{5mm}}}

\qbezier(25, 55)(25, 55)(40, 70)
\qbezier(25, 70)(25, 70)(40, 55)

\put(25, 55){\line(0, -1){15}}
\put(40, 55){\line(0, -1){15}}

\put(25, 40){\line(0, -1){40}}
\put(40, 40){\line(0, -1){40}}
\put(45, 30){$+$}
\end{picture}
\begin{picture}(50, 70)(0,0)
{
\linethickness{1mm}
\put(10, 0){\line(0, 1){70}}
}

\put(10, 25){\uwave{\hspace{5mm}}}
\qbezier(25, 60)(33, 45)(40, 60)
\put(25, 60){\line(0, 1){10}}
\put(40, 60){\line(0, 1){10}}


\qbezier(25, 40)(33, 55)(40, 40)
\put(25, 40){\line(0, -1){40}}
\put(40, 40){\line(0, -1){40}}
\put(45, 30){$-$}
\end{picture}
\begin{picture}(50, 70)(0,0)
{
\linethickness{1mm}
\put(10, 0){\line(0, 1){70}}
}

\put(10, 50){\uwave{\hspace{5mm}}}

\qbezier(25, 40)(33, 25)(40, 40)
\put(25, 70){\line(0, -1){30}}
\put(40, 70){\line(0, -1){30}}


\qbezier(25, 20)(33, 35)(40, 20)
\put(25, 20){\line(0, -1){20}}
\put(40, 20){\line(0, -1){20}}
\put(45, 30){$=0$.}
\end{picture}
\caption{Four-term relation}
\label{fig:4-term}
\end{figure}  

\begin{figure}[h]
\begin{picture}(100, 40)(0,0)
{
\linethickness{1mm}
\put(35, 0){\line(0, 1){40}}
}
\put(35, 20){\uwave{\hspace{7mm}}}

\put(55, 8){\line(0, 1){20}}

\qbezier(55, 28)(70, 38)(75, 0)
\qbezier(55, 8)(70, -2)(75, 40)

\end{picture}
\begin{picture}(100, 40)(20, 0)
\put(0, 16){$=\ - \ $}
{
\linethickness{1mm}
\put(35, 0){\line(0, 1){40}}
}
\put(35, 20){\uwave{\hspace{7mm}}}

\put(55, 0){\line(0, 1){40}}

\put(65, 5){. }

\end{picture} 
\caption{Skew symmetry of  $\BH $}
\label{fig:H-skew}
\end{figure}


The following result easily follows from the definition. 
\begin{lem}\label{lem:4-t-id}  The following identity holds in $\wh{\bf H}$.
\[
[\mathbb{H}_{01}+ \mathbb{H}_{0 2}, \mathbb{H}_{12}]=0.
\]
\end{lem}
\begin{proof} 
This is a consequence of the skew symmetry of $\BH$, and is independent of the four-term identity. 
The proof is entirely straightforward, but provides a good opportunity to illustrate the use of the skew symmetry of $\BH$. We therefore spell it out.   We have
\[
\begin{picture}(110, 40)(0, 0)
\put(-30, 20){$\BH _{0 1} \BH _{12} \  =$}
{
\linethickness{1mm}
\put(35, 0){\line(0, 1){40}}
}
\put(35, 30){\uwave{\hspace{5mm}}}

\put(50, 40){\line(0, -1){25}}
\put(65, 40){\line(0, -1){25}}

\qbezier(50, 15)(50, 15)(65, 0)
\qbezier(65, 15)(65, 15)(50, 0)


\put(80, 18){$-$}
\end{picture}
\begin{picture}(110, 40)(30, 0)
{
\linethickness{1mm}
\put(35, 0){\line(0, 1){40}}
}
\put(35, 30){\uwave{\hspace{5mm}}}

\put(50, 40){\line(0, -1){25}}
\put(65, 40){\line(0, -1){25}}

\qbezier(50, 15)(57, 5)(65, 15)

\qbezier(50, 0)(57, 10)(65, 0)



\put(80, 5){,}
\end{picture}
\]

\[
\begin{picture}(110, 40)(0, 0)
\put(-30, 20){$\BH _{0 2} \BH _{12} \  =$}
{
\linethickness{1mm}
\put(35, 0){\line(0, 1){40}}
}
\put(35, 15){\uwave{\hspace{5mm}}}

\put(50, 25){\line(0, -1){25}}
\put(65, 25){\line(0, -1){25}}

\qbezier(50, 25)(50, 25)(65, 40)
\qbezier(65, 25)(65, 25)(50, 40)


\put(80, 18){$-$}
\end{picture}
\begin{picture}(110, 40)(30, 0)
{
\linethickness{1mm}
\put(35, 0){\line(0, 1){40}}
}
\put(35, 22){\uwave{\hspace{5mm}}}

\put(50, 15){\line(0, 1){10}}
\put(65, 15){\line(0, 1){10}}
\qbezier(50, 25)(50, 25)(65, 40)
\qbezier(50, 40)(50, 40)(65, 25)

\qbezier(50, 15)(57, 5)(65, 15)
\qbezier(50, 0)(57, 10)(65, 0)



\put(80, 5){,}
\end{picture}
\]

%
\[
\begin{picture}(110, 40)(0, 0)
\put(-30, 20){$\BH _{12} \BH _{01} \  =$}
{
\linethickness{1mm}
\put(35, 0){\line(0, 1){40}}
}
\put(35, 15){\uwave{\hspace{5mm}}}

\put(50, 25){\line(0, -1){25}}
\put(65, 25){\line(0, -1){25}}

\qbezier(50, 25)(50, 25)(65, 40)
\qbezier(65, 25)(65, 25)(50, 40)


\put(80, 18){$-$}
\end{picture}
\begin{picture}(110, 40)(30, 0)
{
\linethickness{1mm}
\put(35, 0){\line(0, 1){40}}
}
\put(35, 15){\uwave{\hspace{5mm}}}

\put(50, 0){\line(0, 1){25}}
\put(65, 0){\line(0, 1){25}}

\qbezier(50, 25)(57, 35)(65, 25)

\qbezier(50, 40)(57, 30)(65, 40)



\put(80, 5){,}
\end{picture}
\]

\[
\begin{picture}(110, 40)(0, 0)
\put(-30, 20){$\BH _{12} \BH _{02} \  =$}
{
\linethickness{1mm}
\put(35, 0){\line(0, 1){40}}
}
\put(35, 30){\uwave{\hspace{5mm}}}

\put(50, 40){\line(0, -1){25}}
\put(65, 40){\line(0, -1){25}}

\qbezier(50, 15)(50, 15)(65, 0)
\qbezier(65, 15)(65, 15)(50, 0)


\put(80, 18){$-$}
\end{picture}
\begin{picture}(110, 40)(30, 0)
{
\linethickness{1mm}
\put(35, 0){\line(0, 1){40}}
}
\put(35, 18){\uwave{\hspace{5mm}}}

\qbezier(50, 40)(57, 25)(65, 40)
\qbezier(50, 20)(57, 35)(65, 20)

\put(50, 20){\line(0, -1){10}}
\put(65, 20){\line(0, -1){10}}

\qbezier(50, 10)(50, 10)(65, 0)
\qbezier(50, 0)(50, 0)(65, 10)

\put(80, 5){.}
\end{picture}
\]
The skew symmetry of $\BH $ implies that
\[
\begin{picture}(100, 40)(30, 0)
{
\linethickness{1mm}
\put(35, 0){\line(0, 1){40}}
}
\put(35, 22){\uwave{\hspace{5mm}}}

\put(50, 15){\line(0, 1){10}}
\put(65, 15){\line(0, 1){10}}
\qbezier(50, 25)(50, 25)(65, 40)
\qbezier(50, 40)(50, 40)(65, 25)

\qbezier(50, 15)(57, 5)(65, 15)
\qbezier(50, 0)(57, 10)(65, 0)

\put(80, 20){$ = \ -$}
\end{picture}
\begin{picture}(50, 40)(50, 0)
{
\linethickness{1mm}
\put(35, 0){\line(0, 1){40}}
}
\put(35, 30){\uwave{\hspace{5mm}}}

\put(50, 40){\line(0, -1){25}}
\put(65, 40){\line(0, -1){25}}

\qbezier(50, 15)(57, 5)(65, 15)

\qbezier(50, 0)(57, 10)(65, 0)

\put(75, 5){, }
\end{picture}
%
\begin{picture}(110, 40)(5, 0)
\put(-8, 18){$\text{and}$}
{
\linethickness{1mm}
\put(35, 0){\line(0, 1){40}}
}
\put(35, 18){\uwave{\hspace{5mm}}}

\qbezier(50, 40)(57, 25)(65, 40)
\qbezier(50, 20)(57, 35)(65, 20)

\put(50, 20){\line(0, -1){10}}
\put(65, 20){\line(0, -1){10}}

\qbezier(50, 10)(50, 10)(65, 0)
\qbezier(50, 0)(50, 0)(65, 10)

\put(80, 20){$=\ - $}
\end{picture}
\begin{picture}(60, 40)(35, 0)
{
\linethickness{1mm}
\put(35, 0){\line(0, 1){40}}
}
\put(35, 15){\uwave{\hspace{5mm}}}

\put(50, 0){\line(0, 1){25}}
\put(65, 0){\line(0, 1){25}}

\qbezier(50, 25)(57, 35)(65, 25)

\qbezier(50, 40)(57, 30)(65, 40)



\put(70, 5){.}
\end{picture}
\]
\noindent
Using these relations in the formulae above, we obtain
\[
(\BH _{0 1} + \BH _{0 2} )\BH _{12}=
\BH _{12}(\BH _{0 1} + \BH _{0 2} ),
\]
which completes the proof.  
\end{proof}

The next result is straightforward.

\begin{corollary}\label{cor:4t}
\begin{enumerate} The following identities hold in $\wh{\bf H}$.
\item We have $[\BH_{01},\BH_{12}]=[\BH_{12},\BH_{02}]=[\BH_{02},\BH_{01}]$.
\item  Given only the skew symmetry of $\BH$, the four-term identity is equivalent to
\be\label{eq:4tin}
[\BH_{02},\BH_{01}+\BH_{12}]=0.
\ee
\end{enumerate}
\end{corollary}

\begin{proof}
The equality of the first and third terms of (1) is the four term identity. The equality of the first and second terms
is Lemma \ref{lem:4-t-id}. The equality of the second and third terms is evidently 
a formal consequence of these.

Now we have seen that the proof of Lemma \ref{lem:4-t-id} does not depend on the four-term identity. But it follows from (1)
that the four-term identity is a consequence of \eqref{eq:4tin} and Lemma \ref{lem:4-t-id}. The result follows. 
\end{proof}

\subsection{Polar Brauer category  $\AB(\delta)$}

We now introduce the diagrammatic affine Brauer category  $\AB(\delta)$ and some of its quotient categories, and develop their basic properties. 

\subsubsection{Definition of $\AB(\delta)$}

\begin{definition}
Fix $\delta\in K$. The pre-additive category $\AB(\delta)$ over $K$ is defined by the following properties.
\begin{enumerate}
\item its objects are the elements of $\N$;  
\item for any objects $r, s$ the hom-set $\Hom(r, s)$  is the $K$-module $\wh{\bf H}(r, s)$; and
\item 	the composition of morphisms is vertical multiplication in $\wh{\bf H}$. 
\end{enumerate}
Denote this category by $\AB(\delta)$, and refer to it as the
{\em diagrammatic affine Brauer category}, or {\em polar Brauer category}, with parameter $\delta$. 
\end{definition} 

Part (1) and part (2) of the following theorem are clear from the construction of $\wh{\bf H}$ and definition of $\AB(\delta)$.
\begin{thm} 
\begin{enumerate}
\item
The diagrammatic affine Brauer category $\AB(\delta)$ has a subcategory $\AB_0(\delta)$ with objects $\N$ and hom-sets $\Hom_{\AB_0(\delta)}(r, s)=\wh{\bf H}(r, s)_0$ for any objects $r, s$.
This is isomorphic to the Brauer category $\CB(\delta)$. The diagrams occuring are those of rank $0$.
\item There is a bi-functor 
$
\ot: \AB(\delta)\times \CB(\delta)\lra \AB(\delta), 
$
such that for objects $r$ of $\AB(\delta)$ and $r_0$ of $\CB(\delta)$, $r\ot r_0=r+r_0$ , and for morphisms $\A$ of $\AB(\delta)$ and $B$ of $\CB(\delta)$, $\A\ot B$ is defined by horizontal multiplication \eqref{eq:ot}. Thus $\AB(\delta)$ is a right module category over the monoidal category $\CB(\delta)$. 

\item There is a functor $\varpi: \AB(\delta)\lra \CB(\delta)$ which sends any object $r$ of $\AB(\delta)$ to $r+1$,  morphisms $\I_0$ to $I$, $\BH$ to $H$,  and respects
 tensor product with $\CB(\delta)$ in the sense that the following diagram commutes, 
\[
\begin{tikzcd}
\AB(\delta) \times \CB(\delta)\arrow[d,  "\varpi\times \id"'pos=0.43] \arrow[r, "\ot"] & \AB(\delta)\arrow[d, "\varpi"]\\
\CB(\delta)\times \CB(\delta) \arrow[r, "\ot"' pos=0.43]& \CB(\delta).
\end{tikzcd}
\]

\end{enumerate}
\end{thm}
\begin{proof}
Only part (3) requires proof. Note that the given conditions define $\varpi$ uniquely, so that we only need to check that $\varpi$ respects the four-term relation 
and skew symmetry of $\BH$.  

The skew symmetry of $H$, depicted in Figure 3 and proved just below  that figure,
shows that $\varpi$ respects the skew symmetry of $\BH$. 
The fact that $\varpi$ respects the four-term relation is an immediate consequence of Lemma \ref{lem:4-t-Brauer}. 
To give this fact an independent proof, we let $H_{01}=H\ot I, H_{02}=(I\ot X)(H\ot I)(I\ot X)$
and $H_{12}=I\ot H$. Then it follows Corollary \ref{cor:4t}(2) that 
the four term relation reduces to the following statement in $\CB(\delta)$. 
\beq\label{eq:4t}
[H_{02},H_{01}+ H_{12}]=0.
\eeq
A detailed but straightforward calculation shows that in the notation of  Figures \ref{fig:siei} and \ref{fig:Xij},  we have
 \[
 \baln
 {[H_{02},H_{01}]}&=s_1s_2-s_2s_1+e_2s_1+e_1e_2-e_1e_2s_1-s_1e_2-e_2e_1+s_1e_2e_1\\
 &=[H_{12},H_{02}],
 \ealn
 \]
which proves \eqref{eq:4t}. This completes the proof of the theorem.  
\end{proof}

\begin{remark}
The category algebra of $\AB(\delta)$ is $\wh{\bf H}$ with vertical multiplication. 
\end{remark}

We next  describe the structures of $\Hom_{\AB(\delta)}(0, 0)$ and $\Hom_{\AB(\delta)}(1, 1)$. 
These results will be used for studying universal enveloping superalgebras below.

\subsubsection{Description of $\Hom_{\AB(\delta)}(0, 0)$ and $\Hom_{\AB(\delta)}(1, 1)$}
{\  }

Denote by $\Pi$ and $\amalg$ respectively the affine Brauer $(2, 0)$- and $(0, 2)$-diagrams shown in Figure \ref{fig:A-U}, and let $\X_0= \I_0\ot X$ 
with $X\in \Hom_{\CB(\delta)}(2, 2)$ the usual generator of $\CB(\delta)$ as depicted in Figure \ref{fig:generators}.

\begin{figure}[h]
\begin{picture}(85, 30)(-20, 5)
\put(-30, 10){$\Pi=$}
{
\linethickness{1mm}
\put(0, 0){\line(0, 1){30}}
}
\qbezier(10, 0)(17, 40)(25, 0)
\put(30, 5){,}
\end{picture}
\begin{picture}(30, 30)(-20, 5)
\put(-30, 10){$\amalg=$}
{
\linethickness{1mm}
\put(0, 0){\line(0, 1){30}}
}
\qbezier(10, 30)(17, -10)(25, 30)
\end{picture}
\caption{Affine Brauer diagrams $\Pi$ and $\amalg$}
\label{fig:A-U}
\end{figure}

Let us introduce the following elements of $\Hom_{\AB(\delta)}(0, 0)$
\beq\label{eq:Z-generators}
 Z_\ell=\Pi (\BH^\ell \ot I)\amalg, \quad  \ell\ge 1, 
\eeq
which may be represented graphically by Figure \ref{fig:Z}.
\begin{figure}
\begin{picture}(110, 60)(50, -5)
\put(0, 18){$Z_\ell  =$}
{
\linethickness{1mm}
\put(35, -5){\line(0, 1){55}}
}
\put(35, 35){\uwave{\hspace{7mm}}}
\put(35, 28){\uwave{\hspace{7mm}}}
\put(35, 12){\uwave{\hspace{7mm}}}
\put(44, 13){{\tiny$\vdots$}}

\put(55, 5){\line(0, 1){35}}

\qbezier(55, 40)(70, 50)(73, 22)
\qbezier(55, 5)(70, -5)(73, 22)

\put(95, 15){with $\ell$ connectors.}
\end{picture}
\caption{Central elements $Z_\ell$}
\label{fig:Z}
\end{figure}

We have the following result. 

\begin{lem}   \label{lem:central} The elements $Z_1, Z_2$ and $Z_3$ have the following properties:
\begin{enumerate}

\item If the characteristic of $K$ is not two, then $Z_1=0$, and $ 2 Z_3 = (2-\delta)Z_2$;

\item $Z_2\ot 1$ and $\BH$ commute. It follows that $Z_2$ is central in the sense that $(Z_2\ot I_s) \A= \A (Z_2\ot I_r)$ for all $\A\in\Hom_{\AB(\delta)}(r, s)$. 

\end{enumerate}
\end{lem}
\begin{proof} 
It is an immediate consequence of the skew symmetry of $\BH $ that $Z_1=-Z_1=0$ if $2\neq 0$.  

To prove the second relation in (1), we pre-multiply (vertically) the four-term relation in Figure \ref{fig:4-term} by $\Pi$,  and then apply skew symmetry of $\BH$ several times to obtain the relation depicted in Figure \ref{fig:deg0-relation}.  
\begin{figure}[h]
\begin{picture}(60, 55)(0,0)
{
\linethickness{1mm}
\put(10, 0){\line(0, 1){55}}
}

\put(10, 40){\uwave{\hspace{5mm}}}
\put(10, 25){\uwave{\hspace{5mm}}}

\put(25, 45){\line(0, -1){25}}
\put(40, 45){\line(0, -1){25}}

\qbezier(25, 45)(33, 60)(40, 45)


\put(25, 30){\line(0, -1){15}}
\put(40, 30){\line(0, -1){15}}
\qbezier(25, 15)(25, 15)(40, 0)
\qbezier(25, 0)(25, 0)(40, 15)
\put(50, 23){$-$}
\end{picture}
\begin{picture}(100, 50)(0,0)
{
\linethickness{1mm}
\put(10, 0){\line(0, 1){55}}
}

\put(10, 40){\uwave{\hspace{5mm}}}
\put(10, 25){\uwave{\hspace{5mm}}}

\qbezier(25, 45)(32, 60)(40, 45)

\put(25, 45){\line(0, -1){45}}
\put(40, 45){\line(0, -1){45}}
\put(50, 23){$= \ (\delta-2)$}
\end{picture}
\begin{picture}(50, 50)(0,0)
{
\linethickness{1mm}
\put(10, 0){\line(0, 1){55}}
}

\put(10, 25){\uwave{\hspace{5mm}}}

\qbezier(25, 40)(33, 55)(40, 40)
\put(25, 40){\line(0, -1){40}}
\put(40, 40){\line(0, -1){40}}
\put(45, 5){.}
\end{picture}
\caption{A relation in $\Hom_{\AB(\delta)}(1, 1)$}
\label{fig:deg0-relation}
\end{figure}
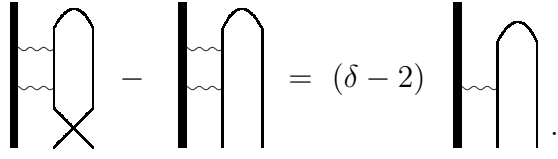  
We then post-multiply (vertically) each diagram in  Figure \ref{fig:deg0-relation}  by $(\BH\ot I) \amalg$ and again apply the skew symmetry of $\BH$. This leads directly to 
$
2 Z_3 = (2-\delta) Z_2. 
$

To prove part (2), we define
\beq
\BB_\ell= (\I_0\ot I\ot \cap)\left((\BH_{02} + \BH_{12})^\ell \ot I\right)(\I_0\ot I\ot \cup).
\eeq
Then the four-term relation easily implies
\beq\label{eq:Xell}
\BH \BB_\ell = \BB_\ell \BH, \quad \forall \ell. 
\eeq
In the case $\ell=2$, we have $\BB_2 =Z_2\ot I+ 2(\delta-1)\I_0\ot I$, and hence 
$\BH (Z_2\ot I)= (Z_2\ot I)\BH$ by \eqref{eq:Xell}, where the relation between $\BB_2$ and $Z_2\ot I$ follows from the formulae 
\[
\baln
&(\I_0\ot I\ot \cap)(\BH_{02}^2\ot I)(\I_0\ot I\ot \cup) = Z_2\ot I, \\
&(\I_0\ot I\ot \cap)((\BH_{02}\BH_{12}+\BH_{12}\BH_{02})\ot I)(\I_0\ot I\ot \cup) = 0, \\
&(\I_0\ot I\ot \cap)(\BH_{12}^2\ot I)(\I_0\ot I\ot \cup) =2(\delta-1)\I_0\ot I. 
\ealn
\]

The second assertion of part (2) follows from the first, because any morphism $\A$ of $\AB(\delta)$ is built by composition and tensor product using $\BH$ and ordinary Brauer diagrams,
both of which are now known to commute with $Z_2\ot I^{\ot r}$.
\end{proof}

\begin{lem} \label{lem:HT2} 
Let $\BG_\ell=Z_\ell\ot I  -  \BH^\ell$, and $\Phi=(1-\delta)\I_0\ot I-\BH$. 
Then the following relations hold in  $\Hom_{\AB(\delta)}(1, 1)$ for ll $\ell\ge 0$, 
\beq
(\BH^{\ell+1})^T  - (\BH^\ell)^T \Phi - \BG_\ell =0, \label{eq:HT-H-1}\\
(\BH^{\ell+1})^T - \Phi (\BH^{\ell})^T - \BG_{\ell} =0, \label{eq:HT-H-2}
\eeq
where $(\BH^{\ell})^T$ is the diagram depicted in Figure \ref{fig:bbH}. 
Furthermore, 
$\BH^k$ and $(\BH^{\ell})^T$ commute for all $k, \ell$.
\end{lem}
\begin{proof}
If $\ell=0$, both equations are equivalent to skew symmetry of $\BH$ by noting that $\BG_0=(\delta-1)\I_0\ot I$. 
For any $\ell\ge 1$, we have the iterated four-term relation Figure \ref{fig:4-term-iterated}.

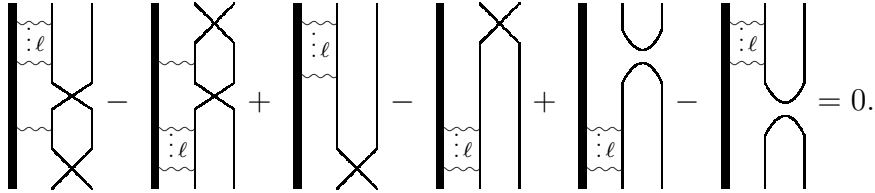
\begin{figure}[h]
\begin{picture}(50, 70)(0,0)
{
\linethickness{1mm}
\put(10, 0){\line(0, 1){70}}
}
\put(10, 65){\uwave{\hspace{5mm}}}
\put(15, 52){\tiny$\vdots\, \ell$}
\put(10, 50){\uwave{\hspace{5mm}}}
\put(10, 25){\uwave{\hspace{5mm}}}

\put(25, 70){\line(0, -1){30}}
\put(40, 70){\line(0, -1){30}}

\qbezier(25, 40)(25, 40)(40, 30)
\qbezier(25, 30)(25, 30)(40, 40)

\put(25, 30){\line(0, -1){15}}
\put(40, 30){\line(0, -1){15}}
\qbezier(25, 15)(25, 15)(40, 0)
\qbezier(25, 0)(25, 0)(40, 15)
\put(45, 30){$-$}
\end{picture}
\begin{picture}(50, 70)(0,0)
{
\linethickness{1mm}
\put(10, 0){\line(0, 1){70}}
}

\put(10, 50){\uwave{\hspace{5mm}}}
\put(10, 25){\uwave{\hspace{5mm}}}
\put(15, 12){\tiny$\vdots\, \ell$}
\put(10, 10){\uwave{\hspace{5mm}}}

\qbezier(25, 55)(25, 55)(40, 70)
\qbezier(25, 70)(25, 70)(40, 55)

\put(25, 55){\line(0, -1){15}}
\put(40, 55){\line(0, -1){15}}

\put(25, 30){\line(0, -1){30}}
\put(40, 30){\line(0, -1){30}}
\qbezier(25, 30)(25, 30)(40, 40)
\qbezier(25, 40)(25, 40)(40, 30)
\put(45, 30){$+$}
\end{picture}
\begin{picture}(50, 70)(0,0)
{
\linethickness{1mm}
\put(10, 0){\line(0, 1){70}}
}

\put(10, 65){\uwave{\hspace{5mm}}}
\put(15, 50){\tiny$\vdots\, \ell$}
\put(10, 45){\uwave{\hspace{5mm}}}

\put(25, 70){\line(0, -1){40}}
\put(40, 70){\line(0, -1){40}}


\put(25, 30){\line(0, -1){15}}
\put(40, 30){\line(0, -1){15}}
\qbezier(25, 15)(25, 15)(40, 0)
\qbezier(25, 0)(25, 0)(40, 15)
\put(45, 30){$-$}
\end{picture}
\begin{picture}(50, 70)(0,0)
{
\linethickness{1mm}
\put(10, 0){\line(0, 1){70}}
}

\put(10, 25){\uwave{\hspace{5mm}}}
\put(15, 12){\tiny$\vdots\, \ell$}
\put(10, 10){\uwave{\hspace{5mm}}}

\qbezier(25, 55)(25, 55)(40, 70)
\qbezier(25, 70)(25, 70)(40, 55)

\put(25, 55){\line(0, -1){15}}
\put(40, 55){\line(0, -1){15}}

\put(25, 40){\line(0, -1){40}}
\put(40, 40){\line(0, -1){40}}
\put(45, 30){$+$}
\end{picture}
\begin{picture}(50, 70)(0,0)
{
\linethickness{1mm}
\put(10, 0){\line(0, 1){70}}
}

\put(10, 25){\uwave{\hspace{5mm}}}
\put(15, 12){\tiny$\vdots\, \ell$}
\put(10, 10){\uwave{\hspace{5mm}}}

\qbezier(25, 60)(33, 45)(40, 60)
\put(25, 60){\line(0, 1){10}}
\put(40, 60){\line(0, 1){10}}


\qbezier(25, 40)(33, 55)(40, 40)
\put(25, 40){\line(0, -1){40}}
\put(40, 40){\line(0, -1){40}}
\put(45, 30){$-$}
\end{picture}
\begin{picture}(50, 70)(0,0)
{
\linethickness{1mm}
\put(10, 0){\line(0, 1){70}}
}

\put(10, 65){\uwave{\hspace{5mm}}}
\put(15, 52){\tiny$\vdots\, \ell$}
\put(10, 50){\uwave{\hspace{5mm}}}

\qbezier(25, 40)(33, 25)(40, 40)
\put(25, 70){\line(0, -1){30}}
\put(40, 70){\line(0, -1){30}}


\qbezier(25, 20)(33, 35)(40, 20)
\put(25, 20){\line(0, -1){20}}
\put(40, 20){\line(0, -1){20}}
\put(45, 30){$=0$.}
\end{picture}
\caption{Iterated four-term relation}
\label{fig:4-term-iterated}
\end{figure}

By pre-multiplying (vertically) the relation in Figure \ref{fig:4-term-iterated} by $\Pi$ and using the skew symmetry of $\BH$, we obtain Figure \ref{fig:iterated-cap}.  
Similarly, by post-multiplying (vertically) the relation Figure \ref{fig:4-term-iterated} by $\amalg$ and using the skew symmetry of $\BH$, we obtain the relation depicted in Figure \ref{fig:iterated-cap-2}. 
\begin{figure}[h]
\begin{picture}(50, 70)(0,0)
{
\linethickness{1mm}
\put(10, 0){\line(0, 1){70}}
}

\qbezier(25, 65)(33, 80)(40, 65)

\put(10, 60){\uwave{\hspace{5mm}}}
\put(15, 47){\tiny$\vdots\, \ell$}
\put(10, 45){\uwave{\hspace{5mm}}}
\put(10, 25){\uwave{\hspace{5mm}}}

\put(25, 65){\line(0, -1){25}}
\put(40, 65){\line(0, -1){25}}

\qbezier(25, 40)(25, 40)(40, 30)
\qbezier(25, 30)(25, 30)(40, 40)

\put(25, 30){\line(0, -1){15}}
\put(40, 30){\line(0, -1){15}}
\qbezier(25, 15)(25, 15)(40, 0)
\qbezier(25, 0)(25, 0)(40, 15)
\put(45, 30){$+$}
\end{picture}
\begin{picture}(50, 70)(0,0)
{
\linethickness{1mm}
\put(10, 0){\line(0, 1){70}}
}

\qbezier(25, 60)(33, 75)(40, 60)

\put(10, 50){\uwave{\hspace{5mm}}}
\put(10, 25){\uwave{\hspace{5mm}}}
\put(15, 12){\tiny$\vdots\, \ell$}
\put(10, 10){\uwave{\hspace{5mm}}}


\put(25, 60){\line(0, -1){60}}
\put(40, 60){\line(0, -1){60}}

\put(25, 30){\line(0, -1){30}}
\put(40, 30){\line(0, -1){30}}
\put(45, 30){$+$}
\end{picture}
\begin{picture}(50, 70)(0,0)
{
\linethickness{1mm}
\put(10, 0){\line(0, 1){70}}
}

\qbezier(25, 60)(33, 75)(40, 60)

\put(10, 55){\uwave{\hspace{5mm}}}
\put(15, 38){\tiny$\vdots\, \ell$}
\put(10, 30){\uwave{\hspace{5mm}}}

\put(25, 60){\line(0, -1){30}}
\put(40, 60){\line(0, -1){30}}


\put(25, 30){\line(0, -1){15}}
\put(40, 30){\line(0, -1){15}}
\qbezier(25, 15)(25, 15)(40, 0)
\qbezier(25, 0)(25, 0)(40, 15)
\put(45, 30){$-$}
\end{picture}
\begin{picture}(90, 70)(-40,0)
\put(-30, 30){$(1-\delta)$}
{
\linethickness{1mm}
\put(10, 0){\line(0, 1){70}}
}

\qbezier(25, 60)(33, 75)(40, 60)

\put(10, 55){\uwave{\hspace{5mm}}}
\put(15, 38){\tiny$\vdots\, \ell$}
\put(10, 30){\uwave{\hspace{5mm}}}


\put(25, 60){\line(0, -1){60}}
\put(40, 60){\line(0, -1){60}}

\put(45, 30){$-$}
\end{picture}
\begin{picture}(50, 70)(0,0)
{
\linethickness{1mm}
\put(10, 0){\line(0, 1){70}}
}

\qbezier(25, 60)(33, 75)(40, 60)

\put(10, 60){\uwave{\hspace{5mm}}}
\put(15, 47){\tiny$\vdots\, \ell$}
\put(10, 45){\uwave{\hspace{5mm}}}

\qbezier(25, 40)(33, 25)(40, 40)
\put(25, 60){\line(0, -1){20}}
\put(40, 60){\line(0, -1){20}}


\qbezier(25, 20)(33, 35)(40, 20)
\put(25, 20){\line(0, -1){20}}
\put(40, 20){\line(0, -1){20}}
\put(45, 30){$=0$.}
\end{picture}
\caption{A relation in $\Hom_{\AB(\delta)}(2, 0)$}
\label{fig:iterated-cap}
\end{figure}  
\begin{figure}[h]
\begin{picture}(50, 70)(0,0)
\put(-5, 30){$-$}
{
\linethickness{1mm}
\put(10, 0){\line(0, 1){70}}
}
\put(10, 65){\uwave{\hspace{5mm}}}
\put(15, 52){\tiny$\vdots\, \ell$}
\put(10, 50){\uwave{\hspace{5mm}}}
\put(10, 25){\uwave{\hspace{5mm}}}

\put(25, 70){\line(0, -1){40}}
\put(40, 70){\line(0, -1){40}}


\put(25, 30){\line(0, -1){20}}
\put(40, 30){\line(0, -1){20}}

\qbezier(25, 10)(32,-5)(40, 10)

\put(45, 30){$-$}
\end{picture}
\begin{picture}(50, 70)(0,0)
{
\linethickness{1mm}
\put(10, 0){\line(0, 1){70}}
}

\put(10, 50){\uwave{\hspace{5mm}}}
\put(10, 28){\uwave{\hspace{5mm}}}
\put(15, 15){\tiny$\vdots\, \ell$}
\put(10, 13){\uwave{\hspace{5mm}}}

\qbezier(25, 55)(25, 55)(40, 70)
\qbezier(25, 70)(25, 70)(40, 55)

\put(25, 55){\line(0, -1){15}}
\put(40, 55){\line(0, -1){15}}

\put(25, 30){\line(0, -1){22}}
\put(40, 30){\line(0, -1){22}}
\qbezier(25, 30)(25, 30)(40, 40)
\qbezier(25, 40)(25, 40)(40, 30)

\qbezier(25, 8)(32,-7)(40, 8)

\put(45, 30){$+$}
\end{picture}
\begin{picture}(50, 70)(0,0)
{
\linethickness{1mm}
\put(10, 0){\line(0, 1){70}}
}

\put(10, 65){\uwave{\hspace{5mm}}}
\put(15, 50){\tiny$\vdots\, \ell$}
\put(10, 45){\uwave{\hspace{5mm}}}

\put(25, 70){\line(0, -1){40}}
\put(40, 70){\line(0, -1){40}}


\put(25, 30){\line(0, -1){20}}
\put(40, 30){\line(0, -1){20}}

\qbezier(25, 10)(32,-5)(40, 10)

\put(45, 30){$-$}
\end{picture}
\begin{picture}(50, 70)(0,0)
{
\linethickness{1mm}
\put(10, 0){\line(0, 1){70}}
}

\put(10, 35){\uwave{\hspace{5mm}}}
\put(15, 22){\tiny$\vdots\, \ell$}
\put(10, 20){\uwave{\hspace{5mm}}}

\qbezier(25, 55)(25, 55)(40, 70)
\qbezier(25, 70)(25, 70)(40, 55)

\put(25, 55){\line(0, -1){15}}
\put(40, 55){\line(0, -1){15}}

\put(25, 40){\line(0, -1){30}}
\put(40, 40){\line(0, -1){30}}

\qbezier(25, 10)(32,-5)(40, 10)

\put(45, 30){$+$}
\end{picture}
\begin{picture}(50, 70)(0,0)
{
\linethickness{1mm}
\put(10, 0){\line(0, 1){70}}
}

\put(10, 35){\uwave{\hspace{5mm}}}
\put(15, 22){\tiny$\vdots\, \ell$}
\put(10, 20){\uwave{\hspace{5mm}}}

\qbezier(25, 60)(33, 45)(40, 60)
\put(25, 60){\line(0, 1){10}}
\put(40, 60){\line(0, 1){10}}


\qbezier(25, 40)(33, 55)(40, 40)
\put(25, 40){\line(0, -1){30}}
\put(40, 40){\line(0, -1){30}}

\qbezier(25, 10)(32,-5)(40, 10)

\put(45, 30){$-$}
\end{picture}
\begin{picture}(50, 70)(0,0)
{
\linethickness{1mm}
\put(10, 0){\line(0, 1){70}}
}

\put(10, 65){\uwave{\hspace{5mm}}}
\put(15, 52){\tiny$\vdots\, \ell$}
\put(10, 50){\uwave{\hspace{5mm}}}

\qbezier(25, 40)(33, 25)(40, 40)
\put(25, 70){\line(0, -1){30}}
\put(40, 70){\line(0, -1){30}}


\qbezier(25, 20)(33, 35)(40, 20)
\put(25, 20){\line(0, -1){10}}
\put(40, 20){\line(0, -1){10}}

\qbezier(25, 10)(32,-5)(40, 10)

\put(45, 30){$=0$.}
\end{picture}
\caption{A relation in $\Hom_{\AB(\delta)}(0, 2)$}
\label{fig:iterated-cap-2}
\end{figure}  

To prove \eqref{eq:HT-H-1}, we post-multiply the relation Figure \ref{fig:iterated-cap} with $\X_0$, and then pull the bottom right end point to the top. We obtain 
\[
(\BH^{\ell+1})^T +  (\BH^\ell)^T \BH +  \BH^\ell - (1-\delta) (\BH^\ell)^T - Z_\ell\ot I =0, 
\]
which leads to equation \eqref{eq:HT-H-1}. 

Equation \eqref{eq:HT-H-2} can be proven similarly. 
We pre-multiply the relation by $\X_0$, and then pull the top right end point to the bottom. We obtain the relation 
\[
-(\BH^{\ell+1})^T - \BH (\BH^{\ell})^T + (\BH^{\ell})^T - \BH^{\ell}  + Z_\ell\ot I - \delta (\BH^{\ell})^T=0,
\]
which leads to equation \eqref{eq:HT-H-2}. 

Taking the difference between equations \eqref{eq:HT-H-1} and \eqref{eq:HT-H-2}, we see that $\BH$ commutes with all $(\BH^{\ell})^T$. This implies the final statement of the lemma. 
\end{proof}

\begin{remark} If the characteristic of $K$ is not $2$, then 
\be\label{eq:quadratic}
(\BH^2)^T = \BH^2 +(\delta-2)\BH; 
\ee
this equation may be represented pictorially as
\[
\begin{picture}(60, 45)(0,0)
{
\linethickness{1mm}
\put(10, 0){\line(0, 1){45}}
}

\put(10, 30){\uwave{\hspace{5mm}}}
\put(10, 15){\uwave{\hspace{5mm}}}

\put(25, 35){\line(0, -1){30}}

\qbezier(25, 35)(35, 45)(40, 0)
\qbezier(25, 5)(35, -5)(40, 40)


\put(50, 18){$-$}
\end{picture}
\begin{picture}(85, 45)(0,0)
{
\linethickness{1mm}
\put(10, 0){\line(0, 1){45}}
}

\put(10, 30){\uwave{\hspace{5mm}}}
\put(10, 15){\uwave{\hspace{5mm}}}


\put(25, 45){\line(0, -1){45}}
\put(35, 18){$= \ (\delta-2)$}
\end{picture}
\begin{picture}(50, 45)(0,0)
{
\linethickness{1mm}
\put(10, 0){\line(0, 1){45}}
}

\put(10, 25){\uwave{\hspace{5mm}}}

\put(25, 45){\line(0, -1){45}}
\put(30, 5){.}
\end{picture}
\]
This follows easily from Figure \ref{fig:deg0-relation}.
\end{remark}

\begin{corollary} \label{cor:HT} 
Retain notation in Lemma \ref{lem:HT2} .  
The following relations hold in  $\Hom_{\AB(\delta)}(1, 1)$ for all $\ell\ge 0$: 
\beq
(\BH^{\ell+1})^T =\sum_{i=1}^\ell \BG_i \Phi ^{\ell -i}  - \BH \Phi^\ell, \quad
(\BH^{\ell+1})^T =\sum_{i=1}^\ell  \Phi ^{\ell -i} \BG_i - \Phi^\ell  \BH, 
\eeq
which in particular imply 
\beq\label{eq:H-Z-comm}
\sum_{i=1}^{\ell-1}  [Z_i\ot I, \Phi^{\ell-i}]=0, \quad \forall \ell. 
\eeq
\end{corollary}
\begin{proof}
By iterating equation \eqref{eq:HT-H-1},  we obtain 
\[
\baln
(\BH^{\ell+1})^T 
&= \sum_{i=1}^\ell \BG_i \Phi ^{\ell -i} +\BH^T \Phi^\ell=\sum_{i=1}^\ell \BG_i \Phi ^{\ell -i}  - \BH \Phi^\ell, 
\ealn
\]
where the last equality is obtained by using skew symmetry of $\BH$. This proves the first relation. 
The second relation is a consequence of equation \eqref{eq:HT-H-2}. 
By taking the difference of the two relations, we obtain \eqref{eq:H-Z-comm}.
This completes the proof.
\end{proof}

\begin{theorem} \label{thm:Z-odd} 

Assume that $2\ne 0$ in $K$  (thus $Z_1=0$). 
Then 
\begin{enumerate}
\item all $Z_\ell$ are central in $\AB(\delta)$ in the following sense: for any $r, s$ and $\ell$,  
\[ 
(Z_\ell\ot I_s) \BD =\BD (Z_\ell\ot I_s), \quad \forall \BD\in\Hom_{\AB(\delta)}(r, s);
\] 
\item the elements $Z_{2j+1}$ for all $j\ge 1$ belong to the subalgebra generated by $1$ and $Z_{2\ell}$ with $\ell\ge 1$.
\end{enumerate}
\end{theorem}
\begin{proof}
To prove part (1), we need only show that $[\BH, Z_\ell\ot I]=0$ for all $\ell$. We will do this by induction on $\ell$  using \eqref{eq:H-Z-comm}. 
Let us first replace $\ell$ by $\ell+1$ and re-write \eqref{eq:H-Z-comm} as
\[
\sum_{i=1}^{\ell}  [\Phi^{\ell+1-i},  Z_i\ot I]=0, \quad \forall \ell.
\]
Recall that $Z_1=0$, and we have already shown that $Z_2$ (and hence $Z_3$) commutes with $\BH$. Thus they also commute with $\Phi= (1-\delta)\I_0\ot I-\BH$. Assume that all $Z_i$ for all $i<\ell$ commute with $\Phi$. Then the above relation immediately gives
$[\Phi,  Z_\ell\ot I]=0$, and hence $[\BH,  Z_\ell\ot I]=0$. 

Now we prove part (2). 
By recalling the expressions for $\Phi$ and $\BG_i$ in terms of $\BH$ and $Z_j$ from Lemma \ref{lem:HT2}, we can re-write  the first relation in Corollary \ref{cor:HT} as 
\beq\label{eq:diff-HT}
(\BH^{\ell+1})^T  = (1)^{\ell+1} \BH^{\ell +1}+\sum_{i=0}^\ell f_i(Z_2, Z_3, \dots, Z_\ell) \BH^i, 
\eeq
where the $f_i(Z_2, Z_3, \dots, Z_\ell)$ are linear combinations of $1, Z_2, Z_3, \dots, Z_\ell$.
Taking the tensor product of both sides of \eqref{eq:diff-HT} with $I$, we obtain a relation in $\Hom_{\AB(\delta)}(3, 3)$. Pre-multiplying the resulting relation by $\Pi$ and then post-multiplying by $\amalg$, and using $\Pi Z_{\ell+1}\left((\BH^{\ell+1})^T \ot I\right) \amalg = Z_{\ell+1}$,  we obtain 
\[
\left(1-(1)^{\ell+1}\right) Z_{\ell+1}= \sum_{i=0}^\ell f_i(Z_2, Z_3, \dots, Z_\ell) Z_i.
\]

Let $\ell=2j$, then this relation allows us to express $Z_{2j+1}$ in terms of $Z_k$ for $k\le 2j$. An induction on $j$, starting from $j=0$ with $Z_1=0$,  completes the proof of 
the theorem. 
%
\end{proof}

Note that by slightly extending the proof of Theorem \ref{thm:Z-odd}(2),  we also obtain the relation $
\sum_{i=0}^{2\ell+1} f_i(Z_2, Z_3, \dots, Z_{2\ell +1}) Z_i=0$ for all $\ell$.

\begin{lemma}   
The elements 
$
Z(k_1, k_2, \dots, k_p)=\Pi \BH^{k_1} \X_0\BH^{k_2} \X_0\dots  \X_0 \BH^{k_p}\amalg$,  
for $p\ge 1$ and $k_i\ge 1$, 
can be expressed, using composition, in terms of $\I_0$ and $Z_{\ell}$ for $\ell\ge 2$. 
\end{lemma}
\begin{proof} 
Retain the notation $\wh{\bf H}(r, s)$ for $\Hom_{\AB(\delta)}(r, s)$. 
Now $Z(k_1, k_2, \dots, k_p)$  belongs to  $F\wh{\bf H}(0, 0)_k$ with $k=\sum_i k_i$. Our proof is by induction on $k$. If $k=1$, the statement is clear. 
Note that $F\wh{\bf H}(0, 0)_k$ is spanned by diagrams of the type we are considering, viz. $\{Z(k_1,k_2, \dots ,k_p)\mid \sum_ik_i\leq k\}$.

If $k_1=1$, it follows the skew symmetry of $\BH$ that $Z(k_1, k_2, \dots, k_p)=-Z(k_1+k_2, k_3, \dots, k_p)$, and there is a similar relation if $k_p=1$. Thus we may assume that $k_1, k_p\ge 2$. 

We now describe two reductions, which apply respectively to the cases $k_2=1$ or $k_2>1$.
\begin{enumerate}
\item If $k_2=1$, by repeatedly applying the four-term relation modulo $F\wh{\bf H}(0, 0)_{k-1}$, we obtain 
\[
Z(k_1, k_2, \dots, k_p)=-Z(k_1+k_2+ k_3, k_4, \dots, k_p) \mod F\wh{\bf H}(0, 0)_{k-1}.
\] 
\item If $k_2\ge 2$, we write $Z(k_1, k_2, \dots, k_p) = \Pi \BH^{k_1} \X_0\BH\X_0 \X_0\BH^{k_2-1} \X_0\dots  \X_0 \BH^{k_p}\amalg$. Then applying the reduction (1), we obtain 
\[
\begin{aligned}
Z(k_1, k_2, \dots, k_p) &= - \Pi \BH^{k_1+1} \X_0\BH^{k_2-1} \X_0\dots  \X_0 \BH^{k_p}\amalg \mod F\wh{\bf H}(0, 0)_{k-1}\\
&=- Z(k_1+1, k_2-1, \dots, k_p) \mod F\wh{\bf H}(0, 0)_{k-1}. 
\end{aligned}
\]
If $k_2-1\ge 2$, repeat this process, eventually reaching 
\[
\begin{aligned}
Z(k_1, k_2, \dots, k_p) 
&=(-1)^{k_2-1} Z(k_1+k_2-1, 1, \dots, k_p) \mod F\wh{\bf H}(0, 0)_{k-1}, 
\end{aligned}
\]
and we are in the situation of (1). 
\end{enumerate}
The two reductions above enable us to reduce $Z(k_1, k_2, \dots, k_p)$ to $Z(k-1, 1)$ or $-Z(k-1, k) $ modulo $F\wh{\bf H}(0, 0)_{k-1}$. These in turn are equal to $\pm Z_k\mod F\wh{\bf H}(0, 0)_{k-1}$
(for some $k$) by skew symmetry of $\BH$. Thus induction on $k$ shows that 
we can express $Z(k_1, k_2, \dots, k_p)$ as a polynomial of $Z_\ell$ for $\ell\le k$ with the coefficient of $Z_k$ being $\pm 1$. 

This completes the proof of the lemma. 
\end{proof}

The following facts are now easily deduced. 
\begin{theorem}\label{thm:Hom01}
The endomorphism algebras $\Hom_{ \AB(\delta)}(0, 0)$ and $\Hom_{ \AB(\delta)}(1, 1)$    have    the following properties.
\begin{enumerate}
\item
The algebra $\Hom_{ \AB(\delta)}(0, 0)$ is commutative, and is generated by the elements $\Z_{2\ell}$ for $\ell\ge 1$. 
\item All $Z\in \Hom_{ \AB(\delta)}(0, 0)$ are central in $\AB(\delta)$ in the sense that for any $r, s$, 
\[ 
(Z\ot I_s) \BD =\BD (Z\ot I_s), \quad \forall   \BD\in\Hom_{\AB(\delta)}(r, s).
\] 
\item The algebra $\Hom_{ \AB(\delta)}(1, 1)$ is generated by $\BH$ and the elements $Z_{2\ell}\ot I$ for all $\ell$, 
and thus  is commutative. 
\item The elements $Z_{2\ell}$ for $\ell\ge 1$ are algebraically independent. 
\end{enumerate}
\end{theorem}
\begin{proof}
Denote by $Z(0, 0)$ the sugalgebra of $\Hom_{ \AB(\delta)}(0, 0)$ generated by the elements $\Z_{2\ell}$ for all $\ell$, which are central by Theorem \ref{thm:Z-odd}. 
Then all $Z_\ell\in Z(0, 0)$ by Theorem \ref{thm:Z-odd}(2). By the above lemma, 
all $Z(k_1, k_2, \dots, k_p)\in Z(0, 0)$, thus are central. Using this fact, we can easily see 
that $\Hom_{ \AB(\delta)}(1, 1)$ is generated by the elements $(\BH^\ell)^T$, $\BH^k$,  
and $Z(k_1, k_2, \dots, k_p)\ot I$ for all $k_1, \dots, k_p$ and  $k, \ell, p$. 
Corollary \ref{cor:HT} states that $(\BH^\ell)^\ell$ can all be expressed by the elements 
$\BH^k$  and $Z(k_1, k_2, \dots, k_p)\ot I$. This implies part (3). 

As the elements 
$\Pi (\BD\ot I)\amalg$ for $\BD\in\Hom_{ \AB(\delta)}(1, 1)$ span $\Hom_{ \AB(\delta)}(0, 0)$, we have  
$\Hom_{ \AB(\delta)}(0, 0)=Z(0, 0)$ by part (3). Now part (2) is clear. 

We will give a proof of part (4) at the end of Section \ref{sect:centre-construct} by using 
a special case of the functor constructed in Theorem \ref{thm:a-funct}.
\end{proof}

\subsection{A polar Temperley-Lieb category}
We now introduce an affine version of the Temperley-Lieb category (cf. \cite{GL03, ILZ}), which is a quotient category of $\AB(\delta)$.

\begin{definition}\label{def:0delta} 
We may define the affine Temperley-Lieb category $\ATL(\delta)$ with parameter $\delta$ as follows. 
The $K$-module of morphisms of $\ATL(\delta)$ is spanned by polar  
Temperley-Lieb diagrams which are analogues of affine 
Brauer diagrams without crossings, and which obey the relation \eqref{eq:TL-H-1} given below.
\end{definition}
This definition is valid without restriction on $\delta$. 

If $\delta$ is a unit, it is more convenient to view $\ATL(\delta)$ as a quotient of $\AB(\delta)$ obtained by 
imposing on the morphisms of $\AB(\delta)$ a relation of the form $\mathbb{D}\ot L=0$, where $L$ is depicted on the right side of Figure \ref{fig:TLQ} with $a$ being a unit in $K$.

\begin{figure}[h]
\begin{picture}(150, 35)(0,0)
\put(-35, 15){$0\ = \ $}
\qbezier(0, 0)(0, 0)(15, 35)
\qbezier(0, 35)(0, 35)(15, 0)
\put(-10,15){$a$}
\put(55,15){$b$}
\put(95,15){$+\;c$}


\put(35, 15){$+ $}


\put(65, 0){\line(0, 1){35}}
\put(75, 0){\line(0, 1){35}}



\qbezier(120, 35)(130, 10)(140, 35)
\qbezier(120, 0)(130, 25)(140, 0)

\end{picture} 
\caption{Temperley-Lieb as a quotient}
\label{fig:TLQ}
\end{figure}
\noindent

To obtain an interesting quotient category, we assume that the images of the quasi idempotents $I_2- X$ and $\cup\circ \cap$ 
(which are orthogonal to each other) are non-zero in the quotient algebra of $B_2(\delta)$ obeying the relation given by Figure \ref{fig:TLQ}.
 Multiplying that relation by $I_2- X$ and $\cup\circ \cap$ respectively, we obtain 
\[ 
a-b=0, \text{ and } a+b+\delta c=0, 
\]
which imply that
$
2 a + \delta c=0. 
$
Therefore, 
\[
\text{$\delta$ is a unit, and $c= - \frac{2a}{\delta}$}. 
\]
%
%
Thus the right side of Figure \ref{fig:TLQ} is a scalar multiple
of the element $\Theta$ defined in Figure \ref{fig:TL},  which is also a quasi idempotent in $B_2(\delta)$.

\begin{figure}[h]
\begin{picture}(150, 35)(0,0)
\put(-35, 15){$\Theta\ = \ $}
\qbezier(0, 0)(0, 0)(15, 35)
\qbezier(0, 35)(0, 35)(15, 0)

\put(35, 15){$+ $}


\put(65, 0){\line(0, 1){35}}
\put(75, 0){\line(0, 1){35}}
\put(90, 15){$-\ \frac{2}{\delta}$}



\qbezier(120, 35)(130, 10)(140, 35)
\qbezier(120, 0)(130, 25)(140, 0)

\end{picture} 
\caption{Temperley-Lieb condition}
\label{fig:TL}
\end{figure}

Assume that $\delta$ is a unit in $K$. 
Taking the quotient category of $\CB(\delta)$ by the $2$-ideal generated by $\Theta$ in $\Hom(\CB(\delta))$, we obtain the Temperley-Lieb category $\TL(\delta)$.  
This leads to the following definition of the affine Temperley-Lieb category $\ATL(\delta)$.
 
\begin{definition} \label{def:TL}
Assume that $\delta$ is a unit in $K$. Let $\langle \Theta\rangle$ be the $2$-ideal in $\Hom(\AB(\delta))$ generated by the elements $\BD\ot\Theta$ for all 
$\BD\in\Hom(\AB(\delta))$.  The quotient category $\ATL(\delta)=\AB(\delta)/\langle \Theta\rangle$,  
is the {\em diagrammatic affine Temperley-Lieb category} or {\em polar Temperley-Lieb category}.
\end{definition}


As noted above, the $K$-module of morphisms of $\ATL(\delta)$ is spanned by diagrams analogous to affine 
Brauer diagrams but without any crossings. Such diagrams will be called {\em affine, or polar,   
Temperley-Lieb diagrams}.

Henceforth we assume that $\delta$ is a unit in $K$. 

We begin by noting that the composition of diagrams now obeys new rules. In particular, we have the following result. 

\begin{lem} \label{lem:TL-H} The four-term relation reduces to the following relation in $\ATL(\delta)$,
\beq\label{eq:TL-H}
\begin{array}{c}
\begin{picture}(100, 40)(0, 0)
{
\linethickness{1mm}
\put(15, 0){\line(0, 1){40}}
}
\put(15, 28){\uwave{\hspace{7mm}}}
\put(15, 12){\uwave{\hspace{7mm}}}

\put(35, 0){\line(0, 1){40}}

\put(45, 16){ $+ \ \frac{\delta-2}2$ }
\end{picture} 
\begin{picture}(100, 40)(25, 0)
{
\linethickness{1mm}
\put(15, 0){\line(0, 1){40}}
}
\put(15, 20){\uwave{\hspace{7mm}}}

\put(35, 0){\line(0, 1){40}}

\put(45, 16){ $ - \ \frac{Z_2}\delta$ }
\end{picture} 
\begin{picture}(60, 40)(55, 0)
{
\linethickness{1mm}
\put(15, 0){\line(0, 1){40}}
}

\put(35, 0){\line(0, 1){40}}

\put(45, 16){ $=\ 0$ }
\put(75, 5){,}
\end{picture} 
\end{array}
\eeq
where the third term is a pictorial representation of $-\frac{1}\delta  Z_2\ot I$. 
\end{lem}
\begin{proof}
Eliminating the crossings of thin arcs in all the diagrams in the four-term relation given in Figure \ref{fig:4-term} by using the relation $\Theta=0$, we obtain the relation 
\beq \label{eq:TL-H-1}
\begin{array}{c}
\begin{picture}(75, 40)(30, 0)
{
\linethickness{1mm}
\put(35, 0){\line(0, 1){40}}
}
\put(35, 35){\uwave{\hspace{5mm}}}
\put(35, 25){\uwave{\hspace{5mm}}}

\put(50, 40){\line(0, -1){25}}
\put(65, 40){\line(0, -1){25}}

\qbezier(50, 15)(57, 5)(65, 15)

\qbezier(50, 0)(57, 10)(65, 0)

\put(80, 15){$-$}
\end{picture}
\begin{picture}(85, 40)(30, 0)
{
\linethickness{1mm}
\put(35, 0){\line(0, 1){40}}
}
\put(35, 18){\uwave{\hspace{5mm}}}
\put(35, 8){\uwave{\hspace{5mm}}}

\put(50, 0){\line(0, 1){25}}
\put(65, 0){\line(0, 1){25}}

\qbezier(50, 25)(57, 35)(65, 25)

\qbezier(50, 40)(57, 30)(65, 40)



\put(80, 15){$-$}
\end{picture}
\begin{picture}(85, 40)(30, 0)
\put(12, 15){$\frac{2-\delta}{2}$}
{
\linethickness{1mm}
\put(35, 0){\line(0, 1){40}}
}
\put(35, 30){\uwave{\hspace{5mm}}}

\put(50, 40){\line(0, -1){25}}
\put(65, 40){\line(0, -1){25}}

\qbezier(50, 15)(57, 5)(65, 15)

\qbezier(50, 0)(57, 10)(65, 0)



\put(80, 15){$+$}
\end{picture}
\begin{picture}(80, 40)(30, 0)
\put(12, 15){$\frac{2-\delta}{2}$}
{
\linethickness{1mm}
\put(35, 0){\line(0, 1){40}}
}
\put(35, 15){\uwave{\hspace{5mm}}}

\put(50, 0){\line(0, 1){25}}
\put(65, 0){\line(0, 1){25}}

\qbezier(50, 25)(57, 35)(65, 25)

\qbezier(50, 40)(57, 30)(65, 40)

\put(80, 15){$=\ 0$}
\put(100, 5){.}
\end{picture}
\end{array}
\eeq

Post-multiplying \eqref{eq:TL-H-1} by $\cup$, we obtain a relation involving three $(0, 2)$-diagrams. 
Bringing the right end point of the thin arc in each of the diagram down to the bottom, we arrive at \eqref{eq:TL-H}.  Finally, given \eqref{eq:TL-H}, 
the relation \eqref{eq:TL-H-1} holds identically in $\ATL(\delta)$.
\end{proof}

\begin{remark}
Note that the above shows that if $\delta-2=0$, then $\BH^2=\frac{1}{\delta}Z_2\ot I$, and many of our results become degenerate. We therefore
assume henceforth that $\delta\neq 0, 2$.
\end{remark}

The following results are immediate from Lemma \ref{lem:TL-H}  and Lemma \ref{lem:central}. 
\begin{cor} \label{lem:TL-4-term} The affine Temperley-Lieb category has the following properties.
\begin{enumerate}[(i)]
\item $(Z_2\ot I) \BH = \BH (Z_2\ot I)$, and hence  ${\mathbb D}(Z_2\ot I_r)= (Z_2\ot I_s){\mathbb D}$ for any affine Temperley-Lieb $(r, s)$-diagram ${\mathbb D}$.

\item $\Hom_{\ATL(\delta)}(0, 0)=K[Z_2]$, the polynomial algebra in $Z_2$;  
 
\item $\Hom_{\ATL(\delta)}(1, 1)=(K[Z_2]\ot I) \oplus (K[Z_2]\ot I)\BH$, which is  a commutative $K$-algebra.
\end{enumerate}
\end{cor}
\begin{proof}
Part (1) is a consequence of Lemma \ref{lem:central}(2). The rest of the lemma follows from \eqref{eq:TL-H} and part (1). 
\end{proof}

To simplify notation, we shall write $f(Z_2) {\mathbb D}$  for both ${\mathbb D}(f(Z_2)\ot I_r)$ and $(f(Z_2)\ot I_s){\mathbb D}$ for any $f(Z_2)\in K[Z_2]$ and affine Temperley-Lieb $(r, s)$-diagram ${\mathbb D}$. This should not cause any confusion because of parts (i) and (ii) of Corollary \ref{lem:TL-4-term}.

Now $\Hom(\ATL(\delta))$ may be thought of as a $K[Z_2]$-module,  which is clearly free.
Note in particular that for any $N$, an affine Temperley-Lieb $(0, 2N)$-diagram is the product of some power of $Z_2$ and a diagram of the form shown in Figure \ref{fig:ATL2N},   
\begin{figure}[h]
\begin{picture}(280, 120)(10,-60)
{
\linethickness{1mm}
\put(10, -60){\line(0, 1){110}}
}

\put(10, 16){\uwave{\hspace{15mm}}}
\put(55, 10){\line(0, 1){40}}
\put(100, 10){\line(0, 1){40}}

\qbezier(55, 10)(77, -10)(100, 10)

\put(17, 25){\line(1, 0){30}}
\put(17, 40){\line(1, 0){30}}
\put(17, 25){\line(0, 1){15}}
\put(47, 25){\line(0, 1){15}}
\put(21, 40){\line(0, 1){10}}
\put(42, 40){\line(0, 1){10}}
\put(26, 28){$A_1$}
\put(26, 44){$\dots$}
\put(28, 48){\tiny$2r_1$}

\put(62, 25){\line(1, 0){30}}
\put(62, 40){\line(1, 0){30}}
\put(62, 25){\line(0, 1){15}}
\put(92, 25){\line(0, 1){15}}
\put(65, 40){\line(0, 1){10}}
\put(88, 40){\line(0, 1){10}}
\put(70, 28){$A_2$}
\put(68, 44){$\dots$}
\put(70, 48){\tiny$2r_2$}

\put(107, 25){\line(1, 0){30}}
\put(107, 40){\line(1, 0){30}}
\put(107, 25){\line(0, 1){15}}
\put(137, 25){\line(0, 1){15}}
\put(110, 40){\line(0, 1){10}}
\put(133, 40){\line(0, 1){10}}
\put(115, 28){$A_3$}
\put(113, 44){$\dots$}
\put(115, 48){\tiny$2r_3$}

\put(140, 33){$\dots$}
\put(157, 25){\line(1, 0){30}}
\put(157, 40){\line(1, 0){30}}
\put(157, 25){\line(0, 1){15}}
\put(187, 25){\line(0, 1){15}}
\put(159, 40){\line(0, 1){10}}
\put(185, 40){\line(0, 1){10}}
\put(159, 28){$A_{2t-1}$}
\put(163, 44){$\dots$}
\put(160, 48){\tiny$2r_{\scriptscriptstyle{2t-1}}$}

\put(207, 25){\line(1, 0){30}}
\put(207, 40){\line(1, 0){30}}
\put(207, 25){\line(0, 1){15}}
\put(237, 25){\line(0, 1){15}}
\put(210, 40){\line(0, 1){10}}
\put(233, 40){\line(0, 1){10}}
\put(215, 28){$A_{2t}$}
\put(213, 44){$\dots$}
\put(215, 48){\tiny$2r_{2t}$}

\put(257, 25){\line(1, 0){30}}
\put(257, 40){\line(1, 0){30}}
\put(257, 25){\line(0, 1){15}}
\put(287, 25){\line(0, 1){15}}
\put(259, 40){\line(0, 1){10}}
\put(285, 40){\line(0, 1){10}}
\put(259, 28){$A_{2t+1}$}
\put(263, 44){$\dots$}
\put(260, 48){\tiny$2r_{2t+1}$}


\put(20, -10){$\begin{array}{c c}
\vdots \vspace{-2.5mm}\\
\vdots
\end{array} t $}

\put(10, -28){\uwave{\hspace{65mm}}}
\put(195, -40){\line(0, 1){90}}
\put(245, -40){\line(0, 1){90}}

\qbezier(195, -40)(220, -60)(245, -40)
\end{picture}
\caption{Standard affine Temperley-Lieb diagram}
\label{fig:ATL2N}
\end{figure}
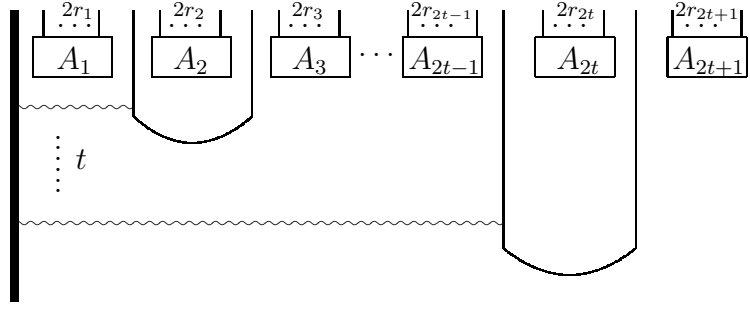
where there are $t$ connectors with $0\le t\le N$, and each $A_j$ is a usual Brauer $(0,2r_j)$-diagram with no intersections,
 such that $\sum_j r_j=N- t$.  We call such an affine Temperley-Lieb diagram {\em standard}.

We have the following result.

\begin{thm}  \label{thm:ATL-dim}
For any non-negative integers $N$ and $r$ such that $r\le 2N$, the rank of $\Hom_{\ATL(\delta)}(r, 2N-r)$ over $K[Z_2]$ is equal to   
$\begin{pmatrix}2N \\ N \end{pmatrix}$.
\end{thm}

\begin{proof}  
We adapt some key ideas of the proof of \cite[Lemma 3.24]{ILZ} to the present context. 
Let $\wh{K}$ be the fraction field of $K[Z_2]$. Let $W(r, s)=\wh{K}\ot_{K[Z_2]}\Hom_{\ATL(\delta)}(r, s)$ for all $r, s$, and denote $W(2N)=W(0, 2N)$. 
The standard affine Temperley-Lieb $(0, 2N)$-diagrams form a basis of $W(0, 2N)$. 
For each $i\le N$, the standard $(0, 2N)$-diagrams with at most $i\le N$ connectors span a
$\wh{K}$-subspace $F_iW(2N)$ of $W(2N)$, and thus we have a
filtration of $W(2N)$ given as follows.
\beq\label{eq:TL-filt}
W(2N)=F_N W(2N)\supset F_{N-1}W(2N)  \supset \dots \supset  F_0 W(2N)\supset 0.
\eeq

We note that for any $n$, the subspace $E_n^0$ of $W(n, n)$  spanned by affine Temperley -Lieb $(n, n)$-diagrams without connectors forms a $\wh{K}$-algebra isomorphic to the Temperley-Lieb algebra ${\rm TL}_n(\delta)$ of degree $n$ with parameter $\delta$; in particular, $E^0_{2N}={\rm TL}_{2N}(\delta)$. 

Now ${\rm TL}_{2N}(\delta)$ naturally acts on $W(2N)$ by composition of morphisms in $\ATL(\delta)$, and each $F_iW(2N)$ forms a ${\rm TL}_{2N}(\delta)$-submodule. Thus \eqref{eq:TL-filt} is a filtration of ${\rm TL}_{2N}(\delta)$-modules. It follows from \eqref{eq:TL-H} and skew symmetry of $\BH$ that $W_{2i}(2N):=\frac{F_i W(2N)}{F_{i-1} W(2N)}$ is a simple ${\rm TL}_{2N}(q)$-module for each $i$. Recall that the dimension of this simple ${\rm TL}_{2N}(\delta)$-module is given by $\dim_{\wh{K}}(W_{2i}(2N))= \begin{pmatrix}2N \\ N-i \end{pmatrix}-\begin{pmatrix}2N \\ N-i-1 \end{pmatrix}$. Therefore,  
$\dim_{\wh{K}}(W(2N))=\sum_i \dim_{\wh{K}}(W_{2i}(2N)) = \begin{pmatrix}2N \\ N \end{pmatrix}$.

However $\dim_{\wh{K}}(W(r, 2N-r))=\dim_{\wh{K}}(W(2N))$ for all $r$ with $0\leq r\leq 2N$, since pulling of strings shows that $W(a, b)$ and $W(0, a+b)$ 
are isomorphic as vector spaces over $\wh{K}$ for all $a, b$. This completes the proof of the Theorem.
\end{proof}

Since $Z_2$ is central in $\ATL(\delta)$, we may take a quotient of the category by setting $Z_2$ to some fixed parameter. 
\begin{definition}  \label{def:TLB}
Fix $\lambda\in K$. Denote by $\TLBC(\delta, \lambda)$ the quotient category of $\ATL(\delta)$ obtained by setting 
\[
\frac{Z_2}{\delta}=-\lambda\left(\frac{\delta-2}{2}-\lambda\right), 
\]
and call it a {\em Temperley-Lieb category of type $B$}. 
\end{definition}
We continue to denote the image of $\BH$ in $\TLBC(\delta, \lambda)$ by the same symbol $\BH$. We have the following result. 
\begin{lemma}
The following relation holds in $\TLBC(\delta, \lambda)$.
\[
(\BH + \lambda)\left(\BH + \frac{\delta-2}{2}-\lambda\right) =0.
\]
\end{lemma}
\begin{proof}
This follows from Lemma \ref{lem:TL-H}, which asserts that $\BH$ satisfies the quadratic relation $\BH^2+(\frac{\delta-2}{2})\BH-\frac{Z_2}{\delta}=0$ over $\wh K$. 
\end{proof}

There is a close relationship between the category $\TLBC(-2, \lambda)$ and a category of infinite dimensional $\fsp_2$-representations, which will be studied in detail in Section \ref{sect:sp2}.

\section{Endomorphism algebras of objects in $\AB(\delta)$}\label{sect:chord}

We next investigate the endomorphism algebras of objects of $\AB(\delta)$, which turn out to be certain quotient algebras of the algebra of chord diagrams (see, e.g., \cite[\S XX.3]{K}, \cite{LZ06}). For any fixed integer $r\ge 2$, 
the algebra of chord diagrams, denoted by $T_r$, is a unital associative $K$-algebra with generating set $\{t_{ij}\mid 1\leq i<j\leq r\}$ and defining relations
\be\label{eq:deftr}
 [t_{ij},t_{k\ell}]=0,\;\;\;[t_{ik}+t_{i\ell},t_{k\ell}]=0,\;\;\; [t_{ij},t_{ik}+t_{jk}]=0,
\ee
where $i,j,k$ and $\ell$ are pairwise distinct indices,  and  $[\  ,\  ]$ denotes the usual commutator defined by $[X, Y ] = XY -
Y X$ for all $X, Y$.


%
%
\subsection{Chord diagrams and the Brauer algebra}

Let us first show how the Brauer algebra $B_r(\delta)$ appears as a quotient of the algebra $T_r$ of chord diagrams. 
Although this is already known from the description in \cite{LZ06} of the endomorphism 
algebras of tensor modules for the orthogonal and symplectic groups using the algebra $T_r$,  what is new here is that we provide a direct link between $T_r$ and $B_r(\delta)$. 
The diagrammatic methods developed below for this purpose may be interesting in their own right. 

Let $s_i, e_i$ with $i=1, 2, \dots, r-1$ be the standard generators of the $r$-string Brauer algebra $B_r(\delta)$ with parameter $\delta$. They are depicted pictorially in Figure \ref{fig:siei}.  

\begin{figure}[h]
\setlength{\unitlength}{0.3mm}
\begin{picture}(150, 70)(-20,-5)
\put(-20, 28){$s_i\  =$}
\put(20, 0){\line(0, 1){60}}
\put(25, 30){...}
\put(40, 0){\line(0, 1){60}}

\qbezier(60, 0)(60, 0)(80, 60)
\qbezier(80, 0)(80, 0)(60, 60)

\put(100, 0){\line(0, 1){60}}
\put(120, 0){\line(0, 1){60}}
\put(105, 30){...}
\put(56, -10){\small$i$}
\put(72, -10){\small{$i$+1}}
\put(130, 0){, }
\end{picture} 
\hspace{.5cm}
\begin{picture}(150, 70)(-20,-5)
\put(-20, 28){$e_i\  =$}
\put(20, 0){\line(0, 1){60}}
\put(25, 30){...}
\put(40, 0){\line(0, 1){60}}

\qbezier(60, 60)(70, 10)(80, 60)
\qbezier(60, 0)(70, 50)(80, 0)

\put(100, 0){\line(0, 1){60}}
\put(120, 0){\line(0, 1){60}}
\put(105, 30){...}
\put(56, -10){\small$i$}
\put(72, -10){\small{$i$+1}}
\put(130, 0){. }
\end{picture}
\caption{Generators of $B_r(\delta)$}
\label{fig:siei}
\end{figure}

Define the following elements of $B_r(\delta)$, 
\[
H_i = s_i - e_i, \quad i=1,2, \dots, r-1. 
\]
We generalise the graphical representation of the element $H\in \Hom_{\CB(\delta)}(2, 2)$ to represent $H_i$ pictorially as shown in Figure \ref{fig:Hi},  
\begin{figure}[h]
\begin{picture}(150, 60)(-20,0)
\put(-20, 28){$H_i\  =$}
\put(20, 0){\line(0, 1){60}}
\put(25, 30){...}
\put(125,30){,}
\put(40, 0){\line(0, 1){60}}

\put(60, 0){\line(0, 1){60}}
\put(60, 30){\line(1, 0){20}}
\put(80, 0){\line(0, 1){60}}


\put(100, 0){\line(0, 1){60}}
\put(120, 0){\line(0, 1){60}}
\put(105, 30){...}
\put(56, -10){\small$i$}
\put(72, -10){\small{$i$+1}}
\end{picture} 
\caption{Picture for $H_i=s_i-e_i$}
\label{fig:Hi}
\end{figure}
%
%
where the right hand side is the linear combination of two Brauer diagrams shown below
\[
\setlength{\unitlength}{0.3mm}
\begin{picture}(140, 70)(0,0)
\put(20, 0){\line(0, 1){60}}
\put(25, 30){...}
\put(40, 0){\line(0, 1){60}}

\put(60, 0){\line(0, 1){60}}
\put(60, 30){\line(1, 0){20}}
\put(80, 0){\line(0, 1){60}}


\put(100, 0){\line(0, 1){60}}
\put(120, 0){\line(0, 1){60}}
\put(105, 30){...}
\put(56, -10){\small$i$}
\put(72, -10){\small{$i$+1}}
\end{picture} 
\begin{picture}(140, 70)(0,0)
\put(-5, 28){$=$}
\put(20, 0){\line(0, 1){60}}
\put(25, 30){...}
\put(40, 0){\line(0, 1){60}}

\qbezier(60, 0)(60, 0)(80, 60)
\qbezier(80, 0)(80, 0)(60, 60)

\put(100, 0){\line(0, 1){60}}
\put(120, 0){\line(0, 1){60}}
\put(105, 30){...}
\put(56, -10){\small$i$}
\put(72, -10){\small{$i$+1}}
\end{picture} 
\begin{picture}(140, 70)(0,0)
\put(-5, 28){$-$}
\put(20, 0){\line(0, 1){60}}
\put(25, 30){...}
\put(40, 0){\line(0, 1){60}}

\qbezier(60, 60)(70, 10)(80, 60)
\qbezier(60, 0)(70, 50)(80, 0)

\put(100, 0){\line(0, 1){60}}
\put(120, 0){\line(0, 1){60}}
\put(105, 30){...}
\put(56, -10){\small$i$}
\put(72, -10){\small{$i$+1}}
\put(130, 0){. }
\end{picture}
\]

\vspace{.3cm}

The elements $H_i$ satisfy some interesting relations. 
\begin{lem}\label{lem:relat-1}
The following relations hold in $B_r(\delta)$. 
\[
\setlength{\unitlength}{0.3mm}
\begin{picture}(140, 70)(0,0)
\put(0, 0){\line(0, 1){60}}
\put(-20, 0){\line(0, 1){60}}
\put(-15, 30){...}

\qbezier(40, 60)(50, 40)(60, 60)
\qbezier(40, 35)(50, 55)(60, 35)
\qbezier(40, 25)(50, 5)(60, 25)
\qbezier(40, 0)(50, 20)(60, 0)

\qbezier(40, 35)(30, 5)(20, 0)
\qbezier(40, 25)(30, 55)(20, 60)


\put(60, 25){\line(0, 1){10}}

\put(60, 30){\line(1, 0){20}}
\put(80, 0){\line(0, 1){60}}


\put(100, 0){\line(0, 1){60}}
\put(120, 0){\line(0, 1){60}}
\put(105, 30){...}
\put(135, 30){$= - $}
\end{picture} 
\setlength{\unitlength}{0.3mm}
\begin{picture}(140, 70)(-45,0)
\put(-20, 0){\line(0, 1){60}}
\put(0, 0){\line(0, 1){60}}
\put(-15, 30){...}

\qbezier(40, 60)(50, 40)(60, 60)
\qbezier(60, 35)(60, 40)(20, 60)
\qbezier(60, 25)(60, 20)(20, 0)
\qbezier(40, 0)(50, 20)(60, 0)

\put(60, 25){\line(0, 1){10}}

\put(60, 30){\line(1, 0){20}}
\put(80, 0){\line(0, 1){60}}


\put(100, 0){\line(0, 1){60}}
\put(120, 0){\line(0, 1){60}}
\put(105, 30){...}
\put(125, 5){;}
\end{picture} 
\]
\[
\setlength{\unitlength}{0.3mm}
\begin{picture}(160, 70)(20,0)
\put(20, 0){\line(0, 1){60}}
\put(25, 30){...}
\put(40, 0){\line(0, 1){60}}

\put(60, 0){\line(0, 1){60}}
\put(60, 30){\line(1, 0){20}}

\qbezier(80, 60)(90, 40)(100, 60)
\qbezier(80, 35)(90, 55)(100, 35)
\qbezier(80, 25)(90, 5)(100, 25)
\qbezier(80, 0)(90, 20)(100, 0)
\put(80, 25){\line(0, 1){10}}


\qbezier(100, 35)(110, 5)(120, 0)
\qbezier(100, 25)(110, 55)(120, 60)

\put(140, 0){\line(0, 1){60}}
\put(145, 30){...}
\put(160, 0){\line(0, 1){60}}
\put(175, 30){$=\ - $}
\end{picture} 
\begin{picture}(160, 70)(-10,0)
\put(20, 0){\line(0, 1){60}}
\put(25, 30){...}
\put(40, 0){\line(0, 1){60}}

\put(60, 0){\line(0, 1){60}}
\put(60, 30){\line(1, 0){20}}

\qbezier(80, 60)(90, 40)(100, 60)
\qbezier(80, 0)(90, 20)(100, 0)
\put(80, 25){\line(0, 1){10}}

\qbezier(80, 35)(80, 40)(120, 60)
\qbezier(80, 25)(80, 20)(120, 0)

\put(140, 0){\line(0, 1){60}}
\put(145, 30){...}
\put(160, 0){\line(0, 1){60}}
\put(165, 5){;}
\end{picture} 
\]
\[
\setlength{\unitlength}{0.3mm}
\begin{picture}(140, 70)(20,0)
\put(0, 0){\line(0, 1){60}}
\put(20, 0){\line(0, 1){60}}
\put(5, 30){...}

\qbezier(40, 60)(50, 40)(60, 60)
\qbezier(40, 35)(50, 55)(60, 35)
\qbezier(40, 25)(50, 5)(60, 25)
\qbezier(40, 0)(50, 20)(60, 0)
\put(40, 25){\line(0, 1){10}}
\put(60, 25){\line(0, 1){10}}

\put(60, 30){\line(1, 0){20}}
\put(80, 0){\line(0, 1){60}}


\put(100, 0){\line(0, 1){60}}
\put(120, 0){\line(0, 1){60}}
\put(105, 30){...}
\put(135, 30){$=$}
\end{picture} 
\begin{picture}(140, 70)(20,0)
\put(20, 0){\line(0, 1){60}}
\put(25, 30){...}
\put(40, 0){\line(0, 1){60}}

\put(60, 0){\line(0, 1){60}}
\put(60, 30){\line(1, 0){20}}

\qbezier(80, 60)(90, 40)(100, 60)
\qbezier(80, 35)(90, 55)(100, 35)
\qbezier(80, 25)(90, 5)(100, 25)
\qbezier(80, 0)(90, 20)(100, 0)
\put(80, 25){\line(0, 1){10}}
\put(100, 25){\line(0, 1){10}}

\put(120, 0){\line(0, 1){60}}
\put(125, 30){...}
\put(140, 0){\line(0, 1){60}}
\put(155, 30){$=\  0$.}
\end{picture} 
\]
\end{lem}

These are easy consequences of skew symmetry of $H$ depicted by Figure \ref{fig:Brauer-H-skew}. 
The first and second relations in Lemma \ref{lem:relat-1} have the following re-interpretations. 
\begin{lem}\label{lem:relat-2}
The following relations hold in $B_r(\delta)$.
\[
\setlength{\unitlength}{0.3mm}
\begin{picture}(140, 70)(0,0)
\put(-20, 0){\line(0, 1){60}}
\put(0, 0){\line(0, 1){60}}
\put(20, 0){\line(0, 1){60}}
\put(-15, 30){...}

\qbezier(40, 60)(40, 60)(60, 45)
\qbezier(40, 45)(40, 45)(60, 60)
\qbezier(40, 25)(50, 5)(60, 25)
\qbezier(40, 0)(50, 20)(60, 0)
\put(40, 25){\line(0, 1){20}}
\put(60, 25){\line(0, 1){20}}

\put(20, 35){\line(1, 0){20}}
\put(80, 0){\line(0, 1){60}}


\put(100, 0){\line(0, 1){60}}
\put(85, 30){...}
\put(115, 30){$= \ - $}
\end{picture} 
\begin{picture}(140, 70)(-35,0)
\put(-20, 0){\line(0, 1){60}}
\put(0, 0){\line(0, 1){60}}
\put(20, 0){\line(0, 1){60}}
\put(-15, 30){...}

\qbezier(40, 25)(50, 5)(60, 25)
\qbezier(40, 0)(50, 20)(60, 0)
\put(40, 25){\line(0, 1){35}}
\put(60, 25){\line(0, 1){35}}

\put(20, 35){\line(1, 0){20}}
\put(80, 0){\line(0, 1){60}}


\put(100, 0){\line(0, 1){60}}
\put(85, 30){...}
\put(105, 5){;}
\end{picture} 
\]
\[
\setlength{\unitlength}{0.3mm}
\begin{picture}(140, 70)(0,0)
\put(-20, 0){\line(0, 1){60}}
\put(0, 0){\line(0, 1){60}}
\put(20, 0){\line(0, 1){60}}
\put(-15, 30){...}

\qbezier(40, 60)(50, 40)(60, 60)
\qbezier(40, 35)(50, 55)(60, 35)

\put(20, 27){\line(1, 0){20}}
\put(80, 0){\line(0, 1){60}}

\put(40, 35){\line(0, -1){20}}
\put(60, 35){\line(0, -1){20}}

\qbezier(40, 15)(40, 15)(60, 0)
\qbezier(40, 0)(40, 0)(60, 15)


\put(100, 0){\line(0, 1){60}}
\put(85, 30){...}
\put(115, 30){$= \ - $}
\end{picture} 
\begin{picture}(140, 70)(-35,0)
\put(-20, 0){\line(0, 1){60}}
\put(0, 0){\line(0, 1){60}}
\put(20, 0){\line(0, 1){60}}
\put(-15, 30){...}

\qbezier(40, 60)(50, 40)(60, 60)
\qbezier(40, 35)(50, 55)(60, 35)

\put(20, 27){\line(1, 0){20}}
\put(80, 0){\line(0, 1){60}}

\put(40, 35){\line(0, -1){35}}
\put(60, 35){\line(0, -1){35}}


\put(100, 0){\line(0, 1){60}}
\put(85, 30){...}

\put(105, 5){;}
\end{picture} 
\]
\[
\setlength{\unitlength}{0.3mm}
\begin{picture}(140, 70)(20,0)
\put(0, 0){\line(0, 1){60}}
\put(20, 0){\line(0, 1){60}}
\put(5, 30){...}

\qbezier(40, 60)(40, 60)(60, 45)
\qbezier(40, 45)(40, 45)(60, 60)
\qbezier(40, 25)(50, 5)(60, 25)
\qbezier(40, 0)(50, 20)(60, 0)
\put(40, 25){\line(0, 1){20}}
\put(60, 25){\line(0, 1){20}}

\put(60, 35){\line(1, 0){20}}
\put(80, 0){\line(0, 1){60}}


\put(100, 0){\line(0, 1){60}}
\put(120, 0){\line(0, 1){60}}
\put(105, 30){...}
\put(135, 30){$= \ - $}
\end{picture} 
\begin{picture}(140, 70)(-15,0)
\put(0, 0){\line(0, 1){60}}
\put(20, 0){\line(0, 1){60}}
\put(5, 30){...}

\qbezier(40, 25)(50, 5)(60, 25)
\qbezier(40, 0)(50, 20)(60, 0)
\put(40, 25){\line(0, 1){35}}
\put(60, 25){\line(0, 1){35}}

\put(60, 35){\line(1, 0){20}}
\put(80, 0){\line(0, 1){60}}


\put(100, 0){\line(0, 1){60}}
\put(120, 0){\line(0, 1){60}}
\put(105, 30){...}
\put(125, 5){;}
\end{picture} 
\]
\[
\setlength{\unitlength}{0.3mm}
\begin{picture}(140, 70)(20,0)
\put(0, 0){\line(0, 1){60}}
\put(20, 0){\line(0, 1){60}}
\put(5, 30){...}

\qbezier(40, 60)(50, 40)(60, 60)
\qbezier(40, 35)(50, 55)(60, 35)

\put(60, 27){\line(1, 0){20}}
\put(80, 0){\line(0, 1){60}}

\put(40, 35){\line(0, -1){20}}
\put(60, 35){\line(0, -1){20}}

\qbezier(40, 15)(40, 15)(60, 0)
\qbezier(40, 0)(40, 0)(60, 15)


\put(100, 0){\line(0, 1){60}}
\put(120, 0){\line(0, 1){60}}
\put(105, 30){...}
\put(135, 30){$= \ - $}
\end{picture} 
\begin{picture}(140, 70)(-15,0)
\put(0, 0){\line(0, 1){60}}
\put(20, 0){\line(0, 1){60}}
\put(5, 30){...}

\qbezier(40, 60)(50, 40)(60, 60)
\qbezier(40, 35)(50, 55)(60, 35)

\put(60, 27){\line(1, 0){20}}
\put(80, 0){\line(0, 1){60}}

\put(40, 35){\line(0, -1){35}}
\put(60, 35){\line(0, -1){35}}


\put(100, 0){\line(0, 1){60}}
\put(120, 0){\line(0, 1){60}}
\put(105, 30){...}

\put(125, 5){.}
\end{picture} 
\]
\end{lem}

For any $1\le i<j\le r$,  let $X_{i j}= s_{j-1} s_{j-2}\dots s_{i+1}$, and $X_{j i} = s_{i+1} s_{i+2}\dots s_{j-1}$; evidently $X_{ji}=X_{ij}\inv$.
These are depicted diagramatically in Figure \ref{fig:Xij}. 
Note that the $X_{ij}$ may be multiplied (generally in many ways) to produce any permutation $w\in\Sym_r\subset\CB_r(\delta)$.

Define the following elements of $B_r(\delta)$.
\beq
H_{i j} = X_{i+1, j} H_i X_{j, i+1}, \quad 1\le i<j\le r.
\eeq
The elements $H_{ij}$ are depicted in terms of diagrams in Figure \ref{fig:Hij}.

\begin{figure}[h]
\setlength{\unitlength}{0.3mm}
\begin{picture}(160, 70)(-20,0)
\put(-20, 28){$X_{i j}=$}
\put(20, 0){\line(0, 1){60}}
\put(25, 30){...}
\put(40, 0){\line(0, 1){60}}

\qbezier(60, 0)(60, 0)(100, 60)
\qbezier(80, 0)(80, 0)(60, 60)
\qbezier(100, 0)(100, 0)(80, 60)

\put(68, 50){...}

\put(120, 0){\line(0, 1){60}}
\put(140, 0){\line(0, 1){60}}
\put(125, 30){...}
\put(56, -10){\small$i$}
\put(96, -10){\small{$j$}}
\put(150, 0){, }
\end{picture} 
\hspace{.5cm}
\begin{picture}(160, 70)(-20,0)
\put(-20, 28){$X_{j i}  =$}
\put(20, 0){\line(0, 1){60}}
\put(25, 30){...}
\put(40, 0){\line(0, 1){60}}

\qbezier(60, 60)(60, 60)(100, 0)
\qbezier(80, 0)(80, 0)(100, 60)
\qbezier(60, 0)(60, 0)(80, 60)

\put(83, 50){...}

\put(120, 0){\line(0, 1){60}}
\put(140, 0){\line(0, 1){60}}
\put(125, 30){...}
\put(56, -10){\small$i$}
\put(96, -10){\small{$j$}}
\end{picture}
\caption{The elements $X_{i j}$ and $X_{j i}$ for $i<j$}
\label{fig:Xij}
\end{figure}
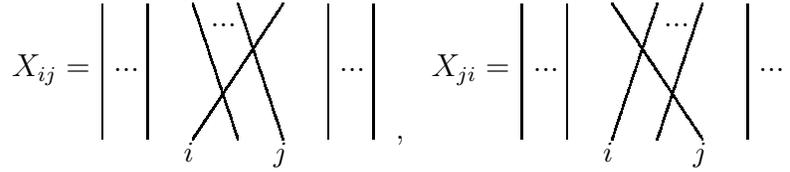

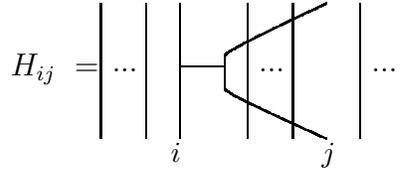
\begin{figure}[h]
\setlength{\unitlength}{0.3mm}
\begin{picture}(180, 70)(-20,0)
\put(-15, 28){$H_{ij}\  =$}
\put(25, 0){\line(0, 1){60}}
\put(30, 30){...}
\put(45, 0){\line(0, 1){60}}

\put(60, 0){\line(0, 1){60}}
\put(60, 32){\line(1, 0){20}}
\put(80, 22){\line(0, 1){15}}

\qbezier(80, 22)(80, 20)(125, 0)
\qbezier(80, 37)(80, 39)(125, 60)


\put(90, 0){\line(0, 1){60}}
\put(110, 0){\line(0, 1){60}}
\put(95, 30){...}
\put(56, -10){\small$i$}
\put(123, -10){\small$j$}

\put(140, 0){\line(0, 1){60}}
\put(145, 30){...}
\put(160, 0){\line(0, 1){60}}
\end{picture} 
\caption{The elements $H_{ij}$}
\label{fig:Hij}
\end{figure}

 \noindent
 We have the following result.
 \begin{lemma} \label{lem:4-t-Brauer}  For all  $i, j, k, \ell$ which are pairwise distinct, we have (cf. Corollary \ref{cor:4t})
\[
[H_{i j}, H_{k \ell}] = 0, \quad [H_{i k} + H_{i\ell}, H_{k\ell}] = 0, \quad [H_{i j} , H_{i k} + H_{j k}] = 0.
\]
\end{lemma}
\begin{proof}
Note that by conjugating the $H$'s by appropriate $X_{i j}$'s as depicted in Figure \ref{fig:Xij}, we may transform the second and third relations into   
\beq\label{eq:essence}
[H_{12}, H_{13}+H_{23}]=0, \quad  [H_{12} + H_{13}, H_{23}]=0. 
\eeq
To prove these relations, by applying a conjugation by a succession of permutations $X_{ij}$, it is clear that we need only verify them for $r=3$. 
A brute force calculation shows that 
\[
\begin{aligned}
H_{12}(H_{13}+H_{23})=(H_{13}+H_{23})H_{12}\\
\setlength{\unitlength}{0.3mm}
\begin{picture}(180, 50)(0,0)
\put(-20, 15){$=$}
\qbezier(0, 40)(0, 40)(30, 0)
\qbezier(20, 40)(20, 40)(0, 0)
\qbezier(30, 40)(30, 40)(10, 0)

\put(35, 15){$+$}

\qbezier(50, 40)(50, 40)(70, 0)
\qbezier(60, 40)(60, 40)(80, 0)
\qbezier(50, 0)(50, 0)(80, 40)

\put(85, 15){$-$}

\qbezier(100, 40)(115, 10)(130, 40)
\qbezier(110, 0)(118, 30)(125, 0)
\qbezier(100, 0)(100, 0)(115, 40)

\put(135, 15){$-$}

\qbezier(150, 40)(150, 40)(165, 0)
\qbezier(160, 40)(168, 10)(175, 40)
\qbezier(150, 0)(165, 30)(180, 0)

\put(185, 5){, }
\end{picture} 
\end{aligned}
\]
and 
\[
\begin{aligned}
(H_{12} + H_{13}) H_{23}=H_{23}(H_{12} + H_{13})\\
\setlength{\unitlength}{0.3mm}
\begin{picture}(180, 50)(0,0)
\put(-20, 15){$=$}
\qbezier(0, 40)(0, 40)(30, 0)
\qbezier(20, 40)(20, 40)(0, 0)
\qbezier(30, 40)(30, 40)(10, 0)

\put(35, 15){$+$}

\qbezier(50, 40)(50, 40)(70, 0)
\qbezier(60, 40)(60, 40)(80, 0)
\qbezier(50, 0)(50, 0)(80, 40)

\put(85, 15){$-$}

\qbezier(100, 40)(115, 10)(130, 40)
\qbezier(105, 0)(113, 30)(120, 0)
\qbezier(130, 0)(130, 0)(115, 40)

\put(135, 15){$-$}

\qbezier(180, 40)(180, 40)(165, 0)
\qbezier(155, 40)(162, 10)(170, 40)
\qbezier(150, 0)(165, 30)(180, 0)

\put(185, 5){. }
\end{picture} 
\end{aligned}
\]
This proves the relations \eqref{eq:essence}. 

To prove the first relation, it suffices to prove it for the three cases respectively with $i<j<k<\ell$,  $i<k<j<\ell$ and $i<k<\ell<j$. Again by conjugating the $H$'s by appropriate $X_{i j}$'s, we can always transform them to the relations 
\[
[H_{12}, H_{34}]=0, \quad [H_{13}, H_{24}]=0, \quad [H_{14}, H_{23}]=0.
\] 
Both the first and the last of these relations are entirely obvious;  the second one is equivalent to $s_2[H_{12}, H_{34}]s_2 =0$, which clearly follows from the first one. 

This completes the proof.
\end{proof}


 The first statement of  the theorem below is a direct consequence of Lemma \ref{lem:4-t-Brauer}. 

\begin{thm} \label{thm:Casimir-Brauer}
The map $t_{i j}\mapsto H_{i j}$, for all $i<j$, extends to a unique  $K$-algebra homomorphism 
$
\varphi_r: T_r\lra B_r(\delta), 
$
which is surjective if $\delta- 2$ is a unit in $K$. 

\begin{proof} 
It is immediate from Lemma \ref{lem:4-t-Brauer} and the defining relations \eqref{eq:deftr} for $T_r$ that  the map $t_{i j}\mapsto H_{i j}$ extends to a unique algebra homomorphism. 

To prove the second statement of the theorem, we note that 
$
H_i^2 = 1 + (\delta- 2) e_i,  
$
and hence using $H_i=s_i-e_i$,
\beq\label{eq:H-generating}
 (\delta- 2) e_i= H_i^2 - 1, \quad   (\delta- 2) s_i =  H_i^2 + (\delta- 2)H_i -1.
\eeq
If $\delta- 2$ is a unit in $K$, \eqref{eq:H-generating} shows that $s_i$ and $e_i$ $(\in B_r(\delta))$ are in the image of $\varphi_r$. 
This proves the surjectivity of $\varphi_r$ under the given condition on $\delta$, completing the proof of the theorem.
\end{proof}
\end{thm}

\begin{remark} 
Theorem \ref{thm:Casimir-Brauer} was implicit in \cite{LZ06}, but the diagrammatic role of the Brauer elements $H_{i j}$ clearly
leads to a more interesting interpretation of this result. 
\end{remark}

\subsection{Chord diagrams and the affine Brauer algebra}\label{sect:affine-Casimir}

We next consider the endomorphism algebra $\Hom_{\AB(\delta)}(r, r)$ for any object $r$ of $\AB(\delta)$. It contains the 
subalgebra generated by the elements $Z_\ell\ot I_r$,  where $Z_\ell\in\Hom_{\AB(\delta)}(0, 0)$ is defined in Lemma  \ref{lem:central}. 

Recall that we have a canonical embedding  $\iota$ of the $K$-module of usual Brauer diagrams in the $K$-module of affine 
Brauer diagrams, which takes a  Brauer diagram $A$, to its image $\iota(A)=\A_0$ with $\A_0$ depicted in Figure \ref{fig:ord0}. In particular, we have $\iota(s_k), \iota(e_k)\in \Hom_{\AB(\delta)}(r, r)$ for 
$k=1, 2, \dots, r-1$, where  $s_k, e_k$ are the generators of $B_r(\delta)$.

Let $AB_r(\delta)$ be the $K$-subalgebra of $\Hom_{\AB(\delta)}(r, r)$ generated by 
\[
\{\BH_{0 j}(r)\mid   j=1, \dots, r\}\bigcup \{\iota(s_k), \iota(e_k)  \mid k=1, 2, \dots, r-1  \}, 
\] 
and call it a {\em diagrammatic affine Brauer algebra} of degree $r$ with parameter $\delta$. Note that elements $Z_\ell\ot I_r$ are not in $AB_r(\delta)$.

For all $i, j=1, 2, \dots, r$ with $i<j$, the elements $\BH_{i j}(r)$ given by Figure \ref{fig:Hijr}. can be expressed as  
\[
\BH _{i j}(r)=\iota(H_{i j}).
\]
These are order $0$ affine Brauer diagrams in $AB_r(\delta)$. We also have the order $1$ affine Brauer diagrams
$\BH _{0 j}(r)\in AB_r(\delta)$ for all $j=1, 2, \dots, r$ given in Figure \ref{fig:Hir}.  We have the following affine analogue of Theorem 
\ref{thm:Casimir-Brauer}.

\begin{thm} 
The  map 
$
t_{i j} \mapsto \BH _{i-1, j-1}(r)$, for all $1\le i<j\le r+1, 
$
extends  to a unique algebra homomorphism 
$
\Phi_r: T_{r+1}\lra AB_r(\delta),
$
 which is surjective if $\delta- 2$ is a unit in $K$. 
\end{thm}
\begin{proof}  
Let us consider the first statement of the theorem. We want to show that  for $0\le i<j \le r$, the $\BH _{i j}(r)$  satisfy the relations in \eqref{eq:deftr}. The  case $r=2$ is already implied by 
the four-term relation in the definition of $\AB(\delta)$ and the relation in Lemma \ref{lem:4-t-id}, and we can also easily extend this to $r=3$ by using Lemma \ref{lem:4-t-Brauer}. 
In fact, $r=3$ is the essential case; all other cases can be reduced to this by using permutations, as in the proof of Lemma \ref{lem:4-t-Brauer}. 

The second statement is true if $\BH_{i j}$ with $1\le i<j\le r$ generate a subalgebra isomorphic to $B_r(\delta)$. The given condition on $\delta$ guarantees this in view of \eqref{eq:H-generating}. 
\end{proof}

To better understand the structure of $AB_r(\delta)$, we first note that the elements defined below,
\beq
\Theta_j(r)= \sum_{0\le a< j} \BH_{a j}, \quad j=1, 2, \dots, r, 
\eeq
have the following properties.
\begin{theorem} \label{lem:JM}\label{thm:JM}
The following relations hold in $AB_r(\delta)$ for all valid indices $i, j, k$:
\beq
&& \iota(e_1)(\Theta_1)^\ell \iota(e_1)= Z_\ell \iota(e_1),  \label{eq:JM-1}\\
 &&\Theta_i(r) \Theta_j(r)=  \Theta_j(r) \Theta_i(r), \  \text{ for all  }\    i, j,  \label{eq:JM-2}\\
 &&\iota(s_k) \Theta_j -  \Theta_j \iota(s_k) =0, \ \text{ if $j\not\in\{k, k+1\}$}, \label{eq:JM-3}\\
 &&\iota(e_k) \Theta_j -  \Theta_j \iota(e_k) =0, \ \text{ if $j\not\in\{k, k+1\}$},  \label{eq:JM-4}\\
 && \iota(s_k) \Theta_{k}- \Theta_{k+1}  \iota(s_k) = \iota(e_k)- \I_r,  \label{eq:JM-5}\\
 && \Theta_k \iota(s_k) -\iota(s_k) \Theta_{k+1} = \iota(e_k)- \I_r,  \label{eq:JM-6}\\
 &&\iota(e_k)\left(\Theta_k(r) + \Theta_{k+1}(r) \right) = (1-\delta) \iota(e_k), \label{eq:JM-7}\\
&&\left(\Theta_k(r) + \Theta_{k+1}(r) \right) \iota(e_k)= (1-\delta)\iota(e_k). \label{eq:JM-8}
  \eeq
\end{theorem}
\begin{proof}
The relations \eqref{eq:JM-1}, \eqref{eq:JM-3} and \eqref{eq:JM-4} are quite clear, while \eqref{eq:JM-5} and \eqref{eq:JM-6} follow from the obvious fact that 
\[
\iota(s_k) \Theta_{k} \iota(s_k) = \Theta_{k+1} - \BH_{k, k+1}.
\]

To prove \eqref{eq:JM-2}, we note that the four-term relation implies  
\[
[\BH_{0 i}(r), \BH_{0 j}(r)+\BH_{i j}(r)]=0, \forall 0<i<j. 
\]
Also, it is not difficult to see diagrammatically that
\beq\label{eq:mixed}
[\BH_{0 i}(r), \BH_{k j}(r)]=0, \quad \forall j\ne i\ne k. 
\eeq
These relations imply that 
$[\BH_{0 i}(r), \Theta_j(r)]=0$ if $i<j$.
Also, it follows from \eqref{eq:mixed} and Lemma \ref{lem:4-t-Brauer} that 
$
[\BH_{k i}(r), \Theta_j(r)]=0  \text{  for $k<i<j$}. 
$
The last two relations imply that $[\Theta_i(r), \Theta_j(r)]=0$ if $i<j$. 

The relations \eqref{eq:JM-7} and \eqref{eq:JM-8}  are 
consequences of the skew symmetry of $\BH$ and of $H$, as well as the relations
\[
\cap H = (1-\delta) \cap, \quad H \cup = (1-\delta) \cup.
\]
To prove them, we define the  elements $\Pi_i(r)= \I_0\ot I^{\ot (i-1)}\ot\cap\ot I^{\ot (r-i-1)}$ and 
$\amalg_i(r)= \I_0\ot I^{\ot (i-1)}\ot\cup\ot I^{\ot (r-i-1)}$ for $i=1, 2, \dots, r-1$.
Note that $\Pi_i(r) \in \Hom_{\AB(\delta)}(r, r-2)$ and 
$\amalg_i(r) \in \Hom_{\AB(\delta)}(r-2, r)$. 
Then \eqref{eq:JM-7} and \eqref{eq:JM-8} are implied by the following relations for any $i<r$:
\beq
&\Pi_i(r)\left(\Theta_i(r) + \Theta_{i+1}(r) \right) = (1-\delta) \Pi_i(r), \label{eq:cap-theta}\\
&\left(\Theta_i(r) + \Theta_{i+1}(r) \right) \amalg_i(r) = (1-\delta)\amalg_i(r). \label{eq:cup-theta}
\eeq
We  verify the $i=1$ case of \eqref{eq:cap-theta} by the following computation:
\[
\baln
\Pi_1(r) \left(\Theta_1(r) +\Theta_2(r)\right) =\Pi_1(r) \BH_{01}(r)+ \Pi_1(r) \BH_{02}(r) + \I_0\ot \cap H\ot I^{\ot{r-2}},
\ealn
\]
where the first two terms on the right side cancel because of skew symmetry of $\BH$ and $H$, 
and the third term is equal to $(1-\delta) \Pi_1(r)$.  A trivial induction on $i$ proves the general case.  
The relation \eqref{eq:cup-theta}  can be proved in the same way. 

This completes the proof of the theorem.
\end{proof}

Note that the formulae in Theorem \ref{thm:JM} look very similar to defining relations 
of the Nazarov--Wenzl algebra of degree $r$ defined in \cite[Definition 2.1]{ES}. 
 The relationship between our algebra $AB_r(\delta)$ and the latter algebra will be discussed in Theorem \ref{thm:NW-alg}. 

\section{Representations of the polar Brauer category}\label{sect:centre}
Henceforth we take $K$ to be the field $\C$ of complex numbers.
\subsection{The Casimir algebra}
Let us first provide some background material needed for studying representations of the diagrammatic affine Brauer category. This will also serve to fix the notation. 

\subsubsection{The Casimir algebra}\label{sect:Casimir}

Let $\fg=\fg_{\bar0}+\fg_{\bar1}$ be a finite dimensional Lie superalgebra over $\C$, which reduces to an ordinary Lie algebra if $\fg_{\bar 1}=0$.   
Denote the Lie product in $\fg$ by $[-,-]$,  and write $\ad$ for
the adjoint action of $\fg$ on itself, so that for $x,y\in\fg$, we have 
$
\ad(x)y:=[x,y].
$ 
We assume that $\fg$ has a non-degenerate supersymmetric even bilinear form $\kappa:\fg\times \fg\lr\C$ which is $\ad(\fg)$-invariant. That is, 
for $x,a$ and $b$ in $\fg$, we have 
\be\label{eq:killing}
\kappa(\ad(x)a, b)+(-1)^{[a]([a]+[b])}  \kappa(a,\ad(x)b)=0,
\ee
where $[a]$ and $[b]$ denote the $\Z_2$-degrees of the  elements $a, b\in \fg$ respectively. 
Here we say that $\kappa$ is supersymmetric if $\kappa(x, y) = (-1)^{[x][y]}\kappa(y, x)$, and is even if $\kappa(\fg_{\bar0}, \fg_{\bar1})=0$.  

Let $\U(\fg)$ be the universal enveloping superalgebra of $\fg$; we regard $\fg$ as embedded in $\U(\fg)$ in the canonical way. Let $\{X_\alpha\mid\alpha=1,2,\dots,\dim(\fg)\}$ be 
a basis of $\fg$ and let $\{X^\alpha\mid\alpha=1,2,\dots,\dim(\fg)\}$ be the dual basis with respect to the form $\kappa$, that is, $\kappa(X^a, X_b)=\delta_b^a$. 
Then the element $C$ defined below, which  lies in the centre of $\U(\fg)$, is the quadratic Casimir of $\fg$.
\beq\label{eq:cas}
C=\sum_{\alpha=1}^{\dim(\fg)}X_\alpha X^\alpha, 
\eeq
 It is evident that $C$ is independent of the basis, and 
$C=\sum_{\alpha=1}^{\dim(\fg)}(-1)^{[X_\alpha]}X^\alpha X_\alpha$ by the supersymmetry of the bilinear form $\kappa$. 

It is well known that $\U(\fg)$ carries a Hopf superalgebra \cite{Z98} structure, with comultiplication given by $\Delta(X)=X\ot 1+ 1\ot X$ for $X\in\fg$. We therefore have
the iterated comultiplications $\Delta^{(r-1)}:\U(\fg)\lr \U(\fg)^{\ot r},$ for $r=2,3,\dots$, where
\be\label{eq:tensoraction}
\Delta^{(r-1)}=\left(\Delta\ot \id_{\U(\fg)}^{\ot (r-2)}\right)\circ\Delta^{(r-2)},
\ee
which are superalgebra homomorphisms. 

It follows that if $M_1,M_2,\dots,M_r$ are any $\U(\fg)$-modules whatsoever, we may define a $\U(\fg)$-action on $M_1\ot\dots\ot M_r$,
using the fact that $M_1\ot\dots\ot M_r$ is evidently a $\U(\fg)^{\ot r}$-module, as well as the homomorphism $\Delta^{(r-1)}:\U(\fg)\lr\U(\fg)^{\ot r}$.

Define the element $t\in\U(\fg)^{\ot 2}$ as follows:
\be\label{eq:t}
t=\frac{1}{2}\left(\Delta(C)-C\ot 1-1\ot C\right).
\ee

It is evident that $t=\sum_{\alpha}X_\alpha\ot X^\alpha=\sum_\alpha (-1)^{[X_\alpha]} X^\alpha\ot X_\alpha$ in view of the supersymmetry of the bilinear form. Further, given a fixed positive integer $r$, we have elements $C_{ij}\in\U(\fg)^{\ot r}$, $1\leq i<j\leq r$, given by
\be\label{eq:cij}
C_{ij}:=\sum_{\alpha=1}^{\dim(\fg)}1^{\ot (i-1)}\ot X_\alpha\ot 1^{\ot (j-i-1)}\ot X^\alpha\ot 1^{\ot (r-j)},
\ee
and in particular, $C_{12}=t\ot 1^{\ot (r-2)}$. 

We refer to $t$ as the {\em tempered Casimir element} of $\U(\fg)^{\ot 2}$ and to the $C_{ij}$ as {\em expansions} of $t$.

\begin{definition}\label{def:Cr}
Denote by $C_r(\fg)$ the subalgebra of $\U(\fg)^{\ot r}$ generated by the elements $C_{i j}$ for all $1\le i<j\le r$, and call it a {\em Casimir algebra} of $\fg$. 
\end{definition}

It is crucial for our purposes to observe that for all $z\in C_r(\fg)$,
we have 
\be\label{eq:tcomm}
[z, \Delta^{(r-1)}(x)]=0, \quad \forall x\in \U(\fg).
\ee 
It is easily shown \cite{LZ06} that for all pairwise distinct indices $i, j, k, \ell$, the following relations hold in $C_r(\fg)$.
\beq\label{eq:C-4-term}
[C_{i j}, C_{k \ell}] = 0, \quad [C_{i k} + C_{i\ell}, C_{k\ell}] = 0, \quad [C_{i j} , C_{i k} + C_{j k}] = 0.
\eeq
These relations together with \eqref{eq:tcomm} imply the following result. 
\begin{thm}[\cite{LZ06}] \label{thm:T-C} The map $t_{ij}\mapsto C_{ij}$, for all $1\le i<j \le r$, extends  
to a unique surjective algebra homomorphism
 $\psi_r: T_r \lra C_r(\fg)$. Furthermore, for any $\zeta \in T_r$, 
\[
[\psi_r(\zeta), \Delta^{(r-1)}(x)] = 0, \quad \forall x\in\U(\fg).
\] 
\end{thm}

%
%

The theorem has a direct bearing on the representation theory of $\U(\fg)$. 
Let $V_1,V_2,\dots,V_r$ be arbitrary (left) modules for $\U(\fg)$. Then $V_1\ot V_2\ot\dots\ot V_r$ is a $\U(\fg)^{\ot r}$-module
in the obvious way, which restricts to a $C_r(\fg)$-module. Hence the homomorphism $\psi:T_r\lr C_r(\fg)$ defines an action of $T_r$ on $V_1\ot\dots\ot V_r$.
On the other hand, iterated co-multiplication $\Delta^{(r-1)}$ defines a $\U(\fg)$-action on $V_1\ot V_2\ot\dots\ot V_r$. This is
the ``tensor product'' of the representations $V_i$, and analysis of such tensor products is the subject of a significant literature for particular choices of $V_i$ and $r$.

Now the action of $T_r$ on $V_1\ot\dots\ot V_r$ commutes with the $\U(\fg)$-action \eqref{eq:tcomm}, whence we have an algebra homomorphism
\be\label{eq:homend}
\eta_{V_1, \dots, V_r}: T_r\lr \End_{\U(\fg)}(V_1\ot\dots\ot V_r).
\ee

Clearly the utility of \eqref{eq:homend} depends to a large extent on the choice of the representations $V_1,\dots,V_r$.  At this stage they could be completely arbitrary. 
However in \cite{LZ06} it is shown that if $V_1=V_2=\dots=V_r=V$, a ``strongly multiplicity free''  module for a simple Lie algebra $\fg$, then $\eta_{V, \dots, V}(T_r)$ is essentially the full 
commutant of the $\U(\fg)$-module $V^{\ot r}$, and is semi-simple. Thus the module $V^{\ot r}$ may be analysed by studying the structure of $\eta_{V, \dots, V}(T_r)$.

Further, it is shown in \cite[\S 4]{LZ06} how the image of the Brauer algebra in $\End_{\fg}(V^{\ot r})$ appears in $\eta_{V, \dots, V}(T_r)$ for $\fg$,
when $V$ is the natural $\fg$-module, where $\fg$ is the orthogonal or symplectic Lie algebra.  This suggests that diagrammatics such as those in the Brauer category could
be applied to study applications of $T_r$  to the representation theory of Lie superalgebras. 

\subsubsection{The orthosymplectic Lie superalgebra}\label{sect:osp}
Details concerning the orthosymplectic Lie superalgebra which are used in the remainder 
of this section and Section \ref{sect:Ug} may be found in Appendix \ref{sect:tensor-notation}. 
Only some basic notions are recalled here to maintain continuity of the presentation. 

 Given a pair of non-negative integers $m$ and $n$ of which at least one is non-zero, let $V=\C^{m|2n}$.  
 We use $[v]=0, 1$ to denote the $\Z_2$-degree of a homogeneous element $v\in V$, and denote by $\dim(V)=m+2n$ and $\sdim(V)=m-2n$ respectively the the dimension and the super dimension of $V$. 
We equip $V$ with a non-degenerate bilinear form $\omega$, which is assumed to be even and supersymmetric. 
Here being even means that $\omega(V_{\bar 0}, V_{\bar 1})= \omega(V_{\bar 1}, V_{\bar 0})=\{0\}$, and supersymmetry means $\omega(v, v')=(-1)^{[v][v']}\omega(v', v)$ for any $v, v'\in V$. 

Let $\osp(V; \omega)$ be the orthosymplectic Lie superalgebra preserving the bilinear form $\omega$. We discuss $\osp(V; \omega)$ and its universal enveloping superalgebra $\U(\osp(V; \omega))$
 in detail in Appendix \ref{sect:tensor-notation}.  The even Lie subalgebra of $\osp(V; \omega)$ is $\osp(V; \omega)_0=\mathfrak{so}(V_{\bar 0}; \omega|_{\bar 0})\oplus  \mathfrak{sp}(V_{\bar 1}; \omega|_{\bar 1})$, 
where $\omega|_{\alpha}$ is the restriction of $\omega$ to $V_{\alpha}\times V_{\alpha}$ for $\alpha={\bar 0}$ and ${\bar 1}$.  
 In particular, $\osp(V; \omega)$ reduces to
$\mathfrak{sp}_{2n}$  if $m=0$, and $\mathfrak{so}_m$ if $n=0$.

Denote by ${\rm OSp}(V; \omega)$ the orthosymplectic supergroup on $V$ preserving the bilinear form $\omega$.  
The even subgroup of ${\rm OSp}(V; \omega)$ is ${\rm OSp}(V; \omega)_0={\rm SO}(V_{\bar 0}; \omega|_{\bar 0})\times {\rm Sp}(V_{\bar 1}; \omega|_{\bar 1})$.  
It is conceptually appealing to consider the orthosymplectic supergroup as a group scheme, however, for the purpose of this paper, we will follow \cite{DM} 
to identify ${\rm OSp}(V; \omega)$ with the Harish-Chandra pair $({\rm OSp}(V; \omega)_0, \U(\osp(V; \omega)))$, where ${\rm OSp}(V; \omega)_0$ acts on  $\U(\osp(V; \omega))$ by conjugation. 

\medskip
\noindent{\bf Notation}.
Henceforth we fix $V=\C^{m|2n}$ and denote
\beq\label{eq:notation}
\phantom{XXX} \fg=\osp(V; \omega),\quad G={\rm OSp}(V; \omega), \quad 
G_0={\rm SO}(V_{\bar 0}; \omega|_{\bar 0})\times {\rm Sp}(V_{\bar 1}; \omega|_{\bar 1}). 
\eeq
We also write $\sdim=\sdim(V)$.

\subsection{A quartet of algebras}\label{sect:quartet}
 We shall explore the relationships among the quartet of algebras $T_r$,  $B_r(\sdim)$, the Casimir algebra $C_r(\fg)$, and $\End_\fg(V^{\ot r})$ for $\fg=\osp(V; \omega)$.

We take $C$ to be the quadratic Casimir operator of $\U(\fg)$ defined by 
\eqref{eq:Casimir-formula}; the corresponding tempered Casimir operator $t$ is given by  \eqref{eq:t-formula}.  We define the elements $C_{ij}$, and hence the Casimir algebra $C_r(\fg)$ of $\fg$, using this $t$,
the tempered Casimir operator.

\begin{remark}\label{rmk:norm} 
In Appendix \ref{sect:tensor-notation}, we calculate the eigenvalue  in $V$ of the Casimir operator $C$ defined by \eqref{eq:Casimir-formula} by direct computation and obtained $\chi_V(C)= \sdim -1$; see \eqref{eq:eigen-C}. 
In general $C$ acts as the eigenvalue $(\lambda+2\rho, \lambda)$ in a simple $\U(\fg)$-module $L_\lambda$ of highest weight $\lambda$, where the bilinear form on the weight space is normalised so that $sup\{|(\alpha, \alpha)| \mid \text{$\alpha$ is a root}\}=4$.  As usual, $2\rho$ is the signed sum of the positive roots. 
\end{remark}

Denote by $\mu$ the natural $\fg$-action on $V$, which extends to a $\U(\fg)$-action. 
Write $E_r(V)=\End_{\U(\fg)}(V^{\ot r})$  for each $r$.
We have the following representation of $C_r(\fg)$,
\[
\mu_r=\mu^{\ot r}: C_r(\fg)\lra E_r(V).
\] 
We also have the representation of $T_r$ defined by \eqref{eq:homend}.
\[
\eta_r=\eta_{V, \dots,V}: T_r\lr E_r(V)
\]

The non-degenerate bilinear form $\omega$ gives rise to an isomorphism $j: V\ot V\stackrel{\sim}\lra \End_\C(V)$ in a canonical way,  with $v\ot v'\mapsto v \omega(v', -)$ for all $v, v'\in V$. This is a $\fg$-module map with respect to the $\fg$-action on $\End_\C(V)$ defined, for any $X\in\fg$ and $\varphi\in \End_\C(V)$, by 
\beq\label{eq:act-endo}
(X.\varphi)(v)= X.(\varphi(v))-  (-1)^{[X][\varphi]} \varphi(X.v), \quad \forall v\in V. 
\eeq

Let $(e_a)$, $a=1,\dots,\sdim(V)$ be a homogeneous basis of $V$ and let $(e^a)$, $a=1,\dots,\sdim(V)$ be the dual basis with respect to $\omega$,
so that
\be\label{eq:defe}
\omega(e^a,e_{a'})=\delta_{a,a'}.
\ee

Now we define the following linear maps, 
\beq\label{eq:cup-cap}
\begin{aligned}
&\hat{c}: V\ot V\lra \C, \quad \hat{c}(v\ot w)=\omega(v, w), \\
& \check{c}: \C\lra  V\ot V, \quad \check{c}(1)=j^{-1}(\id_V).
\end{aligned}
\eeq
These are $\U(\fg)$-module homomorphisms \cite{LZ15, LZ17} and it is easily verified that $\check c(1)=\sum_{a=1}^{\dim(V)}e_a\ot e^a$. 

They satisfy the following relations:
$
(\id_V\ot \hat{c})(\check{c}\ot \id_V)=\id_V$ and $\hat{c}\circ \check{c}(1)=\sdim.
$

Given any $\Z_2$-graded vector spaces $M$ and $M'$, there is the canonical linear map 
\beq\label{eq:tau}
\tau_{M, M'}: M\ot M'\lra M'\ot M, 
\eeq
defined by $v\ot v' \mapsto(-1)^{[v][v']}v'\ot v$ for any $v\in M, \ v' \in M'$. 
In particular,  we have the permutation map $\tau:=\tau_{V, V}: V\ot V\lra V\ot V$.   
Denote  
\beq\label{eq:e}
e=\check{c}\circ \hat{c}. 
\eeq
Both $\tau$ and $e$ belong to $\End_{\U(\fg)}(V\ot V)$.

There exists the canonical representation \cite{LZ17} of the Brauer algebra 
\beq\label{eq:brauer-rep}
\nu_r: B_r(\sdim)\lra E_r(V) 
\eeq
defined by 
$
\nu_r(s_i)=\id_V^{\ot(i-1)}\ot\tau\ot \id_V^{r-i-1)}
$ and $
\nu_r(e_i)=\id_V^{\ot(i-1)}\ot  e\ot \id_V^{r-i-1)}.
$
The theorem below describes the interrelationship among the quartet of algebras $T_r$, $B_r(\sdim)$, $C_r(\fg)$ and $E_r(V)$. 
\begin{thm} \label{thm:T-C-B}
Retain the notation above. Let $C_r(\fg)$ be the Casimir algebra of $\fg$ defined relative to the quadratic Casimir operator \eqref{eq:Casimir-formula}.  The following diagram commutes.
\[
\begin{tikzcd}
T_r \arrow[d, "\psi_r"] \arrow[r, "\phi_r"] \arrow[dr, "\eta_r"] & B_r(\sdim)\arrow[d, "\nu_r"] \\
C_r(\fg) \arrow[r, "\mu_r"] & E_r(V), 
\end{tikzcd}
\]
where $\eta_r= \eta_{V, \dots, V}$ is defined by \eqref{eq:homend}.
\end{thm}
\begin{proof} 
Note that $\eta_r=\mu_r\psi_r$ by definition. 
We have 
$
\nu_r\phi_r(t_{i j})=\nu_r(H_{i j})
$ for all $i<j$,  and also $
\mu_r\psi_r(t_{i j})=\mu_r(C_{i j})
$
by Theorem \ref{thm:T-C}. Lemma \ref{lem:key} below implies $\mu_r(C_{i j})=\nu_r(H_{i j})$ for all $i<j$, and given this, the diagram is commutative.  
\end{proof}

It remains to prove the following result.

\begin{lemma}[Key lemma] \label{lem:key} The tempered Casimir element $t$ satisfies the relation
\beq\label{eq:t-image}
(\mu\ot\mu)(t) = \tau - e.
\eeq
\end{lemma}
\begin{proof}[Comments on the proof] 
This lemma is the statement of Theorem \ref{thm:H-formula}, which is proved  in Appendix \ref{sect:tensor-notation} by writing down the action of $t$ on $V\ot V$ explicitly using the standard basis of $V$, and then verifying  that \eqref{eq:t-image} is indeed true by direct calculations. 

While the proof in Appendix \ref{sect:tensor-notation} is completely elementary,  it is not very illuminating. 
When $\sdim(V)=m-2n\ne 0$, we have a representation theoretical proof of \eqref{eq:t-image}, which is more conceptual and helps one to understand why the relation should be true.  In this case, the tensor product $V\ot V$ is semi-simple and multiplicity free with three simple $\U(\fg)$-submodules, 
$
V\ot V =L_s \oplus L_a \oplus L_0, 
$ 
where $L_s \oplus L_0$ and $L_a$ are eigenspaces of $\tau$ with eigenvalues $+1$ and $-1$ respectively. Let $P_i: V\ot V\lra L_i$ be the idempotent projections 
for $i=s, a, 0$; they are mutually orthogonal and complete.   Then 
\[
P_s+P_0= \frac{1}{2}(1+\tau), \quad P_a=\frac{1}{2}(1-\tau), \quad P_0=\frac{e}{\sdim(V)}.
\]

Denote the eigenvaule of the Casimir operators $C$ (see \eqref{eq:Casimir-formula}) in each simple sub-module $L_i$ by $\chi_i(C)$, and that in $V$ by $\chi_V(C)$.  We have (see Remark \ref{rmk:norm})
\[
\begin{aligned}
&\chi_V(C)=\sdim(V)-1, \quad \chi_s(C)=2\sdim(V), \\
&\chi_a(C)=2(\sdim(V)-2), \quad  \chi_0(C)=0.
\end{aligned}
\]
The eigenvalue of  $t$ in $L_i$ is given by $\chi_i(t)= \frac{1}{2} \chi_i(C)- \chi_V(C)$, thus we have 
\[
\chi_s(t)=1, \quad \chi_a(t)=-1, \quad \chi_0(t)=-\sdim(V)+1. 
\]
Hence 
$
(\mu\ot\mu)(t)= P_s - P_a + (-\sdim(V)+1) P_0 
=\tau - e 
$
as desired.
\end{proof}

\begin{remark}
Note that the stated normalisation of $C$ is crucial for Theorem \ref{thm:T-C-B}. 
\end{remark}

%
%
%
\subsection{Representations of the polar Brauer category}
Retain the notation above. Denote by $\CT(V)$ the full subcategory of $\U(\fg)$-modules with objects $V^{\ot r}$ for all $r\in\N$, where $V^{\ot 0}=\C$.  Note that $\CT(V)$ is a tensor category.  

The following theorem can be deduced from \cite{LZ17} or \cite{DLZ}. 
\begin{thm}  \label{thm:funct}
There exists a tensor functor 
$
\CF: \CB(\sdim)\lra \CT(V),
$
which maps an object $r$ to $V^{\ot r}$, and maps the generators of the morphisms to
\[
\begin{aligned}
&\CF\left(
\begin{picture}(10, 20)(0,10)
\put(5, 0){\line(0, 1){30}}
\end{picture}\right) = \id_V, \quad
&\CF\left(
\begin{picture}(28, 23)(0,15)
\qbezier(0, 0)(0, 0)(25, 35)
\qbezier(0, 35)(0, 35)(25, 0)
\end{picture}\right) = \tau, \\
&\CF\left(\begin{picture}(28, 15)(0,8)
\qbezier(0, 0)(13, 35)(25, 0)
\end{picture}\right)= \hat{c}, \quad
&\CF\left(
\begin{picture}(28, 15)(0,25)
\qbezier(0, 35)(13, 0)(25, 35)
\end{picture}\right)=\check{c}.
\end{aligned}
\]
The restriction of $\CF$ to $B_r(\sdim)=\Hom_{\CB(\sdim)}(r, r)$ coincides with the representation of the Brauer algebra defined by \eqref{eq:brauer-rep}. 
\end{thm}

\begin{remark}\label{rem:funct-F}
\begin{enumerate}
\item 
In \cite{LZ17} we constructed a full tensor functor from $\CB(\sdim)$ to the full subcategory of representations of $G={\rm OSp}(V; \omega)$ with objects $V^{\ot r}$ for all $r\in\N$. Note that 
\[\Hom_G(V^{\ot r},  V^{\ot s})= \left\{ \varphi\in \Hom_{\U(\fg)}(V^{\ot r},  V^{\ot s})\mid \eta.\varphi = \varphi, \forall \eta\in G_0\right\}.\]

\item The functor in Theorem \ref{thm:funct} is not full in general. There is a $1$-dimensional $\U(\fg)$-submodule in $V^{\ot r_c}$ for $r_c= m(2n+1)$ spanned by the super Pfaffian $\wt\Omega$ \cite{LZ17-Pf}, which satisfies $\eta.\wt\Omega= \sdet(\eta) \wt\Omega$ for all $\eta\in G_0$, and hence is not in the image of any morphism in $\CF\left(\Hom_{\CB(\sdim)}(0, r_c)\right)=\Hom_G(\C, V^{\ot r_c})$.   This problem already appears in the invariant theory of $\mathfrak{so}_n$, and has been addressed in \cite{LZ17-Ecate} by introducing an ``enhanced Brauer category''.
\end{enumerate}
\end{remark}

Given any $\U(\fg)$-module $M$, we denote by $\mu_M: \U(\fg)\lra \End_\C(M)$ the corresponding representation. Let $\CT_M(V)$ be the full subcategory of $\U(\fg)$-modules with objects $M\ot V^{\ot r}$ 
for all $r\in\N$. There is a bi-functor $\ot: \CT_M(V)\times \CT(V)\lra \CT_M(V)$ such that for any $r, r_0$,
\[
M\ot V^{\ot r}\times V^{\ot r_0}\mapsto M\ot V^{\ot r}\otimes V^{\ot r_0}=M\ot V^{\ot (r+r_0)} 
\]
and for any $\U(\fg)$-module homomorphisms $g\in \Hom_{\U(\fg)}(M\ot V^{\ot r}, M\ot V^{\ot s})$ and $f\in  \Hom_{\U(\fg)}(V^{\ot k}, V^{\ot \ell})$,
\[
g\times f \mapsto g\ot_\C f\in \Hom_{\U(\fg)}(M\ot V^{\ot (r+k)}, M\ot V^{\ot (s+\ell)}).
\]
The bi-functor described defines a module category structure of $\CT_M(V)$ over $\CT(V)$ as a monoidal category. 

We shall now prove the following affine version of Theorem \ref{thm:funct}.

\begin{thm}\label{thm:a-funct} Retain the above notation. Let $C$ be the quadratic Casimir operator of $\fg$ as given by \eqref{eq:Casimir-formula} and $t=\frac{1}{2}\left(\Delta(C)- C\ot 1 - 1 \ot C\right)$
the corresponding tempered Casimir operator. 
There exists a unique functor 
\[
\CF_M: \AB(\sdim)\lra \CT_M(V),  
\]
which sends any object $r$ to $\CF_M(r)=M\ot V^{\ot r}$, and for any $A\in\Hom(\CB(\sdim))$, the morphisms $\I_0$, $\BH$ and $\A_0$ depicted in Figure \ref{fig:ord0} , 
respectively to
\[
\CF_M(\I_0)=\id_M, \quad \CF_M(\mathbb H)=(\mu_M\ot\mu)(t),   \quad \CF_M(\A_0)=\id_M\ot\CF(A), 
\]
and preserves the tensor product $\AB(\sdim)\ot  \CB(\sdim)\lra  \AB(\sdim)$ in the sense that the following diagram commutes. 
\beq\label{eq:tensor-funct}
\begin{tikzcd}
\AB(\sdim) \times \CB(\sdim) \arrow[d, "\ot"' pos=0.43] \arrow[r, "\CF_M\times \CF"] & \CT_M(V)\times \CT(V)\arrow[d, "\ot"]\\
\AB(\sdim) \arrow[r, "\CF_M"] &  \CT_M(V). 
\end{tikzcd}
\eeq
\end{thm}

\begin{proof} 
Since morphisms of $\AB(\sdim)$ are generated by $\I_0$, $\BH$ and $B$ for all usual Brauer diagrams $B$, the functor $\CF_M$ is defined uniquely. 
In order to prove the theorem, we only need to show that $\CF_M$ preserves the four-term relation and skew symmetry of $\BH $; the rest is quite obvious in view of Theorem \ref{thm:funct}.

Recall the notation $\X_0=\I_0\ot X$, where the Brauer $(2, 2)$-diagram $X$ is given in Figure \ref{fig:t-image}. 
Then we have 
\[
\BH_{01}= \BH\ot I, \quad \BH_{02}=\X_0\BH_{01} \X_0, \quad \BH_{12}=\I_0\ot H,
\]
and it follows that 
\[ 
\begin{aligned}
\CF_M(\BH_{01})	&= (\mu_M\ot \mu\ot\mu)(C_{12}), \\
\CF_M(\BH_{12})&=(\mu_M\ot \mu\ot\mu)(C_{23}), \\
\CF_M(\BH_{02})&= (\mu_M\ot \mu\ot\mu)(C_{13}), \quad C_{i j}\in C_3(\fg).
\end{aligned}
 \]
The first and second relations are clear. To prove the third, 
we use $\CF_M(\BH_{01})=\sum_a\mu_M(X_a)\ot\mu(X^a)\ot\id_V$ to express  $\CF_M(\BH_{02})$ as
\[ 
\begin{aligned}
\CF_M(\BH_{02})&=(\id_M\ot \CF(X))\CF_M(\BH_{01})(\id_M\ot \CF(X))\\
						&=(\id_M\ot\tau)\left(\sum_a\mu_M(X_a)\ot\mu(X^a)\ot\id_V\right)(\id_M\ot\tau)\\
						&=\sum_a\mu_M(X_a)\ot\id_V\ot\mu(X^a) = (\mu_M\ot \mu\ot\mu)(C_{13}).
\end{aligned}
 \]
Now it follows from \eqref{eq:C-4-term} for $C_3(\fg)$ that 
\[
[\CF_M(\BH_{01}), \CF_M(\BH_{02})+\CF_M(\BH_{12})]
=(\mu_M\ot \mu\ot\mu)\left([C_{12}, C_{13}+C_{23}] \right)=0.
\]
This proves that the functor $\CF_M$ preserves the four-term relation. 

To prove that the functor $\CF_M$ respects the skew symmetry of $\BH$, we note that 
\[ 
\begin{aligned}
\CF_M(\BH^T)=\sum_a \mu_M(X_a)\ot \mu^\omega(X^a), 
\end{aligned}
\]
where 
$
\mu^\omega(Y):=(\hat{c}\ot\id_V)(\mu(Y)\ot\tau)(\check{c}\ot\id_V)  
$
for any $Y\in\fg\hookrightarrow\U$.

For any $v\in V$, we have $\mu^\omega(Y)(v)=(\hat{c}\ot\id_V)(\mu(Y)\ot\tau)(\check{c}(1)\ot v)$, which can be re-written as 
\[
\mu^\omega(Y)(v)=  \sum_a  (-1)^{[e_a][v]} \omega(\mu(Y) e^a, v) e_a
\]
by using the definitions of $\check{c}, \hat{c}$.  
The right hand side can be further simplified by using properties of $\omega$, leading to  
\[ 
\begin{aligned}
\mu^\omega(Y)(v)
				&= -  \sum_a  (-1)^{[e_a]([v] + [Y])} \omega( e^a, \mu(Y) v) e_a&&\text{($\fg$ invariance of $\omega$)}\\
				&= -  \sum_a  (-1)^{[e_a]}( e^a, \mu(Y) v) e_a  							&&  \text{(evenness of $\omega$)} \\
				&= -  \sum_a  (\mu(Y) v,  e^a) e_a  				&&  \text{(supersymmetry of $\omega$)} \\
				&= -\mu(Y)v,  
\end{aligned}
 \]
that is, $\mu^\omega(Y)= -\mu(Y)$. It follows that 
\[ 
\begin{aligned}
\CF_M(\BH^T)=-\sum_a \mu_M(X_a)\ot \mu(X^a) = -\CF_M(\BH). 
\end{aligned}
\]
This proves that the functor $\CF_M$ respects the skew symmetry of $\BH$, completing the proof of the theorem.
\end{proof}

Recall that the antipode  of $\U(\fg)$ as a Hopf superalgebra \cite{Z98}, which we denote by $S$,  is a superalgebra anti-automorphism such that $S(u v)= (-1)^{[u][v]} S(v) S(u)$ and $S(X)=-X$ for any $u, v\in \U(\fg)$ and $X\in\fg\hookrightarrow \U(\fg)$. Arguments similar to those above show that $\mu^\omega(u)= \mu(S(u))$ for all $u\in \U(\fg)$. This  leads in particular to 
\beq
\CF_M((\BH^\ell)^T) = (\mu_M\ot\mu S)(t^\ell), \quad \forall \ell\ge 1. 
\eeq

\begin{remark}\label{rem:super-Pf} 
The $\U(\fg)$-module map $Pf: \C\lra V^{\ot m(2n+1)}$,  with $Pf(1)$ being the super Pfaffian,  is not in  $\CF(\Hom(\CB(\sdim))$ by Remark \ref{rem:funct-F}(2), and hence $\id_M\ot Pf\not\in\CF_M(\Hom(\AB(\sdim))$.   
However, we expect that if $M$ is a projective $\U(\fg)$-module, $\CF_M$ induces a full functor from $\AB(\sdim)$ to the quotient category $\CT_M(V)/\langle \id_M\ot Pf\rangle$.  
\end{remark}

Write $E_{M, r}(V)=\End_{\U}(M\ot V^{\ot r})$.
The restriction of $\CF_M$ to $\Hom_{\AB(\sdim)}(r, r)$, denoted by  $\CF_M|_r$, yields a representation of the affine Brauer algebra: 
\[
\CF_M|_r:  AB_r(\sdim)\lra E_{M, r}(V).
\]
The following result is an easy consequence of Theorem \ref{thm:a-funct}.
\begin{corollary}
The following diagram commutes.  
\[
\begin{tikzcd}
T_{r+1} \arrow[d, "\psi_{r+1}"] \arrow[r, "\Phi_r"] \arrow[dr, "\eta_{M, r}"] & AB_r(\sdim)\arrow[d, "\CF_M|_r"] \\
C_{r+1}(\fg) \arrow[r, "\mu_{M, r}"] & E_{M, r}(V),
\end{tikzcd}
\]
where $\eta_{M, r}= \eta_{M, V, \dots, V}$ is defined by \eqref{eq:homend}.
\end{corollary}

Special cases of the functor $\CF_M: \AB(\sdim)\lra \CT_M(V)$ in Theorem \ref{thm:a-funct} will be investigated in depth later in this paper.  

\section{Application to universal enveloping superalgebras}\label{sect:Ug}

We apply the diagrammatics of $\AB(\delta)$ developed above to study the centre of $\U(\osp(V; \omega))$ and give a categorical derivation of characteristic identities \cite[\S 4.10]{B}, \cite{BG, G} for the orthogonal and symplectic Lie algebras. 

A characteristic identity for a Lie algebra  $\fg$ is in the spirit of the Cayley–Hamilton theorem but for a specific type of square matrix with entries in $\fg\hookrightarrow\U(\fg)$. It says that the matrix satisfies a monic polynomial equation over the centre of $\U(\fg)$. Generalisations of characteristic identities to quantum groups \cite{GZB}  are also known. 

Retain the notation of the last section, and denote by  
\beq\label{eq:F-prime-r}
\CF_M|_r:  \Hom_{AB_r(\sdim)}(r, r)\lra E_{M, r}(V)
\eeq
the restriction of the functor 
$
\CF_M: \AB(\sdim)\lra \CT_M(V)
$
 to $\Hom_{\AB(\sdim)}(r, r)$;
this restriction gives rise to a representation of $\Hom_{AB_r(\sdim)}(r, r)$ in the endomorphism algebra on the right side. 

\subsection{Categorical construction of the centre of $\U(\osp(V;\omega))$}\label{sect:centre-construct}

Retain the notation in Section \ref{sect:osp}, in particular, \eqref{eq:notation}, and identify $G=\OSp(V; \omega)$ with the Harish-Chandra pair $(G_0, \U(\fg))$.  We note that any $\eta\in G_0$ is of the form $\eta = \begin{pmatrix}\alpha & 0\\ 0 & \beta\end{pmatrix}$ with $\alpha\in {\rm O}(V_{\bar 0}; \omega_{\bar 0})$ and $\beta \in {\rm Sp}(V_{\bar 1}; \omega_{\bar1})$. Hence the Berezinian $\sdet(\eta)$ is equal  to $\det(\alpha)\det(\beta)^{-1}=\det(\alpha)$, which $1$ or $-1$. 

Both $G$ and $\fg$ act on $\fg$ by the respective adjoint actions, and these actions extend to  actions on
the universal enveloping superalgebra $\U(\fg)$. For any $X\in\fg$ and $g\in G_0$, we denote their actions on any $u\in\U(\fg)$ by  $ X.u$ and $g.u$ respectively. 
The algebraic structure of $\U(\fg)$ is preserved in the sense that, for any $u, v\in\U(\fg)$, 
\beq
X.(u v) = (X.u) v + (-1)^{[X][u]} u (X.v), \quad g.(u v) =  (g.u) (g.v). 
\eeq

Denote by $Z(\U(\fg))$ the centre of $\U(\fg)$, 
which is the subalgebra $\U(\fg)^\fg = \{z\in\U(\fg)\mid X.z=0, \forall X\in\fg\}$ of $\fg$-invariants. 
We also have the subalgebra of $G$-invariants defined by 
$\U(\fg)^G = \{z\in\U(\fg)^\fg\mid \eta.z=z, \forall \eta\in G_0\}$.  
The relationship between $G$-invariants and $\fg$-invariants was investigated in \cite{LZ17-Pf}. 
Consider the following subspace  of $\U(\fg)^\fg$. 
\[
\U(\fg)^{G, \sdet} = \{z\in\U(\fg)^\fg\mid \eta.z=\sdet(\eta) z, \ \forall \eta\in G_0\}.
\]
Then for all  $\alpha, \beta\in \U(\fg)^{G, \sdet}$, we have 
$\alpha\beta\in \U(\fg)^G$. Furthermore, 
\[
\U(\fg)^\fg=\U(\fg)^G\bigoplus \U(\fg)^{G, \sdet}, \quad \text{if $\{\eta\in G_0\mid \sdet(\eta)=-1\}\ne\emptyset$}.
\]

We have the following result. 
\begin{lem}\label{lem:odd}
If $\dim(V_{\bar 0})$ is odd or $0$, then $Z(\U(\fg))=\U(\fg)^G$. 
\begin{proof} 
If $\dim(V_{\bar 0})=0$, we have $G=G_0={\rm Sp}_{2n}$, and hence $\sdet(\eta)=1$ for all $\eta\in G_0$. In this case,  $Z(\U(\fg))=\U(\fg)^\fg=\U(\fg)^G$.

If $\eta_0\in \{\eta\in G_0\mid \sdet(\eta)=-1\}\ne\emptyset$,  then 
$
\U(\fg)^{G, \sdet} = \{z\in \U(\fg)^\fg \mid \eta_0.z=- z\}, 
$
independently of the choice of such an $\eta_0$. Assume that $\dim(V_{\bar 0})$ is odd.  
Let $\eta_0\in G_0$ be the element $\eta_0=-id_V$. 
Then clearly $\sdet(g_0)=-1$, and $\eta_0. X=X$ for all $X\in\fg$.  Hence $
\U(\fg)^{G, \sdet} =0$, and thus $Z(\U(\fg))= \U(\fg)^{G}$. 
\end{proof}
\end{lem}

Since $G$ acts naturally on $V$, and hence on  $V^{\ot r}$ for any $r$, it follows that $G$ acts 
on $\Hom_\C(V^{\ot r}, V^{\ot s})$  ($r, s\in \N$ fixed).  Any $g\in G_0$ and $X\in \fg$ act on $\varphi\in \Hom_\C(V^{\ot r}, V^{\ot s})$
 ($r, s\in \N$ fixed) by 
\[
\baln
& (X.\varphi)({\bf v})= X.(\varphi({\bf v}))- (-1)^{[X][\varphi]} \varphi(X.{\bf v}), \\
& (g.\varphi)({\bf v})= g. (\varphi( g^{-1}.{\bf v})), \quad \forall {\bf v}\in V^{\ot r}. 
\ealn
\]
This $G$-action extends to $\U(\fg)\ot \Hom_\C(V^{\ot r}, V^{\ot s})$ as follows. For any $X\in\fg$, $g\in G_0$, 
$u\in \U(\fg)$, and $\varphi\in \Hom_\C(V^{\ot r}, V^{\ot s})$, 
\beq\label{eq:act-fam}
X.(u\ot \varphi) = X.u \ot \varphi + (-1)^{[X][u]} X.\varphi, \quad
g.(u\ot \varphi) = g. u \ot g.\varphi. 
\eeq
Now it is evident that $\U(\fg)\ot \Hom_\C(V^{\ot r}, V^{\ot s})\ot \Hom_\C(V^{\ot r'}, V^{\ot r'})$ also acquires a $G$-module structure.  

For any $u, u'\in\U(\fg)$, $\varphi\in \Hom_\C(V^{\ot r}, V^{\ot r'})$ and $\psi \in \Hom_\C(V^{\ot r'}, V^{\ot s})$, 
 we have the composition $(u'\ot\psi)(u\ot \varphi)=(-1)^{[u][\psi]}u' u \ot \psi\varphi$. Then one checks easily that
the $G$-action defined by \eqref{eq:act-fam} respects composition.


Let us write $\U=\U(\fg)$, and consider it as a left $\U$-module, the action being multiplication. The functor $\CF_M$ for $M=\U$ then becomes 
\beq
&&\CF_\U: \AB(\sdim)\lra \CT_\U(V) ,
\eeq  
where $\CF_\U(\I_0)=1\in\U$ and $\CF_\U(\BH)=(\id_\U\ot \mu)(t)$. Clearly $\CF_\U(\BH)$ is $G$-invariant with respect to the action \eqref{eq:act-fam}, 
and so is also $\CF(D)$ for any morphism $D$ of $\CB(\sdim)$.   
As $\I_0, \BH$ and all $D$ generate morphisms of $\AB(\sdim)$, 
\beq\label{eq:Hom-G-inv}
\CF_\U(\Hom_{\AB(\sdim)}(r, s)) \subseteq \Hom_\C(\U\ot V^{\ot r}, \U\ot V^{\ot s})^G. 
\eeq

We are now in a postion to prove the following crucial result.
\begin{thm}\label{eq:key-centre}
The identity element together with $\{\CF_\U(Z_{2\ell})\mid \ell\ge 1\}$  generate $\U(\fg)^G$. Further, 
if $\dim(V_{\bar 0})$ is odd or $0$, these elements generate $Z(\U(\fg))$. 
\end{thm}
\begin{proof} 
We begin by laying some groundwork for the proof. 
For any $G$-module $M$,  let $T(M)=\oplus_{r\ge 0} M^{\ot r}$ be the tensor algebra on $M$. 
For each $r$, there is the usual semi-simple $\Sym_r$ action on $M^{\ot r}$, which commutes with the $G$-action. 
Let $\Sigma_{\pm 1}=\sum_{\sigma\in\Sym_r}(\mp 1)^{\ell(\sigma)}\sigma$, where $\ell(\sigma)$ is the length of $\sigma$. 
Then $\Sigma_{-1}$ (resp. $\Sigma_{+1}$) is the Young symmetriser associated with the standard Young tableau of shape 
$(r)$ (resp. $(1^r)$). Write $S_{gr}^k (M)=\Sigma_{-1}M^{\ot r}$ and $\wedge_{gr}^k M= \Sigma_{+1} M^{\ot r}$;
these are $G$-module direct summands of $M^{\ot r}$.  Write $S_{gr}(M)=\oplus_r S_{gr}^k M$ and $\wedge_{gr} (M)= \oplus_r \wedge_{gr}^k M$.
These are our $G$-module direct summands of the tensor algebra $T(M)=\oplus_r M^{\ot r}$.   

Now consider the natural $G$-module $V$, and the following $G$-module  isomorphism
\[
\begin{aligned}
j:\   &  \wedge^2_{gr}V\stackrel{\sim}\lra \fg, & e_a\wedge e_b \mapsto X_{a b} \quad \forall a, b,
\end{aligned}
\]
where $e_a\wedge e_b =e_a\ot e_b- (-1)^{[a][b]}e_b\ot e_a$, and $X_{a b}$ are the basis elements of $\fg$ defined in Section \ref{sect:t-image}. 
We extend the map $j$ to a $\U(\fg)$-module isomorphism
$
\Upsilon:  S_{gr}(\wedge^2_{gr} V) \stackrel{\sim}\lra \U(\fg)
$
in the following way. 

For any $r\ge 2$, the symmetric group
$\Sym_r$ acts on $\U(\fg)^{\ot r}$ by permuting the factors with appropriate sign. 
Explicitly, the Chevalley generators $s_i$ of $\Sym_r$ act by 
$\id_{\U}^{\ot(i-1)}\ot\tau_{\U, \U}\ot \id_{\U}^{\ot(r-i-1)}$, 
where $\tau_{\U, \U}$ is defined by \eqref{eq:tau}. 
As above, let $\Sigma_{-1}\in \C\Sym_r$ be the Young symmetriser associated with the shape $(r)$.
Further, let $\mathfrak{M}_r: \U(\fg)^{\ot r}\lra \U(\fg)$ be the multiplication defined by 
$u_1\ot u_2\ot \dots \ot u_r\mapsto u_1u_2\dots u_r$, for any $u_1, \dots, u_r\in\U(\fg)$. 
Now we define $\Upsilon$ by requiring that it map any element $Y_1 Y_2\dots Y_r\in S_{gr}^r(\wedge^2_{gr} V)$ 
to $\frac{1}{r!}{\mathfrak M}_r\Sigma_{-1}\left(j(Y_1)\ot j(Y_2)\ot \dots\ot  j(Y_r)\right)$ in $\U(\fg)$ for any $r$. 

We now prove the theorem.  Clearly, $(S_{gr}(\wedge^2_{gr} V))^G=\sum_{k\ge 0}(S^k_{gr}(\wedge^2_{gr} V))^G$ is $\Z_+$-graded with 
$
(S_{gr}^0(\wedge^2_{gr} V))^G =\C$ and $(S_{gr}^1(\wedge^2_{gr} V))^G=0.
$
It follows from the first fundamental theorem of invariant theory for $G$ \cite{DLZ, LZ17} that the elements 
\[
\wh{T}[\ell]= \sum_{a_1,\dots, a_{\ell-1},  b} (-1)^{\sum_{i=1}^{\ell-1} [a_i]} j^{-1}(X^{a_1}_b)  j^{-1}( X^{a_2}_{a_1})  j^{-1}(X^{a_3}_{a_2})\dots j^{-1}(X^b_{a_{\ell-1}}), \quad \ell\ge 2,
\]
together with the identity generate $(S_{gr}(\wedge^2_{gr} V))^G$. 
Equation  \eqref{eq:FZ} states that
\[
\CF_\U(Z_\ell)= \sum_{a_1,\dots, a_{\ell-1},  b=1}^{m+2n} (-1)^{\sum_{i=1}^{\ell-1} [a_i]} X^{a_1}_b  X^{a_2}_{a_1} X^{a_3}_{a_2}\dots X^b_{a_{\ell-1}}. 
\]
Thus for any $\ell$, we have 
\beq\label{eq:Upsilon-1}
&\Upsilon^{-1}(\CF_\U(Z_\ell))=\wh{T}[\ell]\mod (S_{gr}^{\ell-1}(\wedge^2_{gr} V))^G. 
\eeq
In particular, $\Upsilon^{-1}(\CF_\U(Z_2))=\wh{T}[2]$,  and by Lemma \ref{lem:central}, 
$\Upsilon^{-1}(\CF_\U(Z_3))=\wh{T}[3]$.
It is also evident that 
\beq\label{eq:Upsilon-2}
\Upsilon^{-1}(\CF_\U(Z_k Z_\ell))=\wh{T}[k] \wh{T}[\ell]\mod (S_{gr}^{k+\ell-1}(\wedge^2_{gr} V))^G. 
\eeq
Using the relations \eqref{eq:Upsilon-1} and \eqref{eq:Upsilon-2},  we can show by induction on the degree of $(S_{gr}(\wedge^2_{gr} V))^G$ that  the identity and
$\Upsilon^{-1}(\CF_\U(Z_\ell))$ for all $\ell\ge 2$ generate  $(S_{gr}(\wedge^2_{gr} V))^G$.
Since $\Upsilon^{-1}(\U(\fg)^G)=(S_{gr}(\wedge^2_{gr} V))^G$, the identity and the elements   $\CF_\U(Z_\ell)$ for $\ell\ge 2$ together generate $\U(\fg)^G$. 
As the elements $Z_{2\ell}$ for all $\ell\ge 1$ generate $\Hom_{\AB(\sdim)}(0, 0)$ by Theorem \ref{thm:Hom01}(2), 
we obtain the first statement of the theorem. 

The second statement of the theorem now follows from Lemma \ref{lem:odd}.  This completes the proof.  
\end{proof}

In summary, we have 
\begin{cor}\label{cor:part-trace}  
For any $r, s$, the functor $\CF_\U$ sends $\Hom_{\AB(\sdim)}(r, s)$ to the subspace $\left(\U\ot_\C \Hom_\C\left(V^{\ot r}, V^{\ot s}\right)\right)^G$ of  $G$-invariants. In particular,  
\beq\label{eq:G-inv-U}
\CF_\U(\Hom_{\AB(\sdim)}(0, 0))=\U(\fg)^G. 
\eeq
Further,  if $\dim(V_{\bar 0})$ is odd or $0$, then 
$Z(\U(\fg))=\CF_\U(\Hom_{\AB(\sdim)}(0, 0))$.
\end{cor}

\begin{proof}
The first statement follows \eqref{eq:Hom-G-inv}, and \eqref{eq:G-inv-U} follows from Theorem \ref{eq:key-centre}. 
If $\dim(V_{\bar 0})$ is odd or $0$, then $\U(\fg)^\fg=\U(\fg)^G$ by Lemma \ref{lem:odd}. 
This completes the proof of the Corollary.
\end{proof}

To close this section, we complete the proof of Theorem \ref{thm:Hom01}. Let us recall some well known facts about central elements of the universal enveloping algebra $\U(\fso_m)$ of the orthogonal Lie algebra $\fso_m$.
\begin{remark}\label{rmk:so-Casimir}
By taking $V=\C^m$ to be purely even and equipped with the bilinear form $\omega(e^i, e^j)=\delta_{i j}$ for all $i, j$ in  Appendix \ref{sect:tensor-notation-1}, we have the orthogonal Lie algebra $\fso_m$, which is spanned by $J_{i j} = E^i_j-E^j_i$ for $i, j=1, 2, \dots, m$, where $E_j^i$ are the $m\times m$ matrix units.  
Denote by  $X_{i j}$ the image  of $J_{i j}$ in $\U(\fso_m)$. Then it follows immediately from the first fundamental theorem of invariant theory for the orthogonal group
 $O_m(\C)$ that as a subalgebra of $\U(\fso_m)$, $\U(\fso_m)^{O_m(\C)}$ is generated by 
\[
I(r)=\sum_{i_1, i_2, \dots, i_r} X_{i_1 i_2} X_{i_2 i_3} \dots X_{i_r, i_1},  \quad r\ge 2.
\]
It is well-known that the central elements $I(2j)$ for $j=1, 2, \dots, \begin{bmatrix}\frac{m}{2}\end{bmatrix}$ are algebraically independent  and generate $\U(\fso_m)^{O_m(\C)}$. 
These facts are easily proved following the line of reasoning suggested in \cite[\S11]{Bin}.

It is also straightforward to verify that $I(2k+1)$ may be explicitly written as a polynomial in the $I(2j)$. 
This also follows from the more general statement in Theorem \ref{thm:Z-odd}(2) (which is proved independently) by using  the functor $\CF_\U$ for $\fso_m$.

\end{remark}

\begin{proof}[Proof of Theorem \ref{thm:Hom01}(4)]
Retain the notation in Remark \ref{rmk:so-Casimir}. 
Let $\fg=\fso_m$ for any $m\ge 3$, and thus $\U(\fg)=\U(\fso_m)$. By using \eqref{eq:FZ},  we obtain
\[
\CF_{\U}(Z_{2j})=I(2j), \quad \forall j\ge 1. 
\]
Now the $I(2j)$ for $j=1, 2, \dots, \left[\frac{m}{2}\right]$ are algebraically independent. Thus the elements 
$Z_{2j}$ for all $j=1, 2, \dots, \left[\frac{m}{2}\right]$ must be algebraically independent, and this is true for any $m\ge 3$. 
This proves Theorem  \ref{thm:Hom01}(4).
\end{proof}

\subsection{Categorification of characteristic identities for $\mathfrak{so}_m$ and $\mathfrak{sp}_{2n}$}

In general, $\ker(\CF_\U)\ne 0$, and non-zero elements of $\ker(\CF_\U|_r)$ lead to interesting identities in $\U\ot\End_\C(V^{\ot r})$. Bracken and Green \cite[\S 4.10]{B}, 
\cite{BG, G} discovered certain `characteristic identities' for Lie algebras in $\U\ot\End_\C(V)$, which may be interpreted this way.   
Such characteristic identities have been widely studied in the mathematical physics literature (see \cite{IWD} for a brief review and further references).

Throughout this section, $V$ is either $\C^{m|0}$ or $\C^{0|2n}$, and thus $\fg$ is $\mathfrak{so}_m$ or $\mathfrak{sp}_{2n}$.  We denote $d=\dim(V)$.

\subsubsection{Characteristic identities for the orthogonal and symplectic Lie algebras}
We begin by stating precisely what is meant by the term ``characteristic identity'' \cite{B, BG, G} for $\fg$ (also see \cite{Go}). Let $E$ be the element of  $\U(\fg)\ot\End_\C(V)$  
given by $E=(\id_{\U(\fg)}\ot\mu)(t)$, where, as usual, $\mu$ denotes the $\fg$-action on $V$ and $t$ is the tempered Casimir element.. 
We say that $E$ satisfies a characteristic identity if there are explicitly known central elements $Q_i\in\Z(\U(\fg))$ such that 
\beq\label{eq:char-2}
E^d + Q_1E^{d-1} + Q_2E^{d-2}+\dots+Q_d=0. 
\eeq
Here the $Q_i$ in \eqref{eq:char-2} are known explicitly  in the sense that their images under the Harish-Chandra isomorphism are explicitly given, as explained in the proof of Theorem \ref{thm:char-id}. 

\begin{lem}\label{lem:even-coef}
The central elements $Q_i$ appearing in any characteristic identity \eqref{eq:char-2} all belong to $\U(\fg)^G$. 
\end{lem}
\begin{proof}
By Lemma \ref{lem:odd}, it is always true that $Q_i\in \U(\fg)^G$ for all $i$ if $\fg=\mathfrak{sp}_{2n}$ or $\mathfrak{so}_m$ for odd $m$. 
If $\fg=\mathfrak{so}_m$ with $m$ even, we let $\eta_0\in {\rm O}_m$ be such that $\det(\eta_0)=-1$. We can express $Q_i$ as $Q_i=Q_{i, +}+Q_{i, -}$ 
with $Q_{i, \pm}\in  Z(\U(\fg))$ such that $\eta_0.Q_{i, \pm}=\pm Q_{i, \pm}$, that is, $Q_{i, +}\in \U(\fg)^G$, and $Q_{i, -} \in \U(\fg)^{G, \sdet}$.
Since $\eta_0. E= E$ and $\eta_0.Q_i=Q_{i, +}-Q_{i, -}$, by combining \eqref{eq:char-2} with
\[
E^d + \eta_0.Q_1 E^{d-1} + \eta_0.Q_2 E^{d-2}+\dots+\eta_0.Q_d =0,  
\]
we obtain 
$
E^d + Q_{1, +}E^{d-1} + Q_{2, +}E^{d-2}+\dots+Q_{d, +}=0.
$
This proves the lemma. 
\end{proof}

\subsubsection{Categorification of characteristic identities}

The following theorem gives a categorical derivation of characteristic identities for  $\mathfrak{so}_m$ and $\mathfrak{sp}_{2n}$. 

\begin{theorem} \label{thm:char-id} Maintain the above notation, and write $d=\dim(V)$. There exist elements $\wh{Q}_i\in \Hom_{\AB(\sdim)}(0, 0)$, for $i=1, 2, \dots, d=\dim(V)$, such that the element 
\[
\wh\Q:=\BH^d + \wh{\Q}_1\BH^{d-1} + \wh{\Q}_2\BH^{d-2} +\dots + \wh{\Q}_{d}
\]
belongs to $\ker(\CF_\U|_1)$, where $\wh{\Q}_i = \wh{Q}_i\ot I$.
Furthermore, the characteristic identity \eqref{eq:char-2} may simply be stated in the present context as  
$
\CF_\U(\wh\Q)=0. 
$
\end{theorem}

\begin{proof} This proof is partly based on \cite{Go}. 
Denote by $M_\lambda$ the Verma module with highest weight $\lambda$, and by $L_\lambda$ the simple quotient module. Then $M_\lambda\ot V= \sum_{\lambda_i}M_{\lambda+\lambda_i}$, 
where the $\lambda_i$ for $i=1, ,2, \dots, d$ are the weights of a basis of weight vectors of $V$. Now 
$E$ takes the eigenvalue 
\[
e_i(\lambda)=\frac{1}{2}\left((\lambda+\lambda_i+2\rho, \lambda+\lambda_i)- (\lambda+2\rho, \lambda)-\chi_V(C)\right),
\]
in each submodule $M_{\lambda+\lambda_i}$, where $\chi_V(C)$ is the eigenvalue of $C$ in $V$.  Let 
\[
R_\lambda=\prod_{i=1}^{d}\left(E- e_i(\lambda)\right). 
\]
Then $(\mu_{M_\lambda}\ot\id)(R_\lambda)=0$ as an element of $\End_{\U}(M_\lambda\ot V)$. This implies 
\beq\label{eq:prod-form}
(\mu_{L_\lambda}\ot\id)(R_\lambda)=0, \quad \forall \lambda. 
\eeq

We expand $R_\lambda$ into a polynomial in $E$, obtaining  
\beq
R_\lambda=E^d + P_1(\lambda) E^{d-1} + P_2 (\lambda) E^{d-2}+\dots+ P_d (\lambda), 
\eeq
where $P_i(\lambda)$ are polynomials in the functions $e_j(\lambda)$ of $\lambda$, 
which are essentially the elementary symmetric functions in the $e_i(\lambda)$, with suitable signs; e.g., $P_1(\lambda)=-\sum_i e_i(\lambda)$, 
$P_2(\lambda)=\sum_{i\ne j} e_i(\lambda)e_j(\lambda)$,  and etc..  
Denote $E(\lambda)=(\mu_{L_\lambda}\ot\id)(E)$. Then \eqref{eq:prod-form} is equivalent to 
\beq\label{eq:char-1}
E(\lambda)^d + P_1(\lambda) E(\lambda)^{d-1} + P_2 (\lambda) E(\lambda)^{d-2}+\dots+ P_d (\lambda) =0, \quad \forall \lambda.  
\eeq

Using the fact that the set of weights of $V$ is invariant 
with respect to the usual action of the Weyl group $W$ of $\fg$, one easily shows, as in \cite{Go} that the 
$P_i(\lambda)$ are all invariant under the dot-action of $W$. Hence by Harish-Chandra's isomorphism, there exist elements $Q_i\in Z(\U)$ such that 
$\mu_{M_\lambda}(Q_i)=P_i(\lambda)$ for all $i$.  We define 
\beq
R=E^d + Q_1 E^{d-1} + Q_2 E^{d-2}+\dots+ Q_d,  
\eeq
which is equal to the left hand side of \eqref{eq:char-2}.  Now equation \eqref{eq:char-1} becomes 
\beq
(\mu_{L_\lambda}\ot\id)(R)=0, \quad \forall \lambda.
\eeq

Note that $R$ is a square matrix with entries in $\U(\fg)$. The above equation states that all entries of $R$ vanish in the irreducible representations $\mu_{L_\lambda}$ for all $\lambda$. This implies that $R=0$, which is equation \eqref{eq:char-2}. 

By Lemma \ref{lem:even-coef} and Corollary \ref{cor:part-trace}, there exist $\wh{Q}_i\in \Hom_{\AB(\sdim)}(0, 0)$ for $1\le i\le d$ such that $Q_i= \CF_\U(\wh{Q}_i)$,  and hence $Q_i\ot \id_V= \CF_\U(\wh{\Q}_i)$. Since $E=\CF_\U(\BH)$, the left hand side of  \eqref{eq:char-2} is the image of the element $\wh Q$ given in Theorem \ref{thm:char-id} under $\CF_\U$. This proves the theorem. 
\end{proof}

\begin{example}\label{eg:sp2} When $V=\C^{0|2}$, we are dealing with the Lie algebra $\mathfrak{sp}_2$. 
Then 
\[
\wh\Q=\BH^2 - 2\BH + \frac{1}{2}Z_2\ot I 
\]
belongs to the kernel of $\CF_\U|_1$, and  
the characteristic identity in this case is given by $\CF_\U(\wh\Q)=0$.

To be more explicit, let $X, Y, T$ be the standard generators of $\mathfrak{sp}_2=\fsl_2$ with 
\[
[T, X]=2X,  \quad [T, Y]=-2Y, \quad [X, Y]=T.
\]  
Then  the quadratic Casimir operator \eqref{eq:Casimir-formula} is given by 
\beq\label{eq:C-sp2}
C= - 2 \left(\frac{H^2}{2}+ X Y + Y X\right)
\eeq
in the current notation. As $V=\C^{0|2}$ has highest weight $1$, the eigenvalue of $C$ in $V$ is equal to $-2\left(\frac{1}{2}+1\right)= - 3= \sdim-1$.  
The standard generators $X,Y,T$ act on $V$ with respect to the standard basis as the matrices $\begin{pmatrix} 0 & 1\\0 & 0 \end{pmatrix}$,
$\begin{pmatrix} 0 & 0\\1 & 0 \end{pmatrix}$ and $\begin{pmatrix} 1 & 0\\0 & -1 \end{pmatrix}$ respectively.

The corresponding tempered Casimir element is  $t=- 2\left( \frac{T\ot T}{2}+ X\ot Y + Y\ot X\right)$. Hence
\[
\begin{aligned}
\CF_\U(\BH)&=(\id_\U\ot\mu)(t)\\
&=-2\left( \frac{1}{2}T\ot  \begin{pmatrix} 1 & 0\\0 & -1 \end{pmatrix}    + X\ot \begin{pmatrix} 0 & 0\\1 & 0 \end{pmatrix}
+Y\ot \begin{pmatrix} 0 & 1\\0 & 0 \end{pmatrix}\right)\\
&=-2 \begin{pmatrix} \frac{T^2}{2} & Y\\  X & -  \frac{T^2}{2} \end{pmatrix}, \text{ and } \\
\CF_\U(Z_2)&=\text{str} (\CF_\U(\BH)^2)=2C.
\end{aligned}
\]
\end{example}

\subsubsection{Comments on characteristic identities for Lie superalgebras and quantum groups}\label{sect:super-char-id}
There are discussions of characteristic identities for Lie superalgebras in the literature (see, e.g., \cite{JG}), but largely in the form of \eqref{eq:char-1}
for any finite dimensional simple module $L_\lambda$ (that is, replacing $M_\lambda$ by $L_\lambda$). 
It appears to be difficul to generalise the method of \cite{Go} to lift \eqref{eq:char-1} for all $\lambda$ to the identity \eqref{eq:char-2} 
in $\U\ot\End_\C(V)$ for $\osp_{m|2n}$ and $\gl_{m|n}$ of general $m, n$. 
The problem lies in the fact that the $P_i(\lambda)$ for atypical $\lambda$ are not invariant 
under translation by scalar multiples of atypical roots $\gamma$, that is, 
$P_i(\lambda)\ne P_i(\lambda+c\gamma)$ in general for nonzero $c\in\C$ and isotropic root $\gamma$ such that $(\lambda+\rho, \gamma)=0$. By the generalised Harish-Chandra isomorphism for Lie superalgebras,  there exist no central elements 
of $\U$ whose eigenvalues are $P_i(\lambda)$ in $M_\lambda$ for all $\lambda$.  This problem does not exist for  $\osp_{1|2n}$, as this Lie superalgebra has no atypical weights.  

However, any non-zero element $\wh\Q\in \ker(\CF_\U|_1)$ would lead to an identity
$\CF_\U(\wh\Q)=0$  in $\U\ot\End_\C(V)$, which may be considered a generalisation of the characteristic identities above. 
For the purpose of constructing such identities, we need only consider the generators of $\ker(\CF_\U|_1)$ as a $2$-sided ideal in $\Hom_{\AB(\sdim)}(1, 1)$.

Characteristic identities for quantum groups were investigated in \cite{GZB} (also see \cite{ZGB}). 
It is quite straightforward to use the restricted coloured tangle category (with one pole only) introduced
 in \cite[\S3.1, \S 3.2]{ILZ} to give a categorical treatment of the quantum characteristic identities analogous to what has been presented here.

\section{Application to category $\mathcal O$ of $\fsp_2$}\label{sect:sp2}

Throughout this section, we assume that $V=\C^{0|2}$ and $\omega=\begin{pmatrix}0 &1\\ -1 & 0\end{pmatrix}$. Thus $\osp(V; \omega)$ is the Lie algebra $\fsp_2(\C)$. 

We investigate the functor of Theorem \ref{thm:a-funct} in some depth in this special case. Clearly $\sdim=-2$. The following result shows that the functor
$\CF_M$ factors through the Temperley-Lieb category in this case.
\begin{lemma} 
For any  $\fsp_2$-module $M$,  the functor $
\CF_M: \AB(-2)\lra \CT_M(V)$ given in Theorem \ref{thm:a-funct} factors through $\ATL(-2)$, that is, 
there exists a unique functor $\CF^{TL}_M: \ATL(-2) \lra \CT_M(V)$ such that the following diagram commutes, 
\[
\begin{tikzcd}
\AB(-2) \arrow[d] \arrow[r, "\CF_M"] & \CT_M(V)\\
\ATL(-2) \arrow[ur, "\CF^{TL}_M"'pos=0.43]. 
\end{tikzcd}
\]
\end{lemma}
\begin{proof} 
Note that $V\ot V=\C^{0|2}\ot\C^{0|2}=L_s \oplus L_a$ as $\fsp_2$-module, where $L_s$ and $L_a$ are simple submodules of dimensions $1$ and $3$ respectively. Thus the permutation $\tau: V\ot V\lra V\ot V$ is given by 
$\tau = -P_a + P_s = 1 + 2 P_s$.  As $\sdim(V)=-2$, the map $e$ in \eqref{eq:e} is given by $e=-2 P_s$ in this case, and hence
$
\tau = 1- e.
$
This shows that $\CF(\Theta)=0$ because $\delta=\sdim(V)=-2$, where $\Theta$ is the element of $\AB(-2)(2,2)$ depicted in Fig. 14. By definition, this shows 
that the functor $\CF_M$ factors through $\ATL(-2)$.
\end{proof}

\begin{lemma} Maintain the notation above and fix $\lambda\in\C$. Assume that the $\fsp_2$-module $M$ is either the Verma module $M_\lambda$ with highest weight $\lambda$ or the simple module $L_\lambda$. Then the functor 
$\CF^{TL}_M: \ATL(-2) \lra \CT_M(V)$ factors through $\TLBC(-2, \lambda)$, that is there exists a unique functor 
$\CF^{TLB}_\lambda: \TLBC(-2, \lambda)\lra \CT_M(V)$
such that the following diagram commutes, 
\[
\begin{tikzcd}
\ATL(-2) \arrow[d] \arrow[r, "\CF^{TL}_M"] & \CT_M(V)\\
\TLBC(-2, \lambda) \arrow[ur, "\CF^{TLB}_\lambda"'pos=0.43], 
\end{tikzcd}
\]
where the vertical functor is the quotient functor (see Definition \ref{def:TL}).  
\end{lemma}
\begin{proof}
We have $\CF^{TL}_M(Z_2)=2C: M\lra M$, where $C$ is the quadratic Casimir given in \eqref{eq:C-sp2}, which acts on $M=M_\lambda$ and {\it a fortiori} on $L_\lambda$ 
as multiplication by the scalar $\chi_\lambda(C)= - \lambda(\lambda +2)$. Hence for $\delta=-2$, 
\[
\CF^{TL}_M\left(\frac{Z_2}{\delta}\right)= \lambda(\lambda +2) = -\lambda\left(\frac{\delta-2}{2}-\lambda\right).
\]
This shows that $\CF^{TL}_M$ factors through $\TLBC(-2, \lambda)$.
\end{proof}

Let us denote $\Z_{<-1}=\{-2, -3, -4, \dots\}$.  For any given $\lambda\in\C\backslash \Z_{<-1}$, the Verma module $M_\lambda$ is projective in category ${\CO}$ of $\fsp_2$. 

The next theorem is an analogue in our context of \cite[Thm. 4.9]{ILZ}.

\begin{theorem} Maintain the above notation and let $M_\lambda$ be the Verma module for $\fsp_2$ with highest weight $\lambda\in\C\backslash \Z_{<-1}$. Then 
the functor $\CF^{TLB}_\lambda: \TLBC(-2, \lambda)\lra \CT_{M_\lambda}(V)$
is an isomorphism of categories.  
\end{theorem}
\begin{proof} 
It is evident that the restriction of $\CF^{TLB}_\lambda$  to objects is an isomorphism. Since $M_\lambda$ with $\lambda\in\C\backslash \Z_{<-1}$ is projective in category ${\CO}$, we have 
\[
\begin{aligned}
\dim \Hom_{\fsp_2}(M_\lambda \ot V^{\ot r}, M_\lambda \ot V^{\ot s})
&\cong \dim \Hom_{\fsp_2}(M_\lambda, M_\lambda \ot V^{\ot (r+s)}) \\
&\cong \dim (V^{\ot (r+s)})_0, \quad \forall r, s, 
\end{aligned}
\]
where $(V^{\ot (r+s)})_0$ is the $0$-weight subspace of $V^{\ot (r+s)}$, which has dimension $\begin{pmatrix}2N \\ N\end{pmatrix}$ if $r+s=2N$ is even, and is $0$ otherwise. 
 It follows  from Theorem \ref{thm:ATL-dim} that for all $r, s$, 
\[
\dim \Hom_{\TLBC(-2, \lambda)}(r, s) =\dim \Hom_{\fsp_2}(M_\lambda \ot V^{\ot r}, M_\lambda \ot V^{\ot s}).
\]
Hence it suffices to show that the restriction of $\CF^{TLB}_\lambda$  to morphisms is injective.  
Only the case with $r+s=2N$ requires proof, and we will adapt the main idea in the treatment of $\U_q(\fsl_2)$ representations in \cite{ILZ} to the present context.  

Write $\CW(r, 2N-r)=\Hom_{\TLBC(-2, \lambda)}(r, 2N-r)$ for all $r$, and $\CW(2N)=\CW(0, 2N)$.
Let $E^0_{2N}$ be the subalgebra of $\CW(2N, 2N)$ spanned by diagrams without connectors, which is isomorphic to the Temperley-Lieb algebra $TL_{2N}(-2)$. Note that 
$\CF^{TLB}_\lambda(E^0_{2N})=\CF(E^0_{2N}) \cong E^0_{2N}$ as algebra. Furthermore, $\CW(2N)$ is a module for this algebra.  
We want to show that $\CF^{TLB}_\lambda(\CW(2N))\cong \CW(2N)$ as $TL_{2N}(-2)$-module.

Consider images of the standard affine Temperley-Lieb diagrams as shown in Figure \ref{fig:ATL2N}, which will be called standard 
affine Temperley-Lieb diagrams of type $B$. They form a basis of $\CW(2N)$. As in the situation in the proof of Theorem \ref{thm:ATL-dim},  
the standard $(0, 2N)$-diagrams with at most $i\le N$ connectors span a $TL_{2N}(-2)$-submodule $F_i\CW(2N)$ of $\CW(2N)$, and 
$\CW_{2i}(2N) = \frac{F_i \CW(2N)}{F_{i-1}\CW(2N)}$ is a simple $TL_{2N}(-2)$-module.  
Hence $\CF^{TLB}_\lambda(\CW_{2i}(2N)) = \frac{\CF^{TLB}_\lambda(F_i \CW(2N))}{\CF^{TLB}_\lambda(F_{i-1}\CW(2N))}\cong \CW_{2i}(2N)$ or $0$.  We show that it is non-zero for all $i$. 

Let $\BD_t$ be the standard $(0, 2N)$-diagram with $t$ connectors as depicted below. 
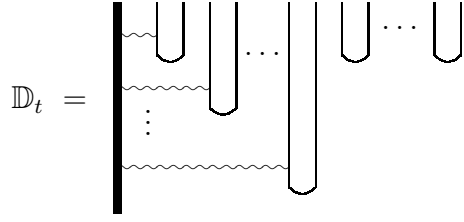
\[
\begin{picture}(180, 80)(-40, 20)
\put(-40, 60){$\BD_t\  =$}
{
\linethickness{1mm}
\put(0, 20){\line(0, 1){80}}
}
\put(0, 90){\uwave{\hspace{5mm}}}

\put(15, 100){\line(0, -1){20}}
\put(25, 100){\line(0, -1){20}}
\qbezier(15, 80)(20, 75)(25, 80)

\put(0, 70){\uwave{\hspace{12mm}}}

\put(35, 100){\line(0, -1){40}}
\put(45, 100){\line(0, -1){40}}
\qbezier(35, 60)(40, 55)(45, 60)

\put(0, 40){\uwave{\hspace{23mm}}}
\put(65, 100){\line(0, -1){70}}
\put(75, 100){\line(0, -1){70}}
\qbezier(65, 30)(70, 25)(75, 30)

\put(48, 80){$\dots$}
\put(10, 50){$\vdots$}

\put(85, 100){\line(0, -1){20}}
\put(95, 100){\line(0, -1){20}}
\qbezier(85, 80)(90, 75)(95, 80)

\put(100, 90){$\dots$}

\put(120, 100){\line(0, -1){20}}
\put(130, 100){\line(0, -1){20}}
\qbezier(120, 80)(125, 75)(130, 80)


\put(135, 25){.}
\end{picture}
\]
Then $\BD_t$ belongs to $F_t\CW(2N)$, but is not in $F_{t-1}\CW(2N)$. 
Postmultiplying $\I_{2t}\ot \underbrace{\cap \dots \cap}_{N-t}$ with $\BD_t$, we obtain the standard $(0, 2t)$-diagram $\wt\BD_t$ shown below
up to a scalar factor. 
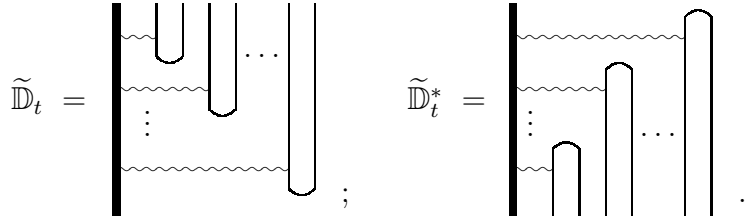
\[
\begin{picture}(150, 80)(-40, 20)
\put(-40, 60){$\wt\BD_t\  =$}
{
\linethickness{1mm}
\put(0, 20){\line(0, 1){80}}
}
\put(0, 90){\uwave{\hspace{5mm}}}

\put(15, 100){\line(0, -1){20}}
\put(25, 100){\line(0, -1){20}}
\qbezier(15, 80)(20, 75)(25, 80)

\put(0, 70){\uwave{\hspace{12mm}}}

\put(35, 100){\line(0, -1){40}}
\put(45, 100){\line(0, -1){40}}
\qbezier(35, 60)(40, 55)(45, 60)

\put(0, 40){\uwave{\hspace{23mm}}}
\put(65, 100){\line(0, -1){70}}
\put(75, 100){\line(0, -1){70}}
\qbezier(65, 30)(70, 25)(75, 30)

\put(48, 80){$\dots$}
\put(10, 50){$\vdots$}

\put(85, 25){;}
\end{picture}
\begin{picture}(150, 80)(-40, 20)
\put(-40, 60){$\wt\BD_t^*\  =$}
{
\linethickness{1mm}
\put(0, 20){\line(0, 1){80}}
}

\put(0, 90){\uwave{\hspace{23mm}}}
\put(0, 70){\uwave{\hspace{12mm}}}
\put(0, 40){\uwave{\hspace{5mm}}}

\put(15, 20){\line(0, 1){25}}
\put(25, 20){\line(0, 1){25}}
\qbezier(15, 45)(20, 50)(25, 45)

\put(35, 20){\line(0, 1){55}}
\put(45, 20){\line(0, 1){55}}
\qbezier(35, 75)(40, 80)(45, 75)

\put(65, 20){\line(0, 1){75}}
\put(75, 20){\line(0, 1){75}}
\qbezier(65, 95)(70, 100)(75, 95)

\put(48, 50){$\dots$}
\put(5, 50){$\vdots$}

\put(85, 25){.}
\end{picture}
\]
Now premultiply the tensor product $\wt{\BD}_t\ot I_{2t}$ with the $(0, 4t)$-diagram 
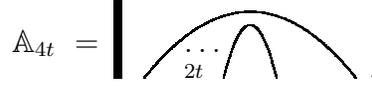
\[
\begin{picture}(150, 40)(-40, -5)
\put(-40, 10){$\A_{4t}\ =$}
{
\linethickness{1mm}
\put(0, 0){\line(0, 1){30}}
}

\qbezier(10,0)(50, 50)(90, 0)
\qbezier(40,0)(50, 40)(60, 0)

\put(25, 10){$\dots$}
\put(25, 0){\tiny$2t$}

\put(95, 0){.}
\end{picture}
\] 
One obtains essentially (up to a non-zero scalar) the $(2t, 0)$-diagram $\wt{\BD}_t^*$ shown above.  
Denote by $m_+$ the highest weight vector of $M_\lambda$.  It is easy to show that 
\beq
\CF^{TLB}_\lambda(\wt{\BD}_t^*)(m_+\ot \underbrace{ v_{-1}\ot\dots\ot v_{-1}}_{2t})= (-2 Y)^t m_+\ne 0. 
\eeq

Let $\BD_{<t}$ be an element in $F_{t-1}\CW(2N)$.  We perform the same operations as above on $\BD_{<t}$ to obtain an element $\wt\BD_{<t}^* \in \Hom_{\TLBC(-2, \lambda)}(2t, 0)$, which can be expressed as a linear combination of $(2t, 0)$-diagrams with no more than $t-1$ connectors. Thus each $(2t, 0)$-diagram in $\wt\BD_{<t}^*$ with non-zero coefficient must have a 
component of the form $\{\cap\}$ which is not connected to the pole via a connector. 
Note that for all $i+j=2t-2$, we have 
$\id_{M_\lambda}\ot \id_V^{\ot i} \ot \CF(\cap)\ot \id_V^{\ot j} (m_+\ot \underbrace{ v_{-1}\ot\dots\ot v_{-1}}_{2t})=0$ because $\omega(v_{-1}, v_{-1})=0$. Thus 
\beq
\CF^{TLB}_\lambda(\wt\BD_{<t}^*)(m_+\ot \underbrace{ v_{-1}\ot\dots\ot v_{-1}}_{2t})=0. 
\eeq

Comparing the two equations above, we conclude that 
$
\CF^{TLB}_\lambda(\CW_{2t}(2N)) \ne 0
$
for all $t$, and hence by cellular theory \cite{GL96}, $\CF^{TLB}_\lambda(\CW_{2t}(2N))$  $\cong \CW_{2t}(2N)
$.
This shows that $\CF^{TLB}_\lambda$ is injective when restricted to morphisms, completing the proof of the theorem. 
\end{proof} 

\section{Comments and outlook}

We point out directions for further research, and make more comments on relationship of the present paper to other work. 

\subsection{Relationship to other work}\label{sect:comm}
\subsubsection{Comments on other work} In contrast to the categories in \cite{RSo, Betal}, both of which are monoidal categories, our category $\AB(\delta)$ was designed to only admit a tensor product functor $\ot: \AB(\delta) \times \CB(\delta)\lra  \AB(\delta)$ with the usual Brauer category $\CB(\delta)$, but not a natural monoidal structure. The reason for this is that we want to use $\AB(\delta)$ to study the category $\CT_M(V)$ of $\fg$-modules with objects $M\ot V^{\ot r}$, which only has a natural tensor product  $\ot: \CT_M(V)\times \CT(V) \lra \CT_M(V)$ with the category $\CT(V)$ of the modules $V^{\ot r}$.  We do not allow for cups and caps which would close the pole into a loop as $M$ may be infinite dimensional. 

In order to compare our category with that of \cite{RSo}, we consider the respective endomorphism algebras. 
The endomorphism algebras of the category in \cite{RSo} are the Nazarov--Wenzl algebra
 at different degrees defined in \cite[Definition 2.1]{ES}  with generators and relations 
 $(VW.1)$ -- $(VW.8)$, which involve an infinite family of 
 arbitrary parameters $w_k$ for all non-negative integers $k$, see $(VW.3)$ and $(VW.4)$ in \cite[Definition 2.1]{ES} . 
 
 An inspection of the structures of the algebras under consideration suggests that we consider quotient algebras of $\Hom_{\AB(\delta)}(r, r)$ defined as follows. Given any scalars $z_{2\ell}$ for $\ell=0, 1, 2, \dots$, we let $\mathfrak{M}_z$ be the maximal ideal  of $\Hom_{\AB(\delta)}(0, 0)$ generated by $Z_{2\ell}-z_{2\ell}$ for all $\ell$. 
Then $\Hom_{\AB(\delta)}(0, 0)/\mathfrak{M}_z\cong K$, as all $Z_{2\ell}$ are algebraically independent, and all $Z_{2j+1}$ ($Z_1=0$) can be expressed as polynomials in the elements $Z_{2\ell}$, as shown in Theorem \ref{thm:Z-odd}. Denote by $z_{2\ell+1}$ the image of  $Z_{2\ell+1}$ for all $\ell$. Then in particular, we have $z_1=0$. 
 
The following fact is immediate by using the description of $\Hom_{\AB(\delta)}(r, r)$ given in Theorem \ref{lem:JM}.
\begin{theorem} \label{thm:NW-alg}  Given any scalars $z_{2\ell}\in\C$ for $\ell\ge 0$ with $z_0=\delta$, and let $z_{2\ell+1}\in\C$ be defined as above. There is an algebra homomorphism from $\Hom_{\AB(\delta)}(r, r)/\mathfrak{M}_z\ot I_r$ to 
the Nazarov--Wenzl algebra of degree $r$ (in the notation of \cite[Definition 2.1]{ES}) with parameters 
$w_0=\delta$ and $w_k= \sum_{i=1}^k \begin{pmatrix}k\\ i\end{pmatrix}\left(\frac{1-\delta}2\right)^i  z_i$ for $k\ge 1$, defined by 
\beq
\iota(s_i)\mapsto s_i, \quad \iota(e_i)\mapsto e_i, \quad 
\Theta_i+\frac{1-\delta}2 \I_r\mapsto y_i.
\eeq
\end{theorem}
This indicates that the affine Brauer category of \cite{RSo} is a quotient of $\AB(\delta)$.

The affine VW supercategory in \cite{Betal} was built for studying 
$\CT_M(V)$ for the periplectic Lie superalgebra $\mathfrak{p}(n)$.  The cups and caps in the 
affine VW supercategory 
(analogues of $\cup$ and $\cap$ in our $\AB(\delta)$) were designed to account for  
the odd invariant non-degenerate bilinear form of the natural  module $V$. 
This makes the supercategory different from $\AB(\delta)$ at a fundamental level. 

However, an inspection reveals that the endomorphism algebras of the affine VW supercategory 
have similar general features as those of $\AB(\delta)$ at $\delta=0$, ignoring the differences 
in signs in the defining relations. 
It would be interesting to follow \cite{Betal} to construct a priori an affine VW category of dotted Brauer diagrams, adapted to the study of $\osp_{m|2n}(\C)$ modules. 
An important question would then be whether it is equivalent to the category introduced in this paper.

\subsubsection*{Quantum versus classical}
 Heuristically \cite{ILZ} may be regarded as the quantum version of the present work. 
 The diagrammatic affine Brauer category studied here may be considered as 
 the ``semi-classical limit'' of a polar BMW category \cite{ILZ}. 
 We identify over and under crossings of thin strings in polar braid diagrams, 
 and replace the combination of polar braid diagrams on the left hand side of the following relation
\[   
\setlength{\unitlength}{0.2mm}
\begin{picture}(50, 50)(-30,40)
{
\linethickness{.7mm}
\put(-15, 30){\line(0, 1){18}}
\put(-15, 52){\line(0, 1){38}}
}
\qbezier(0, 90)(-7, 80)(-12, 72)
\qbezier(-18, 68)(-25, 60)(-15, 50)
\qbezier(-15, 50)(-15, 50)(0, 30)
\put(5, 50){$-$}
\end{picture}
\setlength{\unitlength}{0.2mm}
\begin{picture}(100, 50)(-30,40)

{
\linethickness{.7mm}
\put(-15, 30){\line(0, 1){60}}
}
\put(5, 30){\line(0, 1){60}}
\put(25, 50){$\rightsquigarrow$}
\end{picture}
\setlength{\unitlength}{0.2mm}
\begin{picture}(45, 50)(-30,40)
\put(-35, 50){$\hbar$}
{
\linethickness{.7mm}
\put(-15, 30){\line(0, 1){60}}
}
\put(-14, 60){\uwave{\hspace{4mm}}}
\put(8, 30){\line(0, 1){60}}
\put(20, 50){$+\  o(\hbar^2)$,}
\end{picture}
\]   
by the combination of affine Brauer diagrams on the right side.
Then the affine Brauer diagram 
\begin{picture}(22, 20)(0,3)
{
\linethickness{.7mm}
\put(5, 0){\line(0, 1){20}}
}
\put(3, 12){\uwave{\hspace{5mm}}}

\put(15, 0){\line(0, 1){20}}
\put(18, 5){,}
\end{picture}
which is the new ingredient needed for the development of our diagrammatic affine Brauer category,  can be considered as the $\hbar\to 0$ limit of $\frac{1}{\hbar}$ times the left hand side of the above relation. 
It is such heuristics that led us to the investigation in this paper. We point out that heuristics of a similar nature also appeared in \cite[\S5]{Betal}. 

\subsection{Outlook}\label{sect:conclusion}

The present work points to several directions for further exploration.

\subsubsection*{Oriented affine Brauer category}

Recall from \cite[\S 4]{LZ23} the oriented Brauer category $\OB(\delta)$, which is the natural diagrammatic category required for studying the representations of the general linear superalgebra $\gl(V)$ of the type 
$V^{\ot (\varepsilon_1, \varepsilon_2, \dots, \varepsilon_r)}:=V^{\varepsilon_1}\ot V^{\varepsilon_r}\ot  \dots \ot V^{\varepsilon_r}$,  for all $r\in\N$, with $\varepsilon_i=\pm$, where $V^+=V$ is the natural module, and $V^-=V^*$ is the dual module.

The objects of $\OB(\delta)$ are sequences $(\varepsilon_1, \varepsilon_2, \dots, \varepsilon_r)$ for $r\in\N$ with $\varepsilon_i=\pm$, and morphisms are spanned by oriented Brauer diagrams, that is, Brauer diagrams with oriented strings.   An oriented Brauer $(k, \ell)$-diagram belongs to 
$\Hom_{\OB(\delta)}(s(\Gamma), t(\Gamma))$ with $s(\Gamma)=(\varepsilon_1,   \varepsilon_2, ...,
\varepsilon_k)$ and target $t(\Gamma)=(\varepsilon'_1,  \varepsilon'_2, ...,
 \varepsilon'_\ell)$, if for each $p=1, \dots, k$, the $p$-th bottom vertex is an end point of a string with the arrow pointing outward (resp. inward) when $\varepsilon_p=+$ (resp. $\varepsilon_p=-$), and for each $q=1, \dots, \ell$,
the $q$-th top vertex is an end point of a string with the arrow pointing inward (resp. outward)  when $\varepsilon'_q=+$ (resp. $\varepsilon'_q=-$). Composition of morphisms preserves orientation. 

Define the oriented $H$-diagrams in $\OB(\delta)$ by Figure \ref{fig:H-oriented}. 

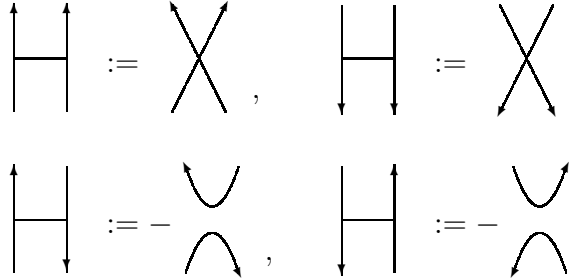
\begin{figure}[h]
\begin{picture}(120, 40)(0, 0)

\put(0, 0){\line(0, 1){40}}
\put(0, 20){\line(1, 0){20}}

\put(20, 0){\line(0, 1){40}}


\put(0, 40){\vector(0, 1){1}}
\put(20, 40){\vector(0, 1){1}}

\put(35, 15){$:=\ $}

\qbezier(60, 40)(80, 0)(80, 0)
\qbezier(60, 0)(80, 40)(80, 40)

\put(60, 40){\vector(-1, 2){1}}
\put(80, 40){\vector(1, 2){1}}

\put(90, 5){,}

\end{picture}
\begin{picture}(80, 40)(0, 0)

\put(0, 0){\line(0, 1){40}}
\put(0, 20){\line(1, 0){20}}

\put(20, 0){\line(0, 1){40}}


\put(0, 0){\vector(0, -1){1}}
\put(20, 0){\vector(0, -1){1}}

\put(35, 15){$:=\ $}

\qbezier(60, 40)(80, 0)(80, 0)
\qbezier(60, 0)(80, 40)(80, 40)

\put(80, 0){\vector(1, -2){1}}
\put(60, 0){\vector(-1, -2){1}}

\end{picture}

\begin{picture}(120, 60)(0, 0)

\put(0, 0){\line(0, 1){40}}
\put(0, 20){\line(1, 0){20}}

\put(20, 0){\line(0, 1){40}}


\put(0, 40){\vector(0, 1){1}}
\put(20, 2){\vector(0, -1){1}}

\put(35, 15){$:= -  $}


\qbezier(65, 40)(75, 10)(85, 40)
\qbezier(65, 0)(75, 30)(85, 0)

\put(65, 40){\vector(-1, 2){1}}
\put(85, 0){\vector(1, -2){1}}
\put(95, 5){,}
\end{picture}
\begin{picture}(80, 60)(0, 0)

\put(0, 0){\line(0, 1){40}}
\put(0, 20){\line(1, 0){20}}

\put(20, 0){\line(0, 1){40}}


\put(0, 0){\vector(0, -1){1}}
\put(20, 40){\vector(0, 1){1}}

\put(35, 15){$:= -  $}


\qbezier(65, 40)(75, 10)(85, 40)
\qbezier(65, 0)(75, 30)(85, 0)

\put(65, 0){\vector(-1, -2){1}}
\put(85, 40){\vector(1, 2){1}}
\end{picture}

\caption{Oriented $H$-diagrams}
\label{fig:H-oriented}
\end{figure}
%

There is a natural oriented version of the diagrammatic affine Brauer category $\AB(\delta)$.  To construct such a category, we introduce two oriented analogues of $\BH$ as depicted in Figure \ref{fig:bbH-or}.  
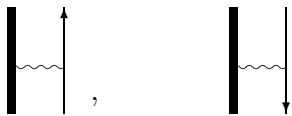
\begin{figure}[h]
\begin{picture}(80, 40)(0,0)
{
\linethickness{1mm}
\put(35, 0){\line(0, 1){40}}
}
\put(35, 20){\uwave{\hspace{7mm}}}

\put(55, 0){\vector(0, 1){40}}

\put(65, 5){, }

\end{picture} 
\begin{picture}(80, 40)(0,0)
{
\linethickness{1mm}
\put(35, 0){\line(0, 1){40}}
}
\put(35, 20){\uwave{\hspace{7mm}}}

\put(55, 40){\vector(0, -1){40}}


\end{picture} 
\caption{Oriented $\BH$-diagrams}
\label{fig:bbH-or}
\end{figure}  
We can then build oriented affine Brauer diagrams as in Figure \ref{fig:AffD} with oriented Brauer diagrams and oriented $\BH$-diagrams, which obey
\begin{enumerate}[i)]
\item oriented versions of four-term-relations,  and 
\item skew symmetry of oriented $\BH$-diagrams.
\end{enumerate}
The oriented four-term-relations involve oriented analogues of $\BH_{12}$, which are the juxtapositions of a pole with any of the oriented $H$-diagrams shown in Figure \ref{fig:H-oriented}. Skew symmetry of oriented $\BH$-diagrams is shown in Figure \ref{fig:OH-skew}.  Note that the oriented $H$-diagrams in Figure \ref{fig:H-oriented} also have similar skew symmetry properties.

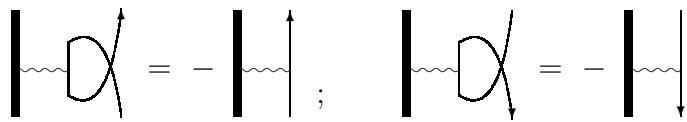
\begin{figure}[h]
\begin{picture}(90, 40)(30, 0)
{
\linethickness{1mm}
\put(35, 0){\line(0, 1){40}}
}
\put(35, 20){\uwave{\hspace{7mm}}}

\put(55, 8){\line(0, 1){20}}

\qbezier(55, 28)(70, 38)(75, 0)
\qbezier(55, 8)(70, -2)(75, 40)
\put(75, 40){\vector(0, 1){1}}

\put(85, 15){$= \ - $}
\end{picture}
\begin{picture}(50, 40)(40,0)
{
\linethickness{1mm}
\put(35, 0){\line(0, 1){40}}
}
\put(35, 20){\uwave{\hspace{7mm}}}

\put(55, 0){\vector(0, 1){40}}

\put(65, 5){; }
\end{picture} 
\begin{picture}(90, 40)(30, 0)
{
\linethickness{1mm}
\put(35, 0){\line(0, 1){40}}
}
\put(35, 20){\uwave{\hspace{7mm}}}

\put(55, 8){\line(0, 1){20}}

\qbezier(55, 28)(70, 38)(75, 0)
\qbezier(55, 8)(70, -2)(75, 40)
\put(75, 0){\vector(0, -1){1}}

\put(85, 15){$= \ - $}
\end{picture}
\begin{picture}(30, 40)(40,0)
{
\linethickness{1mm}
\put(35, 0){\line(0, 1){40}}
}
\put(35, 20){\uwave{\hspace{7mm}}}

\put(55, 40){\vector(0, -1){40}}


\end{picture} 
\caption{Skew symmetry of oriented $\BH$-diagrams}
\label{fig:OH-skew}
\end{figure}  
Now the oriented diagrammatic affine Brauer category $\OAB(\delta)$ has the  same objects as $\OB(\delta)$, and morphisms being linear combinations of oriented affine Brauer diagrams. 

There is a functor from $\OAB(\sdim(V))$ to the category of $\gl(V)$-modules with objects $M\ot V^{\ot (\varepsilon_1, \varepsilon_2, \dots, \varepsilon_r)}$ for an arbitrary but fixed $\gl(V)$-module $M$.  In particular, the work \cite{RS} can be cast into this framework.  The functor is full for any of the special class of typical Kac modules $M$ (which are projectiv in the category of finite dimensional $\gl(V)$-modules) studied in \cite{RS}.  

\subsubsection*{Enhanced affine Brauer category}
As we pointed out in Remark \ref{rem:funct-F},  while the Brauer category enables us to describe all the morphisms among the tensor powers of the natural module $V$ 
for the orthosymplectic supergroup $\OSp(V; \omega)$, it does not yield all the $\osp(V; \omega)$ morphisms among the tensor powers of the natural module. 
This problem is present even for the orthogonal Lie algebra $\fso_m(\C)$. We solved the problem in the $\fso_m(\C)$ case by introducing an enhanced Brauer category 
in \cite{LZ17-Ecate}, which admits a full functor to the full subcategory of $\fso_m(\C)$-modules with objects $(\C^m)^{\ot r}$ for all $r\in\N$.  

It is straightforward to follow the procedure of Section \ref{sect:construct} to develop an enhanced affine Brauer category by replacing usual Brauer diagrams by  the enhanced Brauer diagrams of \cite{LZ17-Ecate}. Such a category will enable us to resolve the problem discussed in Remark \ref{rem:super-Pf}  in the case of $\fso_m$, that is, the map $\id_M\ot \det: M\lra M\ot (\C^m)^{\ot m}$ with
$\C \det(1)=\wedge^m(C^m)$ can be counted for by the category. One can also build such a category in the general $\osp$ case once an ``enhanced Brauer category'' similar to that of  \cite{LZ17-Ecate} is developed for this case.

\appendix
\section{Basics of the Lie superalgebra $\osp(V; \omega)$}\label{sect:tensor-notation}

We retain the notation in Section \ref{sect:osp}. In particular, $V=\C^{m|2n}$ is equipped with the non-degenerate bilinear form $\omega$ which is even and supersymmetric, and we write $\fg=\mathfrak{osp}(V; \omega)$. 

We give a complete proof of \eqref{eq:t-image}. When $\sdim(V)\ne 0$, the formula \eqref{eq:t-image} can be proved
 using representation theoretical arguments as in Section \ref{sect:quartet}. However, the case $\sdim(V)=0$  is  more difficult, 
 and the most efficient way to prove it is by explicit calculations with convenient bases for $V$ and $\fg$.  

We also present a self-contained treatment of the quadratic Casimir of $\U(\fg)$, and obtain explicit formulae for central elements which essentially generate the centre. 

\subsection{Tensorial notation for $\osp(V; \omega)$}\label{sect:tensor-notation-1}
Let us take the standard basis
$\{e^a\mid 1\le a\le {m+2n}\}$ for $V$, which is homogeneous with $e^i$ being even for all $i\le m$, and $e^{m+j}$ being odd for all $j\ge 1$.  We use $[v]=0, 1$ to denote the $\Z_2$-degree of a homogeneous element $v\in V$, and write $[a]=[e^a]$ for all $a$ for simplicity of notation. 

The form $\omega$ can be  represented by an invertible matrix $g^{-1}=(g^{a b})$ with respect to the standard basis $\{e^a\mid a=1, 2, \dots, m+2n\}$ of $V$,  where $g^{a b}=\omega(e^a, e^b)$ for all $a, b$. We write $g=(g_{a b})$.  Note that $g$ has the following symmetry properties: $g_{a b}\ne 0$ only if $[a]=[b]$, and $g_{a b} = (-1)^{[a]} g_{b a} = (-1)^{[b]} g_{b a}$. The elements $e_b= \sum_c g_{b c} e_c$ satisfy 
\[
\omega(e_b, e^a) = \sum_c g_{b c} \omega(e^c, e^a)  = \sum_c g_{b c} g^{c a}  = \delta^a_b, 
\]
where $\delta^a_b$ is the usual Kronecker delta, which is equal to $1$ if $a=b$ and $0$ otherwise.  Thus $\{e_a\mid a=1, 2, \dots, m+2n\}$ is a basis of $V$ dual to the standard basis.  We have $e^a = \sum_b g^{a b} e_b$. 
Note in particular that the element $j^{-1}(\id_V)$ in \eqref{eq:cup-cap} can be expressed as $j^{-1}(\id_V)=\sum_{a=1}^{m+2n} e^a\ot e_a$.

Let $E^a_b\in\End_\C(V)$ (where $a, b=1, 2, \dots, m+2n$) be the matrix units defined relative to the standard basis, that is,  $E^a_b(e^c)=\delta^c_b e^a$. Let $E_{a b}=\sum_c g_{a c} E^c _b$, then $E_{a b}(e^c)=e_a\delta_b^c$. 
Introduce the following elements of $\End_\C(V)$, 
\beq\label{eq:J}
J_{a b} = E_{a b} - (-1)^{[a][b]} E_{b a}, \quad a, b=1, 2, \dots, m+2n, 
\eeq
where clearly $J_{a b}=-(-1)^{[a][b]}J_{b a}$. We also let
\[
\begin{aligned}
J^a_b :=\sum_{a'} g^{a a'} J_{a' b}	&=E^a_b -  (-1)^{[a][b]}  \sum_{b', a'} g^{a a'} g_{b b'} E^{b'}_{a'}.
\end{aligned}
\]
The elements $J_{a b}$ satisfy the following commutation relations \cite{JG, Z08}
\beq
\begin{aligned}
{} [J_{a b}, J_{c d}] &= g_{c b} J_{a d} + (-1)^{[a]([b]+[c])}g_{d a}  J_{b c} \\
&- (-1)^{[c][d]} g_{d b} J_{a c} - (-1)^{[a][b]} g_{c a} J_{b d},  
\end{aligned}
\eeq
where $[X, Y]=XY - (-1)^{[X][Y]} Y X$ is the $\Z_2$-graded commutator of $X, Y$.  
They span the orthosymplectic Lie superalgebra $\fg$. Clearly $J_{i j}$, $J_{i, m+k}$ and  $J_{m+k, m+\ell}$ with $1\le i< j\le m$, $1\le k\le \ell \le 2n$ are linearly independent, and hence $\dim(\fg)=\frac{1}{2}m(m-1) + \frac{1}{2}2n(2n+1) + 2m n$. 

Let us verify that the bilinear form $\omega$ is indeed invariant with respect to this Lie superalgebra. 
Note that 
\beq\label{eq:act}
\begin{aligned}
J_{a b}(e^c)&=e_a \delta^c_b - (-1)^{[a][b]} e_b \delta_a^c.
\end{aligned}
\eeq
Using this relation we obtain 
\[
\begin{aligned}
\omega(J_{a b}(e^c), e^d) &= \delta_a^d \delta^c_b - (-1)^{[a][b]} \delta_b^d\delta_a^c, \\
\omega(e^c, J_{a b}(e^d)) 
&= -(-1)^{([a]+[b])[c]} \left(\delta^d_a \delta^c_b- (-1)^{[a][b]} \delta^d_b \delta^c_a\right),
\end{aligned}
\]
which immediately lead to 
\[
 \omega(J_{a b}(e^c), e^d) + (-1)^{([a]+[b])[c]} \omega(e^c, J_{a b}(e^d))=0, \quad \forall a, b, c, d.  
\]
This proves the invariance of $\omega$ with respect to the
Lie superalgebra spanned by the elements $J_{a b}$. 

\subsection{Proof of Lemma \ref{lem:key}}\label{sect:t-image}
Denote by $X_{a b}$ (resp. $X^a_b$) the image of $J_{a b}$ (resp. $J^a_b$) in $\U(\fg)$ under the canonical embedding $\fg\hookrightarrow \U(\fg)$. Now we use ``tensor contraction'' to write down 
the quadratic Casimir of $\U(\fg)$,   
\beq\label{eq:Casimir-formula}
C=\frac{1}{2}\sum_{a, b=1} ^{m+2n}(-1)^{[b]} X^a_b X^b_a,  
\eeq
where the normalisation factor $\frac{1}{2}$ is chosen in order for Theorem \ref{thm:H-formula} to be valid. 

\begin{remark}
The Casimir operator \eqref{eq:Casimir-formula} in the case of $\mathfrak{sp}_2=\fsl_2$ is $-2$ times the standard Casimir operator of $\fsl_2$, see Example \ref{eg:sp2}. 
\end{remark}

Now the tempered Casimir element $t=\frac{1}{2}\left(\Delta(C)- C\ot 1 - 1 \ot C\right)$ is equal to 
\beq\label{eq:t-formula}
t=\frac{1}{2}\sum_{a, b=1} ^{m+2n} X^a_b \ot (-1)^{[b]} X^b_a. 
\eeq

\begin{lem} We have the following formula.
\beq
(\id_\U\ot\mu)(t)&=&\sum_{a, b} X^a_b\ot (-1)^{[b]} E^b_a. \label{eq:t-HH}
\eeq
\end{lem}
\begin{proof}
Using  $\mu(X^a_b)=J^a_b$ and $X^b_a=\sum_c g^{b c} X_{c b}$, we can express $2(\id_\U\ot \mu)(t)$ as 
\[
\begin{aligned}
2(\id_\U\ot \mu)(t)& = \sum_{a, b} X^b_a\ot (-1)^{[a]}E^a_b \\
&-  \sum_{a, b, c} g^{b c} X_{c a}\ot(-1)^{[a][b]}  \sum_{b', a'} g^{a a'} g_{b b'} E^{b'}_{a'}. 
\end{aligned}
\]
Note that the first term is equal to the right hand side of \eqref{eq:t-HH}, thus we want to show that the second term is the same. Using symmetry properties of $g$, we can re-write the second term as 
\[
- \sum_{a, b, c} X_{c a}\ot (-1)^{[a][c] +[a]+[c]}  g^{a b}  E^{c}_{b}
=
- \sum_{a, b, c} g^{ba}  (-1)^{[a][c]} X_{c a}\ot (-1)^{[c]}   E^{c}_{b}.
\]
Using $X_{c a}=-(-1)^{[a][c]}X_{a c}$, we can re-write the right hand side of the above identity as 
\[
\begin{aligned}
\sum_{a, b, c} g^{b a} X_{a c}\ot (-1)^{[c]}   E^{c}_{b}
= \sum_{b, c}X^b_{c}\ot (-1)^{[c]}   E^{c}_{b}, 
\end{aligned}
\]
which is equal to the right hand side of \eqref{eq:t-HH}. This completes the proof. 
\end{proof}

The following result is the Key Lemma (i.e., Lemma \ref{lem:key}), which plays a crucial role in the main body of this paper. 
\begin{thm}\label{thm:H-formula} The action of $t$ on $V\ot V$ is given by
\beq
(\mu\ot\mu)(t)=\frac{1}{2}\sum_{a, b} J^a_b \ot (-1)^{[b]}J^b_a = \tau - e. 
\eeq
\end{thm}
\begin{proof}
By expressing $J$'s in terms of $E$'s using \eqref{eq:J}, we obtain from \eqref{eq:t-HH}
\[
(\mu\ot\mu)(t)= \frac{1}{2}\sum_{a, b} J^a_b \ot (-1)^{[b]}J^b_a=\sum_{a, b} J^a_b \ot (-1)^{[b]}E^b_a  =T_1 - T_2, 
\]
where
\[
\begin{aligned}
T_1&= \sum (-1)^{[b]} E^a_b\ot  E^b_a, \\
T_2&= \sum (-1)^{[b]} (-1)^{[a]([a]+[b])}   g_{b b''} E^{b''}_{a''} g^{a'' a} \ot E^b_a.
\end{aligned}
\]
Clearly $T_1=\tau$. For any $e^c\ot e^d$, a long handed calculation leads to  
\[
\begin{aligned}
T_2(e^c\ot e^d)
					 &= g^{c d} \check{c}(1)= \check{c}\circ \hat{c}(e^c\ot e^d), 
\end{aligned}
\]
that is,  $T_2=e$.  This completes the proof. 
\end{proof}

Let us compute the eigenvalue of $C$ in $V$.
\begin{lem}
The eigenvalue of $C$ in $V$ is equal to
\beq\label{eq:eigen-C}
\chi_V(C) = \sdim(V)-1.
\eeq
\end{lem}
\begin{proof}
We have  
$
\chi_V(C)\id_V=\frac{1}{2}\sum_{a, b} J^a_b (-1)^{[b]}J^b_a=\sum_{a, b} J^a_b (-1)^{[b]}E^b_a.
$
Expressing $J$'s in terms of $E$'s, we obtain
$\chi_V(C)\id_V
=Q_1 - Q_2$, with
\[
\begin{aligned}
Q_1&= \sum (-1)^{[b]} E^a_b E^b_a, \\  
Q_2&= \sum (-1)^{[b]} (-1)^{[a]([a]+[b])}   g_{b b''} E^{b''}_{a''} g^{a'' a} E^b_a. 
\end{aligned}
\]
By using properties of $g$ and carefully keeping track of the sign factors, one can show that 
$
Q_1 = \sdim(V) \id_V$  and $Q_2 = \id_V$. This immediately leads to 
\eqref{eq:eigen-C}. 
\end{proof}

\subsection{Formulae for central elements of $\U(\osp(V; \omega))$}
Let $T[1]_b^a=X_b^a$, and define $T[k]_b^a$ for $k>1$ recursively by
\[
T[k]_b^a=\sum_c  (-1)^{[c]} X^c_b T[k-1]_c^a.
\]
Then we easily obtain the following formula by using \eqref{eq:t-HH}. 
\beq
(\id_\U\ot\mu)(t^k)&=& \sum_{a, b} T[k]^{a}_{b}\ot (-1)^{[b]} E^b_a.  \label{eq:t-power}
\eeq

Recall the elements $Z_\ell\in \Hom_{\AB(\delta)}(0, 0)$ for $\ell\ge 1$ defined by \eqref{eq:Z-generators}. It was shown that $Z_1=0$. Applying $\CF_\U$ to them and using \eqref{eq:t-power}, we obtain 
$
\CF_\U(Z_\ell)= \sum_c T[\ell]^c_c,  
$
which can be expressed as 
\beq\label{eq:FZ}
\CF_\U(Z_\ell)= \sum_{a_1,\dots, a_{\ell-1},  b=1}^{m+2n} (-1)^{\sum_{i=1}^{\ell-1} [a_i]} X^{a_1}_b  X^{a_2}_{a_1} X^{a_3}_{a_2}\dots X^b_{a_{\ell-1}}. 
\eeq
It is shown in Theorem \ref{eq:key-centre} that these elements are central, and that in the case $\sdim(V)$ is odd, they generate the centre $Z(\U(\fg))$ of $\U(\fg)$.  


\subsection*{Acknowledgements}
We thank Tony Bracken, Mark Gould and Yang Zhang for Zoom discussions and email correspondence, and Catherina Stroppel for discussions.

\end{document}